\documentclass[11pt,a4paper,reqno]{amsart}
\usepackage[margin=2cm]{geometry}
\usepackage{hyperref,amssymb,amsfonts,amsthm,amsrefs,mathtools,verbatim}
\usepackage[shortlabels]{enumitem}
\usepackage{tikz}
\usetikzlibrary{arrows.meta}
\usepackage{microtype}

\newtheorem{Theorem}{Theorem}[section]
\newtheorem{Proposition}[Theorem]{Proposition}
\newtheorem{Lemma}[Theorem]{Lemma}

\newtheorem{ClaimADSimpRSpo}{Claim}
\newtheorem{ClaimSADSimpRSpocd2}{Claim}
\newtheorem*{TheoremRSg}{Rival--Sands theorem for graphs}
\newtheorem*{TheoremRSpo}{Rival--Sands theorem for partial orders}

\theoremstyle{definition}
\newtheorem{Definition}[Theorem]{Definition}

\newtheorem{Question}[Theorem]{Question}
\newtheorem{Construction}[Theorem]{Construction}

\newcommand{\rca}{\mathsf{RCA}_0}
\newcommand{\wkl}{\mathsf{WKL}_0}
\newcommand{\aca}{\mathsf{ACA}_0}
\newcommand{\lpp}{\mathsf{LPP}}
\newcommand{\pica}{\Pi^1_1\mbox{-}\mathsf{CA}_0}
\newcommand{\rt}{\mathsf{RT}}
\newcommand{\srt}{\mathsf{SRT}}
\newcommand{\ads}{\mathsf{ADS}}

\newcommand{\sads}{\mathsf{SADS}}

\newcommand{\cac}{\mathsf{CAC}}

\newcommand{\mmlc}{\mathsf{MMLC}}

\newcommand{\ssrt}[3]{{#1}\mbox{-}\mathrm{stable~}\srt^{#2}_{#3}}

\newcommand{\rspo}{\mathsf{RSpo}}
\newcommand{\cofrspo}{(0,\cof)\mbox{-}\rspo}

\newcommand{\rspocd}{\mathsf{RSpo}^{\mathsf{CD}}}
\newcommand{\cofrspocd}{(0,\cof)\mbox{-}\rspocd}

\newcommand{\cc}{\mathsf{CC}}
\newcommand{\ca}{\mathsf{CA}}

\newcommand{\isig}[1]{\mathsf{I}\Sigma^0_{#1}}
\newcommand{\ipi}[1]{\mathsf{I}\Pi^0_{#1}}

\newcommand{\bsig}[1]{\mathsf{B}\Sigma^0_{#1}}
\newcommand{\bpi}[1]{\mathsf{B}\Pi^0_{#1}}

\DeclareMathOperator{\ran}{\mathrm{ran}}

\DeclareMathOperator{\findlad}{\mathrm{FindLadder}}

\newcommand{\andd}{\wedge}
\newcommand{\orr}{\vee}
\newcommand{\la}{\langle}
\newcommand{\ra}{\rangle}
\newcommand{\da}{{\downarrow}}
\newcommand{\ua}{{\uparrow}}
\newcommand{\imp}{\rightarrow}
\newcommand{\Imp}{\Rightarrow}
\newcommand{\biimp}{\leftrightarrow}
\newcommand{\Biimp}{\Leftrightarrow}
\newcommand{\Nb}{\mathbb{N}}

\newcommand{\smf}{\smallfrown}
\newcommand{\rst}{{\restriction}}

\newcommand{\ltKB}{<_\mathrm{KB}}
\newcommand{\gtKB}{>_\mathrm{KB}}

\newcommand{\wh}[1]{\widehat{#1}}

\newcommand{\comp}{\lessgtr}
\newcommand{\cof}{\mathrm{cof}}

\newcommand{\pwb}{\leq_{\forall\exists}}
\newcommand{\npwb}{\nleq_{\forall\exists}}

\AtBeginDocument{%
   \def\MR#1{}
}

\title{(Extra)ordinary equivalences with the ascending/descending sequence principle}
\author{Marta Fiori-Carones}
\address{Sobolev Institute of Mathematics, pr. Akad. Koptyuga 4, Novosibirsk, 630090 Russia}
\email{marta.fioricarones@outlook.it}
\urladdr{https://martafioricarones.github.io}

\author{Alberto Marcone}
\address{Dipartimento di scienze matematiche, informatiche e fisiche, Universit\`a di Udine, Via delle Scienze 208, 33100 Udine, Italy}
\email{alberto.marcone@uniud.it}
\urladdr{http://users.dimi.uniud.it/~alberto.marcone/}

\author{Paul Shafer}
\address{School of Mathematics\\
University of Leeds\\
Leeds LS2 9JT\\
United Kingdom}
\email{p.e.shafer@leeds.ac.uk}
\urladdr{http://www1.maths.leeds.ac.uk/~matpsh/}

\author{Giovanni Sold\`{a}}
\address{Department of Mathematics:  Analysis, Logic, and Discrete Mathematics\\
Ghent University\\
Krijgslaan 281 S8\\
9000 Ghent\\
Belgium}
\email{giovanni.a.solda@gmail.com}
\urladdr{https://giovannisolda.github.io/}


\date{\today}

\begin{document}

\begin{abstract}
We analyze the axiomatic strength of the following theorem due to Rival and Sands~\cite{RivalSands} in the style of reverse mathematics.  \emph{Every infinite partial order $P$ of finite width contains an infinite chain $C$ such that every element of $P$ is either comparable with no element of $C$ or with infinitely many elements of $C$}.  Our main results are the following.  The Rival--Sands theorem for infinite partial orders of arbitrary finite width is equivalent to $\isig{2} + \ads$ over $\rca$.  For each fixed $k \geq 3$, the Rival--Sands theorem for infinite partial orders of width $\leq\! k$ is equivalent to $\ads$ over $\rca$.  The Rival--Sands theorem for infinite partial orders that are decomposable into the union of two chains is equivalent to $\sads$ over $\rca$.  Here $\rca$ denotes the recursive comprehension axiomatic system, $\isig{2}$ denotes the $\Sigma^0_2$ induction scheme, $\ads$ denotes the ascending/descending sequence principle, and $\sads$ denotes the stable ascending/descending sequence principle.  To our knowledge, these versions of the Rival--Sands theorem for partial orders are the first examples of theorems from the general mathematics literature whose strength is exactly characterized by $\isig{2} + \ads$, by $\ads$, and by $\sads$.  Furthermore, we give a new purely combinatorial result by extending the Rival--Sands theorem to infinite partial orders that do not have infinite antichains, and we show that this extension is equivalent to arithmetical comprehension over $\rca$.
\end{abstract}

\maketitle

\section{Introduction}\label{sec-intro}

One of the major initiatives in reverse mathematics is the logical analysis of theorems of countable combinatorics, with special attention to Ramsey's theorem for pairs and its consequences~\cites{
CholakJockuschSlaman,
ChongSlamanYang,
ChongSlamanYangConvervation,
ChongSlamanYangInductiveStrength,
HirschfeldtShore,
LiuRT22vsWKL,
LiuRT22vsWWKL,
PateyYokoyama,
SeetapunSlaman,
SlamanYokoyama,
LermanSolomonTowsner}.  We continue this tradition by analyzing the second of two theorems from \emph{On the adjacency of vertices to the vertices of an infinite subgraph} by Rival and Sands~\cite{RivalSands}.  To our knowledge, this analysis yields the first examples of theorems from the general mathematics literature whose strength is exactly characterized by the ascending/descending sequence principle.

Both of the Rival--Sands theorems are inspired by Ramsey's theorem for pairs and two colors.  The first theorem is a hybrid \emph{inside/outside} version of Ramsey's theorem for pairs.  Thinking in terms of graphs, Ramsey's theorem for pairs produces an infinite set of vertices $H$ that is either a clique or an independent set in a given countable graph.  However, Ramsey's theorem provides no information concerning the relationship between the vertices \emph{inside} $H$ and the vertices \emph{outside} $H$.  The \emph{Rival--Sands theorem for graphs} balances this situation by producing an infinite set of vertices $H$ that is not necessarily a clique or an independent set, but for which there is information concerning the relationship between the vertices inside $H$ and the vertices outside $H$.

\begin{TheoremRSg}[\cite{RivalSands}]
Every infinite graph $G$ contains an infinite subset $H$ such that every vertex of $G$ is adjacent to precisely none, one, or infinitely many vertices of $H$.
\end{TheoremRSg}

Rival and Sands note that the three options ``none,'' ``one,'' and ``infinitely many'' in their theorem are all necessary.  They ask for a class of graphs for which the ``one'' option may be removed, and they show that this is possible for the class of comparability graphs of infinite partial orders of finite width (i.e., infinite partial orders for which there is a fixed finite upper bound on the size of the antichains).  We call the resulting theorem the \emph{Rival--Sands theorem for partial orders}.

\begin{TheoremRSpo}[\cite{RivalSands}]
Every infinite partial order $P$ of finite width contains an infinite chain $C$ such that every element of $P$ is either comparable with no element of $C$ or with infinitely many elements of $C$.
\end{TheoremRSpo}

Furthermore, Rival and Sands show that in the case of countable partial orders, the ``infinitely many'' option may be strengthened to ``cofinitely many.''

The Rival--Sands theorems do not appear to be immediate consequences of any version of Ramsey's theorem, and neither Rival--Sands theorem appears to be an immediate consequence of the other.  Rival and Sands give direct proofs of both theorems that do not invoke any other Ramsey-theoretic statement.  Interestingly, their proof of the Rival--Sands theorem for partial orders makes essential use of $\pica$ by iterating a maximal chain principle that is equivalent to $\pica$ over $\rca$.  We discuss this in detail in Section~\ref{sec-MaxChains}.

In~\cite{Fiori-CaronesShaferSolda}, Fiori-Carones, Shafer, and Sold\`{a} analyze the axiomatic and computational strength of the Rival--Sands theorem for graphs in the style of reverse mathematics and Weihrauch analysis.  In reverse mathematics, the Rival--Sands theorem for graphs is equivalent to $\aca$ over $\rca$.  Thus the Rival--Sands theorem for graphs is indirectly equivalent to Ramsey's theorem for triples (which is also equivalent to $\aca$, see~\cite{SimpsonSOSOA}), but it does not follow from Ramsey's theorem for pairs (which is strictly weaker than $\aca$~\cite{SeetapunSlaman}).  However, the authors of~\cite{Fiori-CaronesShaferSolda} together with Hirst and Lempp show that a weakened inside-only version of the Rival--Sands theorem for graphs is indeed equivalent to Ramsey's theorem for pairs and two colors.  In terms of the Weihrauch degrees, the main result of~\cite{Fiori-CaronesShaferSolda} is that the Rival--Sands theorem for graphs is strongly Weihrauch-equivalent to the double-jump of weak K\"{o}nig's lemma.  To the authors' knowledge, the Rival--Sands theorem for graphs is the first theorem from the general mathematics literature exhibiting exactly this strength.  Furthermore, combining the aforementioned equivalence with a result of Brattka and Rakotoniaina~\cite{BrattkaRakotoniaina} yields that the Rival--Sands theorem for graphs is Weihrauch-equivalent to the parallelization of Ramsey's theorem for pairs and two colors.  Thus the uniform computational strength of the Rival--Sands theorem for graphs is exactly characterized by the ability to simultaneously solve countably many instances of Ramsey's theorem for pairs in parallel.

In this work, we characterize the axiomatic strength of the Rival--Sands theorem for partial orders in terms of the \emph{ascending/descending sequence principle}, which states that an infinite linear order contains either an infinite ascending sequence or an infinite descending sequence.  Our primary focus is the Rival--Sands theorem for partial orders as stated above, but we also consider the version with ``cofinitely many'' in place of ``infinitely many.''  The main results are the following, which are summarized in Theorem~\ref{thm-RSpoEquivs}.

\begin{itemize}
\item The Rival--Sands theorem for infinite partial orders of arbitrary finite width is equivalent to the ascending/descending sequence principle plus the $\Sigma^0_2$ induction scheme over $\rca$.

\medskip

\item For each fixed standard $k \geq 3$, the Rival--Sands theorem for infinite partial orders of width $\leq\! k$ is equivalent to the ascending/descending sequence principle over $\rca$.

\medskip

\item The Rival--Sands theorem for infinite partial orders that are decomposable into the union of two chains is equivalent to the stable ascending/descending sequence principle over $\rca$.

\medskip

\item The Rival--Sands theorem with ``cofinitely many'' in place of ``infinitely many'' for infinite partial orders of width $\leq\! 2$ is equivalent to the ascending/descending sequence principle over $\rca$ plus the $\Sigma^0_2$ induction scheme.
\end{itemize}

Furthermore, in Theorems~\ref{thm-CofExtendedRSpo} and~\ref{thm-ExtendedRSpo}, we give a new purely combinatorial result by extending the Rival--Sands theorem to all countably infinite partial orders that do not have infinite antichains.  This is non-trivial, as a partial order may have arbitrarily large finite antichains, and therefore \emph{not} have finite width, yet still have no infinite antichain.  In Theorem~\ref{thm-ExtendedRSpoEquiv}, we also show that the extension of the Rival--Sands theorem to countably infinite partial orders without infinite antichains is equivalent to $\aca$ over $\rca$.

Computational aspects of linear and partial orders have long been studied.  In reverse mathematics, the ascending/descending sequence principle is a weak consequence of Ramsey's theorem for pairs that was first isolated and studied by Hirschfeldt and Shore~\cite{HirschfeldtShore}.  There are a few classical statements that are readily equivalent to this principle over $\rca$, such as the statement ``Every countable sequence of real numbers contains a monotone subsequence'' (see~\cite{KreuzerThesis}*{Remark~6.8}), but thus far the principle's primary use has been as an important technical benchmark.  To our knowledge, Theorem~\ref{thm-RSpoEquivs} provides the first examples from the modern mathematical literature of theorems that are equivalent to the ascending/descending sequence principle, the ascending/descending sequence principle plus $\Sigma^0_2$ induction, and to the stable ascending/descending sequence principle.  This explains our title.  It is extraordinary to find theorems from the ordinary literature that are equivalent to the ascending/descending sequence principle.

It follows from our analysis that the Rival--Sands theorem for partial orders without infinite antichains is equivalent to the Rival--Sands theorem for graphs, whereas the Rival--Sands theorem for partial orders of finite width is strictly weaker.  The relationship between the Rival--Sands theorem for partial orders and Ramsey's theorem for pairs is more curious.  Ramsey's theorem for pairs and two colors suffices to prove the Rival--Sands theorem for partial orders of width $k$ for any fixed $k$, but not for partial orders of arbitrary finite width.  However, Ramsey's theorem for pairs and arbitrarily many colors does suffice to prove the Rival--Sands theorem for partial orders of arbitrary finite width.  This is a matter of induction.  Ramsey's theorem for pairs and two colors does not prove the $\Sigma^0_2$ induction scheme~\cite{ChongSlamanYangInductiveStrength}, but Ramsey's theorem for pairs and arbitrarily many colors does~\cite{HirstThesis}.  All together, the Rival--Sands theorem for partial orders of arbitrary finite width is strictly weaker than Ramsey's theorem for pairs and arbitrarily many colors; the Rival--Sands theorem for partial orders of arbitrary finite width is not provable from Ramsey's theorem for pairs and two colors; and the Rival--Sands theorem for partial orders of a fixed finite width $k$ is strictly weaker than Ramsey's theorem for pairs and two colors.

This article is organized as follows.  Section~\ref{sec-background} gives an overview of the relevant reverse mathematics background.  Section~\ref{sec-Dilworth} formalizes several versions of the Rival--Sands theorem for partial orders and discusses principles concerning chains in partial orders, most notably Kierstead's effective analog of Dilworth's theorem~\cite{KiersteadDilworth}.  Section~\ref{sec-FirstProofs} presents our first proofs of the Rival--Sands theorem for partial orders.  These proofs are not axiomatically optimal, but they are easy to understand, and they introduce ideas that we later effectivize in order to give proofs in weaker systems.  Section~\ref{sec-forward} provides the forward directions of the equivalences mentioned above.  Section~\ref{sec-reverse} provides the reversals.  Section~\ref{sec-ExtendedRSpo} provides our extension of the Rival--Sands theorem to infinite partial orders without infinite antichains as well as a proof that this extension is equivalent to $\aca$.  All the proofs of the Rival--Sands theorem for partial orders and its variants given in Sections~\ref{sec-FirstProofs}--\ref{sec-ExtendedRSpo} are new.  In Section~\ref{sec-MaxChains}, we discuss the original proof by Rival and Sands, and we analyze several principles asserting that partial orders contain various sorts of maximal chains.

\section{Reverse mathematics background}\label{sec-background}

We give a brief review of $\rca$, $\wkl$, $\aca$, $\pica$, Ramsey's theorem for pairs, its combinatorial consequences, and the first-order schemes.  For further details, we refer the reader to Simpson's standard reference~\cite{SimpsonSOSOA} and to Hirschfeldt's monograph~\cite{HirschfeldtBook}.

We work in the two-sorted language of second-order arithmetic $0$, $1$, $<$, $+$, $\times$, $\in$, where variables $x$, $y$, $z$, etc.\ typically range over the first sort, thought of as \emph{natural numbers}, and variables $X$, $Y$, $Z$, etc.\ typically range over the second sort, thought of as \emph{sets of natural numbers}.  As usual, the symbol $\Nb$ denotes the first-order part of whatever structure is under consideration.

The axioms of the base system $\rca$ (for \emph{recursive comprehension axiom}) are as follows.
\begin{itemize}
\item A first-order sentence expressing that $(\Nb; 0, 1, <, +, \times)$ forms a discretely ordered commutative semi-ring with identity.

\medskip

\item The \emph{$\Sigma^0_1$ induction scheme} (denoted $\isig{1}$), which consists of the universal closures (by both first- and second-order quantifiers) of all formulas of the form
\begin{align*}
\bigl(\varphi(0) \;\andd\; \forall n \, (\varphi(n) \;\imp\; \varphi(n+1))\bigr) \;\imp\; \forall n \, \varphi(n),
\end{align*}
where $\varphi$ is $\Sigma^0_1$.

\medskip

\item The \emph{$\Delta^0_1$ comprehension scheme}, which consists of the universal closures (by both first- and second-order quantifiers) of all formulas of the form
\begin{align*}
\forall n \, \bigl(\varphi(n) \;\biimp\; \psi(n)\bigr) \;\imp\; \exists X \, \forall n \, \bigl(n \in X \;\biimp\; \varphi(n)\bigr),
\end{align*}
where $\varphi$ is $\Sigma^0_1$, $\psi$ is $\Pi^0_1$, and $X$ is not free in $\varphi$.
\end{itemize}
The `$0$' in `$\rca$' refers to the restriction of the induction scheme to $\Sigma^0_1$ formulas.  $\rca$ suffices to implement the usual bijective encodings of pairs of numbers, finite sequences of numbers, finite sets of numbers, and so on.  See~\cite{SimpsonSOSOA}*{Section~II.2} for details on how this is done.  For example, we may encode a function $f \colon \Nb \imp \Nb$ by its graph $\{\la m, n \ra : f(m) = n\}$.  We may also encode the set $\Nb^{<\Nb}$ of all finite sequences of natural numbers and the set $2^{<\Nb}$ of all finite binary sequences.  For $\sigma, \tau \in \Nb^{<\Nb}$ and $f \colon \Nb \imp \Nb$, let $|\sigma|$ denote the length of $\sigma$; let $\tau \preceq \sigma$ denote that $\tau$ is an initial segment of $\sigma$: $|\tau| \leq |\sigma| \;\andd\; \forall n < |\tau| \; (\tau(n) = \sigma(n))$; and let $\tau \preceq f$ denote that $\tau$ is an initial segment of $f$:  $\forall n < |\tau| \; (\tau(n) = f(n))$.  For $\sigma, \tau \in \Nb^{<\Nb}$, let $\sigma^\smf \tau$ denote the concatenation of $\sigma$ and $\tau$.  When $\tau = \la n \ra$ is a sequence of length $1$, we usually write $\sigma^\smf n$ instead of $\sigma^\smf \la n \ra$.  For $f \colon \Nb \imp \Nb$, $\sigma \in \Nb^{<\Nb}$, and $n \in \Nb$, let $f \rst n = \la f(0), f(1), \dots, f(n-1) \ra$ denote the initial segment of $f$ of length $n$; and if $n \leq |\sigma|$, let $\sigma \rst n = \la \sigma(0), \dots, \sigma(n-1) \ra$ denote the initial segment of $\sigma$ of length $n$.

$\rca$ also suffices to develop the basic theory of oracle Turing machines (see for example~\cite{SimpsonSOSOA}*{Section~VII.1}).  We view such machines as defining Turing functionals, and we write $\Phi(A)$ for the result of applying the functional $\Phi$ to the set $A$.  We also write $\Phi(A)(n)$ for the value of $\Phi(A)$ on input $n$, if it is defined.  We may relativize a Turing functional $\Phi$ to a set $X$, write $\Phi^X$ to denote the relativized functional, and write $\Phi^X(A)$ for $\Phi(X \oplus A)$, where $X \oplus A = \{2n : n \in X\} \cup \{2n+1 : n \in A\}$ as usual.  The $\Phi^X(A)$ notation is intended to convey that $X$ is fixed but $A$ may vary.  We may also iterate a sequence $\Phi_0, \dots, \Phi_{k-1}$ of Turing functionals.  For a set $A$, we say that the iteration $\Phi_{k-1}(\Phi_{k-2}(\cdots\Phi_0(A)\cdots))$ is total if $\Phi_i(\Phi_{i-1}(\cdots\Phi_0(A)\cdots))(n)$ is defined for every $i < k$ and every $n$.  This may be expressed by asserting that for every $n$, there is a sequence $\la \sigma_0, \dots, \sigma_k \ra$ of elements of $\Nb^{<\Nb}$, each of length at least $n$, such that $\sigma_0 \preceq A$ and $\forall i < k\; \forall m < |\sigma_{i+1}|\; \bigl(\text{$\Phi_i(\sigma_i)(m)$ halts within $|\sigma_i|$ steps and $\sigma_{i+1}(m) = \Phi_i(\sigma_i)(m)$}\bigr)$.

Define a \emph{tree} to be a set $T \subseteq \Nb^{<\Nb}$ that is closed under initial segments:  $\forall \sigma \, \forall \tau \, ((\sigma \in T \;\andd\; \tau \preceq \sigma) \;\imp\; \tau \in T)$.  Say that an $f \colon \Nb \imp \Nb$ is a \emph{infinite path} through a tree $T$ if every initial segment of $f$ is in $T$:  $\forall n \, (f \rst n \in T)$.  A tree with no infinite path is called \emph{well-founded}, and a tree with an infinite path is called \emph{ill-founded}.  Finally, a tree $T \subseteq \Nb^{<\Nb}$ is called \emph{finitely branching} if for every $\sigma \in T$ there are only finitely many $n$ with $\sigma^\smf n \in T$.  A tree $T \subseteq 2^{<\Nb}$ is necessarily finitely branching.  All of these definitions can be made in $\rca$.  The axioms of $\wkl$ (for \emph{weak K\"{o}nig's lemma}) are those of $\rca$, plus the statement that every infinite tree $T \subseteq 2^{<\Nb}$ has an infinite path.

The axioms of $\aca$ (for \emph{arithmetical comprehension axiom}) are of those of $\rca$, plus the \emph{arithmetical comprehension scheme}, which consists of the universal closures of all formulas of the form
\begin{align*}
\exists X \, \forall n \, \bigl(n \in X \;\biimp\; \varphi(n)\bigr),
\end{align*}
where $\varphi$ is an arithmetical formula in which $X$ is not free.  To show that some statement $\varphi$ implies $\aca$ over $\rca$, a common strategy is to use $\rca + \varphi$ to show that the ranges of injections exist as sets and appeal to the following well-known lemma.

\begin{Lemma}[\cite{SimpsonSOSOA}*{Lemma~III.1.3}]\label{lem-ACAinjection}
The following are equivalent over $\rca$.
\begin{enumerate}[(1)]
\item $\aca$.

\medskip

\item If $f \colon \Nb \imp \Nb$ is an injection, then there is a set $X$ such that $\forall n \, (n \in X \;\biimp\; \exists s \, (f(s) = n))$.
\end{enumerate}
\end{Lemma}

Unlike its weak version, the full version of K\"{o}nig's lemma, which states that every infinite finitely branching tree $T \subseteq \Nb^{<\Nb}$ has an infinite path, is equivalent to $\aca$ over $\rca$ (see~\cite{SimpsonSOSOA}*{Theorem~III.7.2}).

The axioms of $\pica$ (for \emph{$\Pi^1_1$ comprehension axiom}) are those of $\rca$, plus the \emph{$\Pi^1_1$ comprehension scheme}, which consists of the universal closures of all formulas of the form
\begin{align*}
\exists X \, \forall n \, \bigl(n \in X \;\biimp\; \varphi(n)\bigr),
\end{align*}
where $\varphi$ is a $\Pi^1_1$ formula in which $X$ is not free.  To show that some statement $\varphi$ implies $\pica$ over $\rca$, a useful tool is to use $\rca + \varphi$ to show that every ill-founded tree has a leftmost path.  For functions $f,g \colon \Nb \imp \Nb$, say that \emph{$g$ is to the left of $f$} if $\exists n \, \bigl(g(n) < f(n) \;\andd\; \forall i < n \; (g(i) = f(i))\bigr)$.  The \emph{leftmost path principle} ($\lpp$) states that for every ill-founded tree $T \subseteq \Nb^{<\Nb}$, there is an infinite path $f$ through $T$ such that no infinite path through $T$ is to the left of $f$.

\begin{Theorem}[\cite{MarconeNW}*{Theorem~6.5}]\label{thm-LPP}
The following are equivalent over $\rca$.
\begin{enumerate}[(1)]
\item $\pica$.

\medskip

\item $\lpp$.
\end{enumerate}
\end{Theorem}

A huge amount of research in reverse mathematics is devoted to understanding the strength of Ramsey's theorem for pairs and its consequences.  For a set $X \subseteq \Nb$, let $[X]^2$ denote the set of two-element subsets of $X$, which may be encoded as $[X]^2 = \{\la x, y \ra: x, y \in X \;\andd\; x < y\}$.  A function $c \colon [\Nb]^2 \imp k$ is called a \emph{$k$-coloring of pairs}, and an infinite set $H \subseteq \Nb$ is called \emph{homogeneous} for a $k$-coloring of pairs $c$ if $c$ is constant on $[H]^2$.  Furthermore, a $k$-coloring of pairs $c$ is called \emph{stable} if $\lim_s c(n,s)$ exists for every $n$.  \emph{Ramsey's theorem for pairs and $k$ colors} ($\rt^2_k$) states that for every $k$-coloring of pairs $c$, there is a set that is homogeneous for $c$.  \emph{Stable Ramsey's theorem for pairs and $k$ colors} ($\srt^2_k$) is the restriction of $\rt^2_k$ to stable $k$-colorings of pairs $c$.  $\rt^2_{<\infty}$ abbreviates $\forall k \, \rt^2_k$, and $\srt^2_{<\infty}$ abbreviates $\forall k \, \srt^2_k$.  Likewise, a function $c \colon \Nb \imp k$ is called a \emph{$k$-coloring of singletons}, and an infinite $H \subseteq \Nb$ is called \emph{homogeneous} for a $k$-coloring of singletons $c$ if $c$ is constant on $H$.  \emph{Ramsey's theorem for singletons and $k$ colors} ($\rt^1_k$) states that for every $k$-coloring of singletons $c$, there is a set that is homogeneous for $c$.  $\rt^1_{<\infty}$ abbreviates $\forall k \, \rt^1_k$, which we think of as expressing the \emph{infinite pigeonhole principle}.  $\rca$ proves $\rt^1_k$ for each fixed standard $k$, but it does not prove $\rt^1_{<\infty}$.

Of the many combinatorial consequences of $\rt^2_2$, we are primarily concerned with the \emph{ascending/descending sequence principle} ($\ads$), stating that every countably infinite linear order has either an infinite ascending sequence or an infinite descending sequence, as well as its stable version $\sads$.  To be precise, let $(L, <_L)$ be a linear order.  A set $S \subseteq L$ is an \emph{ascending sequence} if $\forall x, y \in S \; (x < y \imp x <_L y)$, and it is a \emph{descending sequence} if $\forall x, y \in S \; (x < y \imp y <_L x)$.  The principle $\ads$ then states that for every infinite linear order $(L, <_L)$, there is an infinite $S \subseteq L$ that is either an ascending sequence or a descending sequence.  Often it is more convenient to phrase $\ads$ by stating that for every infinite linear order $(L, <_L)$, there is an infinite sequence $\la x_n : n \in \Nb \ra$ of elements of $L$ such that either $x_0 <_L x_1 <_L x_2 <_L \cdots$ or $x_0 >_L x_1 >_L x_2 >_L \cdots$.  $\rca$ proves that the two phrasings of $\ads$ are equivalent because it proves that any sequence $x_0, x_1, x_2, \dots$ of distinct elements of $\Nb$ can be thinned to an $<$-increasing subsequence $x_{i_0} < x_{i_1} < x_{i_2} < \cdots$ where the set $\{x_{i_n} : n \in \Nb\}$ exists.  Call a linear order $(L, <_L)$ \emph{stable} if every element either has only finitely many predecessors or has only finitely many successors.  $\sads$ is the restriction of $\ads$ to infinite stable linear orders.  Closely related to $\ads$ is the \emph{chain/antichain principle} ($\cac$), which states that every infinite partial order has either an infinite chain or an infinite antichain.  Figure~\ref{fig-summary} summarizes the relationships among several of the systems and principles mentioned thus far.

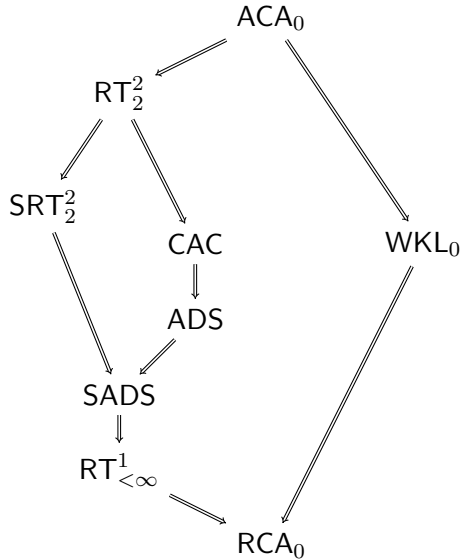
\begin{figure}[ht]
\begin{tikzpicture}
\node (ACA) at (0,10) {$\aca$};
\node (RT22) at (-2,9) {$\rt^2_2$};
\node (SRT22) at (-3,7.5) {$\srt^2_2$};
\node (CAC) at (-1,7) {$\cac$};
\node (ADS) at (-1,6) {$\ads$};
\node (SADS) at (-2,5) {$\sads$};
\node (RT1) at (-2,4) {$\rt^1_{<\infty}$};
\node (WKL) at (2,7) {$\wkl$};
\node (RCA) at (0,3) {$\rca$};

\draw[-{Implies},double] (ACA) --  (RT22);
\draw[-{Implies},double] (ACA) --  (WKL);
\draw[-{Implies},double] (WKL) --  (RCA);
\draw[-{Implies},double] (RT22) -- (SRT22);
\draw[-{Implies},double] (RT22) --  (CAC);
\draw[-{Implies},double] (SRT22) --  (SADS);
\draw[-{Implies},double] (CAC) --  (ADS);
\draw[-{Implies},double] (ADS) --  (SADS);
\draw[-{Implies},double] (SADS) --  (RT1);
\draw[-{Implies},double] (RT1) --  (RCA);
\end{tikzpicture}
\caption{Selected principles and systems and their implications and non-implications over $\rca$.  An arrow indicates that the source principle/system implies the target principle/system over $\rca$.  No further arrows may be added, except those that may be inferred via transitivity.  No arrows reverse.  Proofs of these implications and separations may be found in~\cites{ChongLemppYang, ChongSlamanYang, HirschfeldtShore, HirstThesis, LermanSolomonTowsner, LiuRT22vsWKL, PateyImmunity, SeetapunSlaman, SimpsonSOSOA}.}
\label{fig-summary}
\end{figure}

Finally, we recall the induction schemes and their cousins, to which we collectively refer as the \emph{first-order schemes}.

\begin{itemize}
\item The \emph{induction axiom for $\varphi$} is the universal closure of the formula
\begin{align*}
\bigl(\varphi(0) \;\andd\; \forall n \, (\varphi(n) \;\imp\; \varphi(n+1))\bigr) \;\imp\; \forall n \, \varphi(n).
\end{align*}

\medskip

\item The \emph{least element principle for $\varphi$} is the universal closure of the formula
\begin{align*}
\exists n \, \varphi(n) \;\imp\; \exists n \, \bigl(\varphi(n) \;\andd\; \forall m < n \; \neg\varphi(m)\bigr).
\end{align*}

\medskip

\item The \emph{bounded comprehension axiom for $\varphi$} is the universal closure of the formula
\begin{align*}
\forall b \, \exists X \, \forall n \, \bigl(n \in X \;\biimp\; (n < b \;\andd\; \varphi(n))\bigr),
\end{align*}
where $X$ is not free in $\varphi$.

\medskip

\item The \emph{bounding (or collection) axiom for $\varphi$} is the universal closure of the formula
\begin{align*}
\forall a\, \bigl(\forall n < a \; \exists m \; \varphi(n,m) \;\imp\; \exists b \; \forall n < a \; \exists m < b \; \varphi(n,m)\bigr),
\end{align*}
where $a$ and $b$ are not free in $\varphi$.
\end{itemize}

For a class of formulas $\Gamma$, the \emph{$\Gamma$ induction scheme} ($\mathsf{I}\Gamma$) consists of the induction axioms for all $\varphi \in \Gamma$, the \emph{$\Gamma$ least element principle} consists of the least element principles for all $\varphi \in \Gamma$, the \emph{bounded $\Gamma$ comprehension scheme} consists of the bounded comprehension axioms for all $\varphi \in \Gamma$, and the \emph{$\Gamma$ bounding scheme} ($\mathsf{B}\Gamma$) consists of the bounding axioms for all $\varphi \in \Gamma$.  For example, $\bsig{2}$ consists of the bounding axioms for all $\Sigma^0_2$ formulas.  Beyond $\rca$, we are mostly interested in $\bsig{2}$ and $\isig{2}$.  The following list summarizes the relationships between the systems and principles of Figure~\ref{fig-summary} and the first-order schemes.

\begin{itemize}
\item $\aca$ proves the induction axiom, least element principle, bounded comprehension axiom, and bounding axiom for every arithmetical formula.

\medskip

\item In addition to $\isig{1}$, $\rca$ proves $\ipi{1}$, the $\Sigma^0_1$ least element principle, the $\Pi^0_1$ least element principle, the bounded $\Sigma^0_1$ comprehension scheme, the bounded $\Pi^0_1$ comprehension scheme, and $\bsig{1}$ (see~\cite{HajekPudlak}*{Section~I.2} and~\cite{SimpsonSOSOA}*{Section~II.3}).

\medskip

\item Neither $\rca$ nor $\wkl$ proves $\bsig{2}$ (see~\cite{SimpsonSOSOA}*{Sections~IX.1 and~IX.2}).

\medskip

\item $\rca$ proves that $\bpi{1}$, $\bsig{2}$, and $\rt^1_{<\infty}$ are equivalent; that $\isig{2}$, $\ipi{2}$, the $\Sigma^0_2$ least element principle, the $\Pi^0_2$ least element principle, the bounded $\Sigma^0_2$ comprehension scheme, and the bounded $\Pi^0_2$ comprehension scheme are all equivalent; and that $\isig{2}$ implies $\bsig{2}$ (see~\cite{HajekPudlak}*{Section~I.2}, \cite{HirstThesis}, and~\cite{SimpsonSOSOA}*{Section~II.3}).

\medskip

\item $\sads$ and all the systems and principles above it in Figure~\ref{fig-summary} imply $\bsig{2}$ over $\rca$.

\medskip

\item $\rca + \rt^2_2$ does not prove $\isig{2}$~\cite{ChongSlamanYangInductiveStrength}.
\end{itemize}

We emphasize that $\bsig{2}$ and $\rt^1_{<\infty}$ are equivalent over $\rca$ because we use this equivalence often and without special mention.

\section{From one principle to many}\label{sec-Dilworth}

We begin studying the Rival--Sands theorem for partial orders from the perspective of reverse mathematics.  Typically we use $\omega$ to denote the order-type of $(\Nb, <)$, use $\omega^*$ to denote the order-type of its reverse, and use $\zeta$ to denote the order-type $\omega^* + \omega$ of the integers.  In some places we may write $<_\Nb$, $\leq_\Nb$, etc.\ instead of $<$, $\leq$, etc.\ to help disambiguate several orders under discussion.

\begin{Definition}\label{def-HandW}
Let $(P, <_P)$ be a partial order.
\begin{itemize}
\item Elements $p, q \in P$ are \emph{comparable} (written $p \comp_P q$) if either $p \leq_P q$ or $q \leq_P p$.  If $p$ and $q$ are not comparable, then they are \emph{incomparable} (written $p \mid_P q$).

\medskip

\item A \emph{chain} in $P$ is a set $C \subseteq P$ of pairwise comparable elements.

\medskip

\item An \emph{antichain} in $P$ is a set $X \subseteq P$ of pairwise incomparable elements.

\medskip

\item $P$ has \emph{width} $\leq\! k$ if every antichain in $P$ has at most $k$ elements.  $P$ has \emph{width} $k$ if it has width $\leq\! k$ but not width $\leq\! k-1$.  $P$ has \emph{finite width} if $P$ has width $\leq\! k$ for some $k$.

\medskip

\item $P$ has \emph{height} $\leq\! k$ if every chain in $P$ has at most $k$ elements.  $P$ has \emph{height} $k$ if it has height $\leq\! k$ but not height $\leq\! k-1$.  $P$ has \emph{finite height} if $P$ has height $\leq\! k$ for some $k$.
\end{itemize}
\end{Definition}

We use the \emph{homogeneous} terminology from Ramsey's theorem to describe chains that are as in the conclusion to the Rival--Sands theorem for partial orders.

\begin{Definition}\label{def-ChainHom}
Let $(P, <_P)$ be a partial order.  An infinite chain $C \subseteq P$ is:
\begin{itemize}
\item \emph{$(0,\infty)$-homogeneous} for $P$ if every $p \in P$ is either comparable with no element of $C$ or is comparable with infinitely many elements of $C$;

\medskip

\item \emph{$(0,\cof)$-homogeneous} for $P$ if every $p \in P$ is either comparable with no element of $C$ or is comparable with cofinitely many elements of $C$.
\end{itemize}
\end{Definition}

Every infinite subset of a $(0,\cof)$-homogeneous chain in a partial order is also $(0,\cof)$-homogeneous, but an infinite subset of a $(0,\infty)$-homogeneous chain need not be $(0,\infty)$-homogeneous.  Rival and Sands observed that if $C$ is a chain of order-type $\zeta$ in a partial order $(P, <_P)$, then $C$ is automatically $(0,\infty)$-homogeneous for $P$.  To wit, if $p \in P$ is comparable with some $q \in C$, then either $p \leq_P q$ and hence $p$ is below infinitely many elements of $C$, or $p \geq_P q$ and hence $p$ is above infinitely many elements of $C$.  Similarly, if $C$ is a $(0,\infty)$-homogeneous chain of order-type either $\omega$ or $\omega^*$, then $C$ is automatically $(0,\cof)$-homogeneous.  For example, if $C$ has order-type $\omega$ and $p \in P$ is comparable with infinitely many elements of $C$, then either $p$ is above all elements of $C$, or $p$ is below some element of $C$ and therefore below cofinitely many elements of $C$.

We also apply the \emph{$(0,\infty)$-homogeneous} and \emph{$(0,\cof)$-homogeneous} terminology to sequences.  An infinite sequence $\la x_n : n \in \Nb \ra$ of distinct elements in a partial order $(P, <_P)$ is \emph{$(0,\infty)$-homogeneous} if every $p \in P$ is either comparable with $x_n$ for no $n$ or is comparable with $x_n$ for infinitely many $n$; and it is \emph{$(0,\cof)$-homogeneous} if every $p \in P$ is either comparable with $x_n$ for no $n$ or is comparable with $x_n$ for cofinitely many $n$.  As with chains of order-type $\omega$ and $\omega^*$, an infinite sequence that is $(0,\infty)$-homogeneous and either ascending or descending is automatically $(0,\cof)$-homogeneous, and therefore all of its infinite subsequences are $(0,\cof)$-homogeneous as well.

We introduce several formulations of the Rival--Sands theorem for partial orders.

\begin{Definition}
{\ }
\begin{itemize}
\item $\rspo_k$ is the statement ``Every infinite partial order of width $\leq\! k$ has a $(0, \infty)$-homogeneous chain.''

\medskip

\item $\rspo_{<\infty}$ abbreviates $\forall k \, \rspo_k$.

\medskip

\item $\cofrspo_k$ is the statement ``Every infinite partial order of width $\leq\! k$ has a $(0, \cof)$-homogeneous chain.''

\medskip

\item $\cofrspo_{<\infty}$ abbreviates $\forall k \, \cofrspo_k$.
\end{itemize}
\end{Definition}

Immediately, $\rca \vdash \forall k\, \bigl(\cofrspo_k \imp \rspo_k\bigr)$ and therefore $\rca \vdash \cofrspo_{<\infty} \imp \rspo_{<\infty}$.  Also, Proposition~\ref{prop-COFvsAscDesc} shows that for every $k \geq 2$, $\cofrspo_k$ is equivalent to the statement ``Every infinite partial order of width $\leq\! k$ has a $(0,\infty)$-homogeneous chain of order-type either $\omega$ or $\omega^*$.''

When working with partial orders of finite width, it often helps to decompose the partial order into a finite union of chains.

\begin{Definition}
A \emph{$k$-chain decomposition} of a partial order $(P, <_P)$ is a collection of $k$ chains $C_0, C_1, \dots, C_{k-1} \subseteq P$ where $P = \bigcup_{i < k} C_i$.  If $P$ has a $k$-chain decomposition, then it is called \emph{$k$-chain decomposable}.\end{Definition}

We emphasize that if a partial order $P$ is assumed to be $k$-chain decomposable, then $P$ comes along with a $k$-chain decomposition.  Of course, we may always assume that the chains of a chain decomposition are pairwise disjoint.

Recall now Dilworth's theorem, which in this terminology states that for every $k$, every partial order of width $\leq\! k$ is $k$-chain decomposable.  Hirst~\cite{HirstThesis} shows that Dilworth's theorem for countable partial orders is equivalent to $\wkl$ over $\rca$.  He also shows that Dilworth's theorem remains equivalent to $\wkl$ even when restricted to partial orders of width $2$.  It follows that there is a recursive partial order of width $2$ that cannot be decomposed into $2$ recursive chains.  Thus Dilworth's theorem is not available when working in $\rca + \isig{2} + \ads$ because $\rca + \isig{2} + \ads \nvdash \wkl$.  However, for our purposes it is not necessary to decompose a partial order of finite width into the optimal number of chains---any decomposition into finitely many chains will do.  Thus we replace Dilworth's theorem by Kierstead's effective analog, which states that every recursive partial order of width $\leq\! k$ can be decomposed into at most $(5^k - 1)/4$ recursive chains~\cite{KiersteadDilworth}.  Nowadays much better sub-exponential bounds are known for the number of recursive chains into which a recursive partial order of finite width can be decomposed~\cite{BosekKiersteadKrawczykMateckiSmith}.

Let $(P, <_P)$ be a partial order of width $\leq\! k$.  Kierstead's proof is phrased as an induction on $k$.  By unwinding the induction, the proof can be viewed as an on-line algorithm computing a sequence of partial orders $(P, <_P) = (P_0, <_{P_0}), (P_1, <_{P_1}), \dots, (P_{k-2}, <_{P_{k-2}})$, where $P_i$ has width $\leq\! k-i$ for each $i \leq k-2$, together with a $(5^{k-i} - 1)/4$-chain decomposition of $(P_i, <_{P_i})$ for each $i \leq k-2$.  Kierstead himself comments along these lines following the proof of~\cite{KiersteadDilworth}*{Theorem~1.10}.  With this view, it is possible to verify that Kierstead's theorem is provable in $\rca$.

\begin{Theorem}[essentially Kierstead~\cite{KiersteadDilworth}]\label{thm-Kierstead}
$\rca$ proves the statement ``For every $k$, every partial order of width $\leq\! k$ is $(5^k - 1)/4$-chain decomposable.''
\end{Theorem}

For notational ease, we work with $5^k$-chain decompositions in place of $(5^k - 1)/4$-chain decompositions.  Again, for our purposes, any primitive recursive bound suffices.

For completeness, we mention that there is a dual version of Dilworth's theorem, due to Mirsky, which states that every partial order of height $\leq\! k$ can be decomposed into a union of $k$ antichains.  Hirst proved that Mirsky's theorem for countable partial orders is equivalent to $\wkl$ over $\rca$~\cite{HirstThesis}.  There is also an effective analog of Mirsky's theorem in the spirit of Theorem~\ref{thm-Kierstead}~\cite{KiersteadOrderedSets}.

We now formalize versions of $\rspo$ with the assumption ``width $\leq\! k$'' replaced by ``$k$-chain decomposable.''

\begin{Definition}
{\ }
\begin{itemize}
\item $\rspocd_k$ is the statement ``Every infinite $k$-chain decomposable partial order has a $(0, \infty)$-homogeneous chain.''

\medskip

\item $\rspocd_{<\infty}$ abbreviates $\forall k \, \rspocd_k$.

\medskip

\item $\cofrspocd_k$ is the statement ``Every infinite $k$-chain decomposable partial order has a $(0, \cof)$-homogeneous chain.''

\medskip

\item $\cofrspocd_{<\infty}$ abbreviates $\forall k \, \cofrspocd_k$.
\end{itemize}
\end{Definition}

We have that $\rca \vdash \forall k \, \bigl((\rspocd_{5^k} \imp \rspo_k) \;\andd\; (\rspo_k \imp \rspocd_k)\bigr)$ and analogously for the $(0,\cof)$-homogeneous versions by Theorem~\ref{thm-Kierstead} and the fact that $k$-chain decomposable partial orders have width $\leq\! k$.  It follows that $\rca \vdash \rspo_{<\infty} \biimp \rspocd_{<\infty}$ and that $\rca \vdash \cofrspo_{<\infty} \biimp \cofrspocd_{<\infty}$.  Additionally, $\wkl \vdash \forall k \, (\rspo_k \biimp \rspocd_k)$ and analogously for the $(0,\cof)$-homogeneous versions because $\wkl$ proves Dilworth's theorem.

We conclude this section with a few other useful applications of Theorem~\ref{thm-Kierstead}.  If $(P, <_P)$ is an infinite partial order of width $\leq\! k$, then $\cac$ implies that $P$ contains an infinite chain because $P$ does not contain an infinite antichain.  However, we may argue more effectively by instead applying Theorem~\ref{thm-Kierstead} to $P$ to obtain a $5^k$-chain decomposition of $P$ and then by applying the pigeonhole principle to conclude that one of these chains must be infinite.  Dually, $\cac$ implies that an infinite partial order of height $\leq\! k$ contains an infinite antichain, and this fact can be effectivized as well.  We show that these special cases of $\cac$ are provable in $\rca$ for each fixed $k$ and are equivalent to $\bsig{2}$ over $\rca$ for arbitrary $k$.

\begin{Definition}
{\ }
\begin{itemize}
\item $\cc_k$ is the statement ``Every infinite partial order of width $\leq\! k$ has an infinite chain.''

\medskip

\item $\cc_{<\infty}$ abbreviates $\forall k \, \cc_k$.

\medskip

\item $\ca_k$ is the statement ``Every infinite partial order of height $\leq\! k$ has an infinite antichain.''

\medskip

\item $\ca_{<\infty}$ abbreviates $\forall k \, \ca_k$.
\end{itemize}
\end{Definition}

\begin{Proposition}\label{prop-CC-CA}
{\ }
\begin{enumerate}[(1)]
\item For each fixed standard $k$, $\rca \vdash \cc_k$ and $\rca \vdash \ca_k$.

\medskip

\item\label{it-CC-CA-BSig2} $\cc_{<\infty}$, $\ca_{<\infty}$, and $\bsig{2}$ are pairwise equivalent over $\rca$.
\end{enumerate}
\end{Proposition}

\begin{proof}
For $\rca \vdash \cc_k$, let $(P, <_P)$ be an infinite partial order of width $\leq\! k$.  By Theorem~\ref{thm-Kierstead}, $P$ has a $5^k$-chain decomposition, and at least one of these chains is infinite by $\rt^1_{5^k}$.  Thus $P$ has an infinite chain.  The proof that $\rca + \bsig{2} \vdash \cc_{<\infty}$ is the same, except we must use $\rt^1_{<\infty}$ instead of $\rt^1_{5^k}$ because now $k$ is not fixed in advance.

For $\rca \vdash \ca_k$, we give a direct proof instead of appealing to an effective analog of Mirsky's theorem.  Let $(P, <_P)$ be an infinite partial order of height $\leq\! k$.  Define a coloring $c \colon P \imp k^2$ by $c(p) = \la x, y \ra$, where $x$ is the greatest size of a $<_P$-chain in $\{q <_\Nb p : q <_P p\}$ and $y$ is the greatest size of a $<_P$-chain in $\{q <_\Nb p : p <_P q\}$.  Both $x$ and $y$ are less than $k$ because $P$ has height $\leq\! k$.  By $\rt^1_{k^2}$, let $H \subseteq P$ be an infinite set that is homogeneous for $c$.  We claim that $H$ is an antichain.  Suppose for a contradiction that $a, b \in H$ and $a <_P b$.  If $a <_\Nb b$ and $q_0 <_P q_1 <_P \cdots <_P q_{n-1}$ is a chain in $\{q <_\Nb a : q <_P a\}$, then $q_0 <_P q_1 <_P \cdots <_P q_{n-1} <_P a$ is a chain in $\{q <_\Nb b : q <_P b\}$.  This means that if $x$ is the maximum size of a chain in $\{q <_\Nb a : q <_P a\}$, then the maximum size of a chain in $\{q <_\Nb b : q <_P b\}$ is at least $x+1$.  So $c(a) \neq c(b)$, contradicting that $H$ is homogeneous.  Similar reasoning shows that if $b <_\Nb a$, then $c(a) \neq c(b)$ as well.  Therefore $H$ is an infinite antichain in $P$.  The proof that $\rca + \bsig{2} \vdash \ca_{<\infty}$ is the same, except we must use $\rt^1_{<\infty}$ instead of $\rt^1_{k^2}$ because now $k$ is not fixed in advance.

For $\rca + \cc_{<\infty} \vdash \bsig{2}$, let $c \colon \Nb \imp k$ be a $k$-coloring, and define a partial order $(P, <_P)$ with $P = \Nb$ by setting $p <_P q$ if and only if $p <_\Nb q$ and $c(p) = c(q)$.  The partial order $P$ has width $\leq\! k$, so by $\cc_{<\infty}$ it has an infinite chain $H$.  By the definition of $<_P$, $H$ must be homogeneous for $c$.  Therefore $\rt^1_{<\infty}$ holds.

For $\rca + \ca_{<\infty} \vdash \bsig{2}$, let $c \colon \Nb \imp k$ be a $k$-coloring, and define a partial order $(P, <_P)$ with $P = \Nb$ by setting $p <_P q$ if and only if $c(p) < c(q)$.  The partial order $P$ has height $\leq\! k$, so by $\ca_{<\infty}$ it has an infinite antichain $H$.  By the definition of $<_P$, $H$ must be homogeneous for $c$.  Therefore $\rt^1_{<\infty}$ holds.
\end{proof}

We now show that $\ads$ is equivalent to the statement ``Every infinite partial order of finite width contains either an infinite ascending sequence or an infinite descending sequence.''

\begin{Proposition}
The following are equivalent over $\rca$.
\begin{enumerate}[(1)]
\item\label{it-ADSFinWidth} $\ads$.

\medskip

\item\label{it-FinWidthADS} Every infinite partial order of finite width contains either an infinite ascending sequence or an infinite descending sequence.

\medskip

\item\label{it-FixedWidthADS}  For each fixed standard $k \geq 1$, the statement ``Every infinite partial order of width $\leq\! k$ contains either an infinite ascending sequence or an infinite descending sequence.''
\end{enumerate}
\end{Proposition}

\begin{proof}
For \ref{it-ADSFinWidth}~$\Imp$~\ref{it-FinWidthADS}, let $(P, <_P)$ be an infinite partial order of width $\leq\! k$ for some $k$.  Using the fact that $\rca + \ads \vdash \bsig{2}$, we may appeal to Proposition~\ref{prop-CC-CA} item~\ref{it-CC-CA-BSig2} and apply $\cc_{<\infty}$ to $P$ to obtain an infinite chain $C$ in $P$.  Now apply $\ads$ to $C$ to obtain either an infinite ascending sequence in $C$ or an infinite descending sequence in $C$.

The implications \ref{it-FinWidthADS}~$\Imp$~\ref{it-FixedWidthADS} and \ref{it-FixedWidthADS}~$\Imp$~\ref{it-ADSFinWidth} are immediate because every partial order of width $\leq\! k$ has finite width, and every infinite linear order is an infinite partial order of width $1$.
\end{proof}

\section{First proofs of \texorpdfstring{$\rspo_{<\infty}$}{RSpo\_(< infinity)} and \texorpdfstring{$\cofrspo_{<\infty}$}{(0,cof)-RSpo\_(< infinity)}}\label{sec-FirstProofs}

We give a proof of $\rspo_{<\infty}$ in $\aca$ and a proof of $\cofrspo_{<\infty}$ in $\pica$.  The proofs are not axiomatically optimal, but they can be presented in ordinary mathematical language---meaning without reference to relative computability, technical uses of restricted induction, etc.---and are straightforward to formalize.  The $\aca$ proof in particular strikes a good balance between axiomatic simplicity and conceptual simplicity.  It is based on Dilworth's theorem, the fact that chains of order-type $\zeta$ are automatically $(0,\infty)$-homogeneous, and the observation that a linear order containing no suborder of type $\zeta$ can be partitioned into a well-founded part and a reverse well-founded part.  This last observation requires the full strength of $\aca$, as shown by Lemma~\ref{lem-ChainSplitting}.  This section also serves to introduce key concepts that will be refined in the next section to prove $\rspo_{<\infty}$ in $\rca + \isig{2} + \ads$.

\begin{Definition}\label{def-poDefs}
Let $(P, <_P)$ be a partial order, and let $X, Y \subseteq P$.
\begin{itemize}
\item Write $X <_P Y$ if every element of $X$ is strictly below every element of $Y$:  $\forall x \in X \; \forall y \in Y \; (x <_P y)$.  In the case of singletons, write $x <_P Y$ and $X <_P y$ in place of $\{x\} <_P Y$ and $X <_P \{y\}$.

\medskip

\item Write $X \pwb Y$ if every element of $X$ is below some element of $Y$:  $\forall x \in X \; \exists y \in Y \; (x \leq_P y)$.

\medskip

\item Write $X \mid_P Y$ if every element of $X$ is incomparable with every element of $Y$:  $\forall x \in X \; \forall y \in Y \; (x \mid_P y)$.  In the case of singletons, write $x \mid_P Y$ in place of $\{x\} \mid_P Y$.
\end{itemize}
\end{Definition}

We also extend the notation of Definition~\ref{def-poDefs} to sequences.  For example, if $A = \la a_n : n \in \Nb \ra$ and $B = \la b_n : n \in \Nb \ra$ are sequences in a partial order $(P, <_P)$, then we write $A \pwb B$ if $\forall n \, \exists m \, (a_n \leq_P b_m)$.

\begin{Definition}
Let $(P, <_P)$ be a partial order, and let $X \subseteq P$.
\begin{itemize}
\item  Let $X\da = \{p \in P : \exists x \in X \; (p \leq_P x)\}$ and $X\ua = \{p \in P : \exists x \in X \; (p \geq_P x)\}$ denote the downward and upward closures of $X$ in $P$, respectively. These sets may be formed in $\aca$.

\medskip

\item  In the case of singletons $X = \{x\}$, write $x\da$ and $x\ua$ in place of $\{x\}\da$ and $\{x\}\ua$.  These sets may be formed in $\rca$.

\medskip

\item Call $X$ \emph{well-founded} if it contains no infinite descending sequence.  Otherwise call $X$ \emph{ill-founded}.  Likewise, call $X$ \emph{reverse well-founded} if it contains no infinite ascending sequence.  Otherwise call $X$ \emph{reverse ill-founded}.
\end{itemize}
\end{Definition}

For a partial order $(P, <_P)$ and non-empty subsets $X, Y, Z \subseteq P$, $\rca$ suffices to show that $X <_P Y <_P Z$ implies $X <_P Z$ and that $X \pwb Y \pwb Z$ implies $X \pwb Z$.  Also, notice that $X \pwb Y$ simply means that $X \subseteq Y\da$, but beware that forming the set $Y\da$ requires $\aca$ in general.

As mentioned above, we show that partitioning a linear order with no suborder of type $\zeta$ into a well-founded part and a reverse well-founded part is equivalent to $\aca$ over $\rca$.  More generally, we show that isolating the well-founded part of a partial order that contains no infinite antichain and no suborder of type $\zeta$ is equivalent to $\aca$ over $\rca$.  Finding the well-founded part of such a partial order is used to extend $\rspo_{<\infty}$ to infinite partial orders without infinite antichains in Section~\ref{sec-ExtendedRSpo}.  The reversal exploits the tool of \emph{true and false numbers} of an injection $f \colon \Nb \imp \Nb$.

\begin{Definition}\label{def-true}
Let $f \colon \Nb \imp \Nb$ be an injection.  An $n \in \Nb$ is a \emph{true number} if $\forall k > n \; (f(n) < f(k))$, and otherwise $n$ is a \emph{false number}.  Additionally, an $n \in \Nb$ is \emph{true at stage $m$} if $\forall k \, (n < k \leq m \;\imp\; f(n) < f(k))$, and otherwise $n$ is \emph{false at stage $m$}.
\end{Definition}

The idea of true numbers appears to have originated with Dekker~\cite{Dekker}, who called them \emph{minimal}.  True numbers are important because the range of $f$ is recursive in the join of $f$ with any infinite set of true numbers. In fact, if $n$ is a true number, then one can determine $\ran(f)$ up to $f(n)$ by simply evaluating $f$ on inputs $0, \dots, n$.  In reverse mathematics, true numbers facilitate reversals to $\aca$.  To prove that some statement $\varphi$ implies $\aca$ over $\rca$, one strategy is to let $f \colon \Nb \to \Nb$ be an injection, use $\varphi$ to produce an infinite set $S$ of true numbers, use $f$ and $S$ to compute $\ran(f)$, and then apply Lemma~\ref{lem-ACAinjection}.  For example, this strategy is used in~\cites{MarconeShore, FrittaionHendtlassMarconeShaferVanderMeeren, FrittaionMarcone} in the form of the following well-known construction.

\begin{Construction}\label{const-TFstages}
Let $f\colon \Nb \to \Nb$ be an injection. Define a linear order $(L, <_L)$ where $L = \{\ell_n : n \in \Nb\}$ and for each $n < m$ the following hold:
\begin{enumerate}[(1)]
\item $\ell_n <_L \ell_m$ if $f(k) < f(n)$ for some $k$ with $n <k \leq m$ (i.e., $n$ is false at stage $m$), and

\medskip

\item $\ell_m <_L \ell_n$ if $f(n) < f(k)$ for all $k$ with $n < k \leq m$ (i.e., $n$ is true at stage $m$).
\end{enumerate}
This construction can be carried out in $\rca$.
\end{Construction}

Given an injection $f \colon \Nb \imp \Nb$, Construction~\ref{const-TFstages} produces a stable linear order either of type $\omega + \omega^*$ (if $f$ has infinitely many false numbers) or of type $k + \omega^*$ for some finite $k$ (otherwise).  $\rca$ proves that $n$ is true if and only if $n$ is in the $\omega^*$-part of $L$.  Therefore, $\rca$ proves that if there is an infinite subset of the $\omega^*$-part of $L$, or, equivalently, if there is an infinite descending sequence in $L$, then the range of $f$ exists.  For further details, see the proofs of~\cite{MarconeShore}*{Lemma~4.2} and~\cite{FrittaionMarcone}*{Theorem~4.5}.

\begin{Lemma}\label{lem-ChainSplitting}
The following are equivalent over $\rca$.
\begin{enumerate}[(1)]
\item\label{it-ACACS} $\aca$.

\medskip

\item\label{it-poSplitWF} For every partial order $(P, <_P)$, if $P$ has no infinite antichain and no suborder of type $\zeta$, then there is a set $W \subseteq P$ such that $\forall p \in P \; (p \in W \;\biimp\; \textup{$p\da$ is well-founded})$.

\medskip

\item\label{it-ChainSplit} Every linear order $(L, <_L)$ with no suborder of type $\zeta$ can be partitioned as $L = W \cup R$, where
\begin{itemize}
\item $W <_L R$,
\item $W$ is well-founded, and
\item $R$ is reverse well-founded.
\end{itemize}
\end{enumerate}
\end{Lemma}

\begin{proof}
For \ref{it-ACACS}~$\Imp$~\ref{it-poSplitWF}, let $(P, <_P)$ be a partial order with no infinite antichain and no suborder of type $\zeta$.  For the purposes of this proof, call a descending sequence $p_0 >_P p_1 >_P \cdots >_P p_n$ in $P$ \emph{discrete} if for all $i < n$ there is no element of $P$ strictly between $p_i$ and $p_{i+1}$.  Define a sequence of trees $\la T_p : p \in P \ra$, where for each $p \in P$, $T_p$ consists of the finite discrete descending sequences starting at $p$.  That is, for each $p \in P$, $T_p$ consists of the $\sigma \in P^{<\Nb}$ such that
\begin{align*}
(|\sigma| &> 0 \;\imp\; \sigma(0) = p)\\
&\andd \quad \forall i < |\sigma|-1 \; \Bigl(\sigma(i+1) <_P \sigma(i) \;\andd\; \neg \exists x \in P \; \bigl(\sigma(i+1) <_P x <_P \sigma(i)\bigr)\Bigr).
\end{align*}
Let $W = \{p \in P : \text{$T_p$ is finite}\}$.  We show that $p \in W$ if and only if $p\da$ is well-founded.

First, consider a $p \in P$ where $T_p$ is infinite.  Given any $\sigma \in T_p$, the set $\{q \in P : \sigma^\smf q \in T_p\}$ is an antichain on account of the discreteness condition on the elements of $T_p$ and therefore is finite because $P$ has no infinite antichain.  Thus $T$ is an infinite finitely branching tree, and therefore $T$ has an infinite path $f$ by K\"{o}nig's lemma.  This path provides an infinite descending sequence in $P$ below $p$, so $p\da$ is ill-founded.

Conversely, consider a $p \in P$ where $T_p$ is finite.  Suppose for a contradiction that $p\da$ is ill-founded, and let $p = q_0 >_P q_1 >_P q_2 >_P \cdots$ be an infinite descending sequence below $P$.  Notice that $\sigma = \la p \ra$ is in $T_p$ and satisfies $\sigma(0) >_P q_1$.  Let $\sigma \in T_p$ be a non-empty sequence of maximum length for which there is an $n$ such that $\sigma(|\sigma|-1) >_P q_n$.  Such a $\sigma$ exists because $T_p$ is finite.  Now define an infinite ascending sequence $A = \la a_i : i \in \Nb \ra$ with $\sigma(|\sigma|-1) >_P a_i \geq_P q_n$ for all $i$ as follows.  Let $a_0 = q_n$.  Given $a_i$, let $a_{i+1}$ be the $<$-least element of $P$ with $\sigma(|\sigma|-1) >_P a_{i+1} >_P a_i$.  Such an $a_{i+1}$ exists because otherwise we would have $\sigma^\smf a_i \in T_p$ and $a_i >_P q_{n+1}$, which contradicts that $\sigma$ has maximum length.  We now have that $\{a_i : i \in \Nb\} \cup \{q_i : i > n\}$ is a chain in $P$ of order-type $\zeta$, which is a contradiction.  Thus $p\da$ is well-founded, as desired.

For \ref{it-poSplitWF}~$\Imp$~\ref{it-ChainSplit}, let $(L, <_L)$ be a linear order with no suborder of type $\zeta$.  Then $L$ has no infinite antichain, so by item~\ref{it-poSplitWF} there is a set $W$ containing exactly the $p \in P$ for which $p\da$ is well-founded.  Clearly $W$ is downward-closed.  Let $R = L \setminus W$.  Then $L = W \cup R$ and $W <_L R$.  We show that $R$ is reverse well-founded.  First, observe that $R$ has no least element.  If $r$ were the least element of $R$, then $r\da$ would be well-founded because every $p <_L r$ would be in $W$ and $W$ is well-founded.  This would imply that $r \in W$, which is a contradiction.  Thus either $R = \emptyset$, in which case $R$ is reverse well-founded, or $R$ is non-empty and has no minimum element, in which case we can define an infinite descending sequence $\la d_n : n \in \Nb \ra$ that is coinitial in $R$.  If $R$ also has an infinite ascending sequence $B = \la b_n : n \in \Nb \ra$, then $b_0 >_L d_{n_0}$ for some $n_0$, in which case $\{d_n : n > n_0\} \cup B$ is a contradictory suborder of $L$ of type $\zeta$.  Thus $R$ is reverse well-founded.

For \ref{it-ChainSplit}~$\Imp$~\ref{it-ACACS}, let $f \colon \Nb \to \Nb$ be an injection.  We show that the true numbers for $f$ form a set.  This implies that the range of $f$ exists as a set, which implies $\aca$ by Lemma~\ref{lem-ACAinjection}.

If $f$ has only finitely many false numbers, then the set of all false numbers exists by bounded $\Sigma^0_1$ comprehension, in which case the set of true numbers also exists.

Suppose instead that $f$ has infinitely many false numbers.  Let $(L, <_L)$ be the linear order defined as in Construction~\ref{const-TFstages} for $f$.  Recall that in this case $L = \{\ell_n : n \in \Nb\}$ is a linear order of type $\omega + \omega^*$, where, for each $n$, $\ell_n$ is in the $\omega$-part if $n$ is false and $\ell_n$ is in the $\omega^*$-part if $n$ is true.  We define a new linear order $(S, <_S)$ from $L$ by replacing each element in the $\omega^*$-part of $L$ by an infinite descending sequence and by replacing each element in the $\omega$-part of $L$ by a finite descending sequence.  This way, $\rca$ suffices to verify that elements in the $\omega^*$-part of $L$ give rise to elements in the ill-founded part of $S$ and that elements in the $\omega$-part of $L$ give rise to elements in the well-founded part of $S$.  To do this, let $S = \{s_{n,m} : \text{$n, m \in \Nb$ and $n$ is true at stage $m$}\}$ (note that if $m \leq n$, then $n$ is true at stage $m$), and define
\begin{align*}
s_{n_0, m_0} <_S s_{n_1, m_1} \quad\Biimp\quad (\ell_{n_0} <_L \ell_{n_1}) \;\orr\; (\ell_{n_0} = \ell_{n_1} \;\andd\; m_0 >_{\Nb} m_1).
\end{align*}
Observe that if $n_0$ is false and $n_1$ is true, then $\ell_{n_0} <_L \ell_{n_1}$, so $s_{n_0, m_0} <_S s_{n_1, m_1}$ for every $m_0$ and $m_1$.  Thus no infinite ascending sequence in $S$ can contain an element $s_{n,m}$ where $n$ is true, and no infinite descending sequence in $S$ can contain an element $s_{n,m}$ where $n$ is false.  It follows that $S$ cannot contain a suborder of type $\zeta$ because such a suborder would have to contain some element $s_{n,m}$, and $s_{n,m}$ is either in no infinite ascending sequence or in no infinite descending sequence.  We may therefore apply item~\ref{it-ChainSplit} to $S$ and obtain a partition $S = W \cup R$ where $W <_L R$, $W$ is well-founded, and $R$ is reverse well-founded.  We claim that $s_{n, 0} \in R$ if and only if $n$ is true.  If $n$ is true, then $s_{n,m} \in S$ for every $m$, and $s_{n,0} >_S s_{n,1} >_S \cdots$ is an infinite descending sequence in $S$.  Thus $s_{n,0}$ cannot be in $W$ as then $W$ would be ill-founded.  So $s_{n,0} \in R$.  Conversely, if $n$ is false, then, using the assumption that there are infinitely many false numbers, we can define an infinite ascending sequence $\ell_n = \ell_{k_0} <_L \ell_{k_1} <_L \cdots$ in $L$ as follows.  Set $k_0 = n$.  Given $k_i$, search for the first pair $\la k, m \ra$ where $\ell_{k_i} <_L \ell_k$ and $k$ is false at stage $m$, and set $k_{i+1} = k$.  We then have the corresponding infinite ascending sequence $s_{n,0} = s_{k_0, 0} <_S s_{k_1, 0} <_S \cdots$ in $S$.  Thus $s_{n,0}$ cannot be in $R$ as then $R$ would not be reverse well-founded.  Therefore $\{n : s_{n,0} \in R\}$ is the set of true numbers for $f$, which completes the proof.
\end{proof}

By taking complements and/or reversing the partial order, the statement ``$p\da$ is well-founded'' may be replaced by any of ``$p\da$ is ill-founded,'' ``$p\ua$ is reverse well-founded,'' and ``$p\ua$ is reverse ill-founded'' in Lemma~\ref{lem-ChainSplitting} item~\ref{it-poSplitWF} and the lemma remains true.

The proof that $\aca$ implies Lemma~\ref{lem-ChainSplitting} item~\ref{it-poSplitWF} is uniform with respect to the partial order.  That is, $\aca$ proves that if $P_0, P_1, P_2, \dots$ is a sequence of partial orders without infinite antichains or suborders of type $\zeta$, then there is a sequence  of sets $W_0, W_1, W_2, \dots$ where, for each $i$, $W_i$ consists of exactly the $p \in P_i$ such that $p\da$ is well-founded.  The proof that $\aca$ implies Lemma~\ref{lem-ChainSplitting} item~\ref{it-ChainSplit} is similarly uniform.  In fact, to conclude Lemma~\ref{lem-ChainSplitting} item~\ref{it-ChainSplit} for a finite sequence of linear orders $L_0, L_1, \dots, L_{n-1}$, as we shall need in Theorem~\ref{thm-RSpoInACA}, one application of Lemma~\ref{lem-ChainSplitting} item~\ref{it-poSplitWF} suffices.  Given linear orders $L_0, L_1, \dots, L_{n-1}$ without suborders of type $\zeta$, let $P$ be the partial order consisting of the disjoint union of the linear orders $L_i$ for $i < n$.  Then $P$ has no infinite antichain and no suborder of type $\zeta$, so by Lemma~\ref{lem-ChainSplitting} item~\ref{it-poSplitWF}, there is a set $W$ consisting of exactly the $p \in P$ such that $p\da$ is well-founded.  Then let $W_i = L_i \cap W$ and $R_i = L_i \setminus W$ for each $i < n$.

If an infinite ascending sequence $A = \la a_n : n \in \Nb \ra$ in a partial order $(P, <_P)$ is not $(0,\infty)$-homogeneous, then there is a $p \in P$ that is comparable with some elements of $P$, but only finitely many.  As $A$ is an ascending sequence, this means that there is an $n_0$ such that $p >_P a_{n_0}$, but $\forall n > n_0 \, (p \mid_P a_n)$.  We think of such a $p$ as a \emph{counterexample} to $A$ being $(0,\infty)$-homogeneous.  Indeed, $p$ is also a counterexample to $\la a_n :  n \geq n_0 \ra$ being $(0,\infty)$-homogeneous.

\begin{Definition}\label{def-tail}
Let $(P, <_P)$ be a partial order, and let $A = \la a_n : n \in \Nb \ra$ be an ascending sequence in $P$.  Then $A_{\geq n_0}$ denotes the ascending sequence $\la a_n : n \geq n_0 \ra$.  Sequences of the form $A_{\geq n_0}$ are called \emph{tails} of $A$, which can be formed in $\rca$.
\end{Definition}

\begin{Definition}\label{def-counterexample}
Let $(P, <_P)$ be a partial order, and let $A = \la a_n : n \in \Nb \ra$ be an infinite ascending sequence in $P$.  A $p \in P$ is called a \emph{counterexample to $A$} if there is an $n$ such that $p >_P a_n$ and $p \mid_P A_{\geq n+1}$.  An infinite ascending sequence $B = \la b_\ell : \ell \in \Nb \ra$ is called a \emph{counterexample sequence for $A$} if $B$ contains counterexamples to infinitely many tails of $A$:  $\forall m \; \exists n > m \; \exists \ell \; (b_\ell >_P a_n \;\andd\; b_\ell \mid_P A_{\geq n+1})$.
\end{Definition}

Notice that if $B$ is a counterexample sequence for an infinite ascending sequence $A$ in some partial order $(P, <_P)$, then $A \pwb B$, but $B \npwb A$.

Suppose that $A$ is an infinite ascending sequence in a partial order $(P, <_P)$ where no tail of $A$ is $(0,\infty)$-homogeneous.  Then for every $n$, there is a counterexample $p$ to $A_{\geq n}$.  If $P$ has finite width, then we can make a counterexample sequence out of such counterexamples.

\begin{Lemma}\label{lem-CXChain}
The following is provable in $\aca$. Let $(P,<_P)$ be an infinite partial order with $k$-chain decomposition $C_0, \dots, C_{k-1}$.  Let $A = \la a_n : n \in \Nb \ra$ be an infinite ascending sequence in $C_i$ for some $i < k$, and assume that no tail of $A$ is $(0,\infty)$-homogeneous.  Then there is an infinite ascending sequence $B = \la b_n : n \in \Nb \ra$ in $C_j$ for some $j < k$ that is a counterexample sequence for $A$.
\end{Lemma}

\begin{proof}
We assume that no tail of $A$ is $(0,\infty)$-homogeneous, so every tail of $A$ has a counterexample $p$.  For each $n$, let $p_n$ be the $<$-least counterexample to the tail $A_{\geq n}$.  By $\rt^1_{<\infty}$, there are a $j < k$ and an infinite $X \subseteq \Nb$ such that $p_n \in C_j$ for all $n \in X$.  Now, for every $n \in X$, we have that $p_n <_P p_m$ for all sufficiently large $m \in X$.  To see this, let $n \in X$.  As $p_n$ is a counterexample to $A_{\geq n}$, there is an $s \geq n$ such that $p_n >_P a_s$ and $p_n \mid_P A_{\geq s+1}$.  Let $m \in X$ be such that $m > s+1$, and consider $p_m$.  The chain $C_j$ contains both $p_n$ and $p_m$, so $p_n \comp_P p_m$.  As $p_m$ is a counterexample to $A_{\geq m}$, there is a $t \geq m$ such that $p_m >_P a_t$.  Thus we cannot have have $p_m \leq_P p_n$ because this would yield $a_{s+1} <_P a_t <_P p_m \leq_P p_n$, contradicting that $p_n \mid_P a_{s+1}$.  Note here that $s+1 < m \leq t$, so $a_{s+1} <_P a_t$ because $A$ is an ascending sequence.  Thus it must be that $p_n <_P p_m$.  We may then define the desired counterexample sequence $B$ as follows.  Let $n_0$ be the $<$-least element of $X$.  Given $n_\ell$, let $n_{\ell+1}$ be the $<$-least element of $X$ with $n_\ell < n_{\ell+1}$ and $p_{n_\ell} <_P p_{n_{\ell+1}}$.  Finally, take $b_\ell = p_{n_\ell}$ for each $\ell$.
\end{proof}

We are now prepared to give a proof of $\rspo_{<\infty}$ in $\aca$.

\begin{Theorem}\label{thm-RSpoInACA}
$\aca \vdash \rspo_{<\infty}$.
\end{Theorem}

\begin{proof}
It suffices to show that $\aca \vdash \rspocd_{<\infty}$ because $\rspo_{<\infty}$ and $\rspocd_{<\infty}$ are equivalent over $\rca$ as explained in Section~\ref{sec-Dilworth}.  In fact, here we may use that $\wkl \vdash \forall k \, (\rspo_k \biimp \rspocd_k)$, which is simpler than appealing to Theorem~\ref{thm-Kierstead}.

Let $(P, <_P)$ be an infinite partial order with $k$-chain decomposition $C_0, \dots, C_{k-1}$ for some $k$.  Assume for a contradiction that $P$ does not contain a $(0,\infty)$-homogeneous chain.  Then $P$ contains no chain of order-type $\zeta$ because such a chain is automatically $(0,\infty)$-homogeneous.  In particular, no $C_i$ for $i < k$ contains a chain of order-type $\zeta$, so each $C_i$ may be viewed as a linear order with no suborder of type $\zeta$.  As discussed above, the \ref{it-ACACS}~$\Imp$~\ref{it-ChainSplit} direction of Lemma~\ref{lem-ChainSplitting} generalizes to simultaneously handle any sequence of linear orders without suborders of type $\zeta$.  Apply this to the chains $C_0, \dots, C_{k-1}$ to partition $C_i$ for each $i < k$ into $C_i = W_i \cup R_i$, where $W_i <_P R_i$, $W_i$ is well-founded, and $R_i$ is reverse well-founded.

By $\rt^1_{<\infty}$, either $W_i$ is infinite for some $i < k$ or $R_i$ is infinite for some $i < k$.  Without loss of generality, we may assume that some $W_i$ is infinite because otherwise we could work with the reversed partial order $(P, >_P)$ instead.  Relabel the chains $C_0, \dots, C_{k-1}$ so that the infinite chains $W_i$ are exactly $W_0, \dots, W_{u-1}$ for some $u$ with $0 < u \leq k$.  For each $i < u$, let $\wh{W}_i = \{p \in W_i : \forall n \; \exists q > n \; (q >_P p \;\andd\; q \in W_i)\}$ be the elements of $W_i$ with infinitely many successors in $W_i$.  The set $W_i \setminus \wh{W}_i$ is finite for each $i < u$ because if $W_i \setminus \wh{W}_i$ were infinite, then it would be a chain in $W_i$ of order-type $\omega^*$, which contradicts that $W_i$ is well-founded.  It follows that $\wh{W}_i$ is infinite, well-founded, and has no maximum element for each $i < u$.  Therefore, for each $i < u$ we can define an infinite ascending sequence $A_i$ in $\wh{W}_i$ that is cofinal in $\wh{W}_i$.

We assume that $P$ does not contain a $(0,\infty)$-homogeneous chain, so no tail of $A_i$ is $(0,\infty)$-homogeneous for any $i < u$.  The argument of Lemma~\ref{lem-CXChain} generalizes to simultaneously handle any sequence of infinite ascending sequences.  Apply this to the sequences $A_0, \dots, A_{u-1}$ to obtain infinite ascending sequences $B_0, \dots, B_{u-1}$ and a function $h \colon u \imp k$ such that for each $i < u$, $B_i$ is a counterexample sequence to $A_i$ that is contained in chain $C_{h(i)}$.  Notice that $B_i \cap R_{h(i)} = \emptyset$ for each $i < u$ because $B_i$ is an ascending sequence and $R_{h(i)}$ is reverse well-founded.  Thus $B_i \subseteq W_{h(i)}$, which means that $W_{h(i)}$ is infinite and therefore that $h(i) < u$.  Thus $h$ is in fact a function $h \colon u \imp u$.  Furthermore, $B_i \subseteq \wh{W}_{h(i)}$ because every element of $B_i$ has infinitely many successors in $W_{h(i)}$.

Now observe that $B_i \pwb A_{h(i)}$ for each $i < u$ because $B_i \subseteq \wh{W}_{h(i)}$ and $A_{h(i)}$ is cofinal in $\wh{W}_{h(i)}$.  Additionally, $A_{h(i)} \pwb B_{h(i)}$ for each $i < u$ because $B_{h(i)}$ is a counterexample sequence for $A_{h(i)}$.  Therefore $B_i \pwb A_{h(i)} \pwb B_{h(i)}$, so $B_i \pwb B_{h(i)}$ for each $i < u$ by transitivity.  Let $h^n$ denote the $n$\textsuperscript{th} iterate of $h$, where $h^0(i) = i$ and $h^{n+1}(i) = h(h^n(i))$ for each $i < u$.  By induction, we obtain that $B_{h^m(i)} \pwb B_{h^n(i)}$ for all $i < u$ whenever $m \leq n$.  By the pigeonhole principle, there are $m < n \leq u$ with $h^m(0) = h^n(0)$.  We then have that
\begin{align*}
B_{h^{m+1}(0)} \pwb B_{h^n(0)} = B_{h^m(0)} \pwb A_{h^{m+1}(0)}
\end{align*}
and therefore that $B_{h^{m+1}(0)} \pwb A_{h^{m+1}(0)}$ by transitivity.  This contradicts that $B_{h^{m+1}(0)}$ is a counterexample sequence for $A_{h^{m+1}(0)}$, which completes the proof.
\end{proof}

If we want a $(0,\cof)$-homogeneous chain rather than a $(0,\infty)$-homogeneous chain in a given partial order $(P, <_P)$ of finite width, then we may no longer assume that $P$ contains no suborder of type $\zeta$, and we may no longer apply Lemma~\ref{lem-ChainSplitting} item~\ref{it-ChainSplit} to partition each chain $C_i$ into well-founded and reverse well-founded parts.  Instead, we may directly define the reverse ill-founded part of each $C_i$ and then proceed as before.  This pushes the complexity up to $\pica$ because defining the set of reverse ill-founded elements in a linear order requires $\pica$ in general, as shown in Theorem~\ref{thm-MaxLess}.

\begin{Theorem}\label{thm-cofRSpoInPica}
$\pica \vdash \cofrspo_{<\infty}$.
\end{Theorem}

\begin{proof}
It suffices to show that $\pica \vdash \cofrspocd_{<\infty}$ because $\cofrspo_{<\infty}$ and $\cofrspocd_{<\infty}$ are equivalent as explained in Section~\ref{sec-Dilworth}.

Let $(P, <_P)$ be an infinite partial order with $k$-chain decomposition $C_0, \dots, C_{k-1}$ for some $k$.  Apply $\rt^1_{<\infty}$ to conclude that $C_i$ is infinite for some $i < k$, and apply $\ads$ to $C_i$ to conclude that $C_i$ contains either an infinite ascending sequence or an infinite descending sequence.  By reversing the partial order if necessary, we may assume that $C_i$ contains an infinite ascending sequence.  We show that $P$ contains a $(0,\infty)$-homogeneous chain of order-type $\omega$, which is necessarily $(0,\cof)$-homogeneous as discussed following Definition~\ref{def-ChainHom}.  Thus assume for a contradiction that $P$ does not contain a $(0,\infty)$-homogeneous chain of order-type $\omega$.

Use $\pica$ (and the fact that the $\Sigma^1_1$ sets are the complements of the $\Pi^1_1$ sets) to simultaneously define for each $i < k$ the subset $W_i = \{p \in C_i : \text{$p\ua \cap C_i$ is reverse ill-founded}\}$ of $C_i$ consisting of the elements of $C_i$ that have infinite ascending sequences above them in $C_i$.  At least one $W_i$ is non-empty because at least one $C_i$ contains an infinite ascending sequence.  Now proceed as in the proof of Theorem~\ref{thm-RSpoInACA}, except notice that now the sets $W_i$ are not necessarily well-founded.  What matters is that $W_i$ has no maximum element and that $C_i \setminus W_i$ is reverse well-founded for each $i < k$.  Relabel the chains $C_0, \dots, C_{k-1}$ so that the non-empty $W_i$ are exactly $W_0, \dots, W_{u-1}$ for some $u$ with $0 < u \leq k$.  Each $W_i$ for $i < u$ is non-empty and has no maximum element.  Therefore, for each $i < u$ we can define an infinite ascending sequence $A_i$ in $W_i$ that is cofinal in $W_i$.  No tail of $A_i$ is $(0,\infty)$-homogeneous for any $i < u$ by the assumption that $P$ does not contain a $(0,\infty)$-homogeneous chain of order-type $\omega$.  As in the proof of Theorem~\ref{thm-RSpoInACA}, apply Lemma~\ref{lem-CXChain} to the sequences $A_0, \dots, A_{u-1}$ to obtain infinite ascending sequences $B_0, \dots, B_{u-1}$ and a function $h \colon u \imp k$ such that for each $i < u$, $B_i$ is a counterexample sequence to $A_i$ that is contained in chain $C_{h(i)}$.  The sequence $B_i$ is contained in $W_{h(i)}$ for each $i < u$ by the definition of $W_{h(i)}$.  Thus $W_{h(i)}$ is infinite for each $i < u$, so $h(i) < u$ for each $i < u$.  Therefore $h$ is a function $h \colon u \imp u$.  As in the proof of Theorem~\ref{thm-RSpoInACA}, there are $m < n \leq u$ with $h^m(0) = h^n(0)$, which yields that $B_{h^{m+1}(0)} \pwb A_{h^{m+1}(0)}$, which contradicts that $B_{h^{m+1}(0)}$ is a counterexample sequence for $A_{h^{m+1}(0)}$.  Ultimately, we contradicted the assumption that $P$ does not contain a $(0,\infty)$-homogeneous chain of order-type $\omega$.  Thus $P$ does contain a $(0,\infty)$-homogeneous chain of order-type $\omega$, and such a chain is $(0,\cof)$-homogeneous.
\end{proof}

\section{Proofs of \texorpdfstring{$\rspo_{<\infty}$}{RSpo\_(< infinity)}, \texorpdfstring{$\rspo_k$}{RSpo\_k}, and \texorpdfstring{$\cofrspocd_2$}{(0,cof)-RSpoCD\_2} in \texorpdfstring{$\rca + \isig{2} + \ads$}{RCA\_0 + ISigma\_2 + ADS} and below}\label{sec-forward}

The goal of this section is to show the following.
\begin{itemize}
\item $\rca + \isig{2} + \ads \vdash \rspo_{<\infty}$ (Theorem~\ref{thm-ForwardRSpo}).

\medskip

\item $\rca + \ads \vdash \rspo_k$ for each fixed $k$ (Theorem~\ref{thm-ForwardRSpok}).

\medskip

\item $\rca + \sads \vdash \rspocd_2$ (Theorem~\ref{thm-ForwardRSpoCD2}).

\medskip

\item $\rca + \isig{2} + \ads \vdash \cofrspo_2$ (Theorem~\ref{thm-cofRSpo2InADSISig2}).
\end{itemize}

The main tools used in the proof of $\rspo_{<\infty}$ in $\aca$ from Theorem~\ref{thm-RSpoInACA} are the counterexamples and counterexample sequences of Definition~\ref{def-counterexample}.  Let $A = \la a_n : n \in \Nb \ra$ be an infinite ascending sequence in a partial order $(P, <_P)$ such that no tail of $A$ is $(0,\infty)$-homogeneous.  Given a $p \in P$, $p$ being a counterexample to a given tail of $A$ is a $\Pi^0_1$ property, thus when working strictly below $\aca$ we may not necessarily be able to produce a counterexample sequence for $A$ as in Lemma~\ref{lem-CXChain}.  However, given an $a_m$, we can effectively search for an $a_n$ and a $p \in P$ with $p \geq_P a_m$ and $p \mid_P a_n$.  This search procedure can be used to produce \emph{ladders} for $A$ according to the following definition.  These ladders play the role of the counterexample sequences.

\begin{Definition}\label{def-ladder}
Let $(P, <_P)$ be a partial order, and let $S = \la s_n : n \in \Nb \ra$ be an infinite sequence in $P$.  An infinite ascending sequence $B = \la b_n : n \in \Nb \ra$ is called a \emph{ladder for $S$} if $\forall n \, (s_n \leq_P b_n)$.
\end{Definition}

In Definition~\ref{def-ladder}, the sequence $S$ is not required to be ascending, but in practice it usually is.  Notice also that if $A$ is an infinite ascending sequence in a partial order $(P, <_P)$, then $A$ is a ladder for itself.

Let $(P, <_P)$ be a partial order that has been decomposed into $k$ chains as $P = C_0 \cup \cdots \cup C_{k-1}$, and let $\vec{P}$ denote $\vec{P} = P \oplus {<_P} \oplus C_0 \oplus \cdots \oplus C_{k-1}$.  Suppose that $A$ is an infinite ascending sequence that is contained in some chain $C_j$.  To produce a $(0,\infty)$-homogeneous chain for $P$, we look for ladders for $A$ in the chains $C_i$ with $i \neq j$.  To do this, we define a Turing functional $\findlad^{\vec{P}} \colon P^\Nb \times k \imp P^\Nb$ relative to $\vec{P}$, where $\findlad^{\vec{P}}(A, i)$ attempts to compute a ladder for $A$ in $C_i$.  We then use $\isig{2}$ in the form of bounded $\Pi^0_2$ comprehension to determine the $i < k$ with $i \neq j$ for which $\findlad^{\vec{P}}(A, i)$ is total.  If $\findlad^{\vec{P}}(A, i)$ is total, we then want to find ladders for it in the chains $C_\ell$ with $\ell \neq i$ and continue this process either until finding enough ladders to knit together into a $(0,\infty)$-homogeneous chain or until realizing that there are so few ladders that a tail of one of them must already be $(0,\infty)$-homogeneous.  In fact, we consider all the iterations of $\findlad^{\vec{P}}$ that we may eventually need up front and apply bounded $\Pi^0_2$ comprehension only once.

Lemma~\ref{lem-CtrExPts} says that if $(P, <_P)$ is a partial order decomposed into chains $C_0, \dots, C_{k-1}$ and $A$ is an infinite ascending sequence in $P$ with no $(0,\infty)$-homogeneous tail, then it is possible to search for a ladder for $A$ in some $C_i$.

\begin{Lemma}\label{lem-CtrExPts}
The following is provable in $\rca + \bsig{2}$. Let $(P, <_P)$ be an infinite partial order with $k$-chain decomposition $C_0, \dots, C_{k-1}$, and let $A = \la a_n : n \in \Nb \ra$ be an infinite ascending sequence in $P$.  If no tail of $A$ is $(0,\infty)$-homogeneous, then there is an $i < k$ such that
\begin{align*}
\forall m \; \exists n \; \exists p \in C_i \; (p \geq_P a_m \;\andd\; p \mid_P a_n).
\end{align*}
\end{Lemma}

\begin{proof}
Suppose that no tail $A_{\geq m}$ of $A$ is $(0,\infty)$-homogeneous.  Then for every $m$, there is a counterexample $p$ to the tail $A_{\geq m}$.  This implies that
\begin{align*}
\forall m \; \exists n \; \exists p \in P \; (p \geq_P a_m \;\andd\; p \mid_P a_n).
\end{align*}
Thus we may define a function $f \colon \Nb \imp k$ as follows.  Given $m$, search for an $n$ and a $p$ with $p \geq_P a_m$ and $p \mid_P a_n$, find the $i < k$ with $p \in C_i$, and output $f(m) = i$.  By $\rt^1_{<\infty}$, there is an $i < k$ such that $f(m) = i$ for infinitely many $m$.  This $i$ satisfies the conclusion of the lemma.
\end{proof}

Again let $(P, <_P)$ be a partial order with $k$-chain decomposition $C_0, \dots, C_{k-1}$.  Lemma~\ref{lem-LadderFinder} shows how to search for a ladder for a given infinite sequence $A$ within a target chain $C_i$.  It may be thought of as a weaker, yet effective, version of Lemma~\ref{lem-CXChain}.

\begin{Lemma}\label{lem-LadderFinder}
The following is provable in $\rca$.  Let $(P, <_P)$ be a partial order with $k$-chain decomposition $C_0, \dots, C_{k-1}$. Let $\vec{P}$ denote $\vec{P} = P \oplus {<_P} \oplus C_0 \oplus \cdots \oplus C_{k-1}$.  Then there is a Turing functional $\findlad^{\vec{P}} \colon P^\Nb \times k \imp P^\Nb$ relative to $\vec{P}$ with the following properties for every infinite sequence $A = \la a_n : n \in \Nb \ra$ in $P$ and every $i < k$.
\begin{enumerate}[(1)]
\item\label{it-IsTot} If
\begin{align*}
\forall m \; \exists n \; \exists p \in C_i \; \bigl(p \geq_P a_m \;\andd\; p \mid_P a_n\bigr),
\end{align*}
then $\findlad^{\vec{P}}(A, i)$ is total.

\medskip

\item\label{it-TotLadder} If $\findlad^{\vec{P}}(A, i)$ is total, then it computes a ladder for $A$ in $C_i$.
\end{enumerate}
\end{Lemma}

\begin{proof}
Let $A = \la a_n : n \in \Nb \ra$ be an infinite sequence in $P$, and let $i < k$.  Compute $\findlad^{\vec{P}}(A, i)(0)$ by searching for an $n$ and a $p_0 \in C_i$ with $p_0 \geq_P a_0$ and $p_0 \mid_P a_n$.  If $p_0$ is found, then output $\findlad^{\vec{P}}(A, i)(0) = p_0$.  Compute $\findlad^{\vec{P}}(A, i)(m+1)$ by first computing $p_m = \findlad^{\vec{P}}(A, i)(m)$.  Then search for an $n$ and a $p_{m+1} \in C_i$ with $p_{m+1} >_P p_m$, $p_{m+1} \geq_P a_{m+1}$, and $p_{m+1} \mid_P a_n$.  If $p_{m+1}$ is found, then output $\findlad^{\vec{P}}(A, i)(m+1) = p_{m+1}$.

Item~\ref{it-TotLadder} follows immediately from the definition of $\findlad^{\vec{P}}$.  If $A = \la a_n : n \in \Nb \ra$ is an infinite sequence in $P$, if $i < k$, and if $\findlad^{\vec{P}}(A, i)$ is total, then it must be that $\findlad^{\vec{P}}(A, i)(m) \in C_i$, that $a_m \leq_P \findlad^{\vec{P}}(A, i)(m)$, and that $\findlad^{\vec{P}}(A, i)(m) <_P \findlad^{\vec{P}}(A, i)(m+1)$ for every $m$.

For item~\ref{it-IsTot}, let $A = \la a_n : n \in \Nb \ra$ be an infinite sequence in $P$, let $i < k$, and suppose that for every $m$ there are an $n$ and a $p \in C_i$ with $p \geq_P a_m$ and $p \mid_P a_n$.  We use $\isig{1}$ to show that $\findlad^{\vec{P}}(A, i)$ is total.  By assumption, there are an $n$ and a $p_0 \in C_i$ with $p_0 \geq_P a_0$ and $p_0 \mid_P a_n$.  Thus $\findlad^{\vec{P}}(A, i)(0)$ is defined.  Inductively assume that $p_m = \findlad^{\vec{P}}(A, i)(m)$ is defined.  Then there is an $\ell$ such that $p_m \mid_P a_\ell$.  By assumption, there are an $s$ and a $p \in C_i$ with
\begin{align*}
p \geq_P a_{m+1} \quad\text{and}\quad p \mid_P a_s
\end{align*}
and also a $t$ and a $q \in C_i$ with
\begin{align*}
q \geq_P a_\ell \quad\text{and}\quad q \mid_P a_t.
\end{align*}
The elements $p$ and $q$ are both in the chain $C_i$, so they are comparable.  By taking $p_{m+1} = \max_{<_P}\{p, q\}$ and $n$ as either $s$ or $t$ as appropriate, we obtain an $n$ and a $p_{m+1} \in C_i$ with
\begin{align*}
p_{m+1} \geq_P a_{m+1} \quad\text{and}\quad p_{m+1} \geq_P a_\ell \quad\text{and}\quad p_{m+1} \mid_P a_n.
\end{align*}
Again, $p_{m+1}$ and $p_m$ are both in the chain $C_i$, so they are comparable.  However, we cannot have that $p_{m+1} \leq_P p_m$ because this would imply that $a_\ell \leq_P p_{m+1} \leq_P p_m$, which contradicts that $p_m \mid_P a_\ell$.  Therefore $p_m <_P p_{m+1}$.  Thus there are an $n$ and a $p_{m+1} \in C_i$ with $p_{m+1} >_P p_m$, $p_{m+1} \geq_P a_{m+1}$, and $p_{m+1} \mid_P a_n$.  So $\findlad^{\vec{P}}(A, i)(m+1)$ is defined.
\end{proof}

The following lemma concerning finite labeled trees helps organize the proof of $\rspo_{<\infty}$ in $\rca + \isig{2} + \ads$.  Recall that for a finite rooted tree $T$, the \emph{height} of a vertex is the length of the path from the vertex to the root, the \emph{height} of the tree is the maximum height of a vertex in $T$, and \emph{level $k$} of $T$ consists of all the vertices of $T$ that have height $k$.  (We warn the reader that the height of a tree as defined here is one less than the height of the tree according to Definition~\ref{def-HandW} when considering the tree as a partial order.  For example, the one-element tree consisting of only the root has height $0$ as a tree but height $1$ as a partial order.)  For a rooted tree $T$, let $\preceq$ denote the associated \emph{tree-order} on $T$, where $\sigma \preceq \tau$ if $\sigma$ is on the (unique) path from the root to $\tau$.  Let $\prec$ denote the strict version of $\preceq$.

\begin{Lemma}\label{lem-TreeHelper}
The following is provable in $\rca$.
Let $k \geq 2$, and let $T$ be a finite rooted tree with the following properties.
\begin{itemize}
\item $T$ has height $k$.

\medskip

\item Every leaf of $T$ is at level $k$.

\medskip

\item The vertices of $T$ are labeled by a function $\ell \colon T \imp k$ in such a way that if $\sigma \in T$ is not a leaf and $\tau_0, \dots, \tau_{n-1}$ are the children of $\sigma$, then $\ell(\tau_0), \dots, \ell(\tau_{n-1})$ are distinct elements of $k \setminus \{\ell(\sigma)\}$.
\end{itemize}
Then there is a $\sigma \in T$ that is not a leaf such that for every child $\tau$ of $\sigma$ there is an $\eta \succ \tau$ with $\ell(\eta) = \ell(\sigma)$.
\end{Lemma}

\begin{proof}
Proceed by $\Pi^0_1$ induction on $k \geq 2$.  For the base case $k = 2$, the only possibility for $T$ is a path of length $2$ consisting of the root $r$, $r$'s child $\tau$, and $\tau$'s child $\eta$, where $\ell(r) \neq \ell(\tau)$ and $\ell(\tau) \neq \ell(\eta)$.  As $k = 2$, it must be that $\ell(r) = \ell(\eta)$, so we may take $\sigma = r$.

Now suppose that the lemma holds for $k$.  Consider a tree $T$ of height $k+1$ where every leaf is at level $k+1$, and consider also a labeling $\ell \colon T \imp k+1$ that labels $T$ according to the hypothesis of the lemma.  Let $r$ be the root, and let $\tau_0, \dots, \tau_{n-1}$ be $r$'s children.  If for every $\tau_i$ there is an $\eta \succ \tau_i$ with $\ell(\eta) = \ell(r)$, then we may take $\sigma = r$.  Otherwise, there is a $\tau_i$ such that no $\eta \succ \tau_i$ has $\ell(\eta) = \ell(r)$.  Let $S = \{\eta \in T : \eta \succeq \tau_i\}$ be the complete subtree of $T$ above $\tau_i$ rooted at $\tau_i$.  Then $S$ has height $k$, and every leaf of $S$ is at level $k$ of $S$.  Moreover, $S$ is labeled by $\ell$, which only uses labels from the set $(k+1) \setminus \{\ell(r)\}$.  Fix a bijection $f \colon (k+1) \setminus \{\ell(r)\} \imp k$, and define the labeling $\wh{\ell} \colon S \imp k$ by $\wh{\ell} = f \circ \ell$.  By the induction hypothesis applied to $S$ and $\wh{\ell}$, there is a $\sigma \in S$ that is not a leaf such that for every child $\tau$ of $\sigma$ in $S$ there is an $\eta \succ \tau$ in $S$ with $\wh{\ell}(\eta) = \wh{\ell}(\sigma)$.  This $\sigma$ also satisfies the conclusion of the lemma for $T$ because $S$ is the complete subtree of $T$ above $\tau_i$:  for every child $\tau$ of $\sigma$ in $T$, there is an $\eta \succ \tau$ in $T$ with $\wh{\ell}(\eta) = \wh{\ell}(\sigma)$ and hence with $\ell(\eta) = \ell(\sigma)$.
\end{proof}

\begin{Theorem}\label{thm-ForwardRSpo}
$\rca + \isig{2} + \ads \vdash \rspo_{<\infty}$.
\end{Theorem}

\begin{proof}
It suffices to show that $\rca + \isig{2} + \ads \vdash \rspocd_{<\infty}$ because $\rspo_{<\infty}$ and $\rspocd_{<\infty}$ are equivalent over $\rca$ as explained in Section~\ref{sec-Dilworth}.  Thus let $(P, <_P)$ be an infinite partial order with $k$-chain decomposition $C_0, \dots, C_{k-1}$ for some $k$.  We may assume that $k \geq 2$ because if $P$ is a chain, then $P$ itself is $(0,\infty)$-homogeneous.

Some $C_i$ is infinite by $\rt^1_{<\infty}$ and therefore contains either an infinite ascending sequence or an infinite descending sequence by $\ads$.  By relabeling the chains and by reversing the partial order if necessary, we may assume that there is an infinite ascending sequence $A = \la a_n : n \in \Nb \ra$ contained in chain $C_0$.  Let $\vec{P} = P \oplus {<_P} \oplus C_0 \oplus \cdots \oplus C_{k-1}$, and let $\findlad^{\vec{P}} \colon P^\Nb \times k \imp P^\Nb$ be the Turing functional relative to $\vec{P}$ from Lemma~\ref{lem-LadderFinder}.

Let $R = (k-1)^{\leq k}$ be the complete $(k-1)$-ary tree of height $k$.  Assign a label $\ell(\sigma) \in \{0, \dots, k-1\}$ to each $\sigma \in R$ as follows.  First, label the root $\emptyset \in R$ with $\ell(\emptyset) = 0$.  Now suppose that $\sigma \in R$ is not a leaf and has been labeled $\ell(\sigma)$.  Index the $k-1$ children of $\sigma$ as $\la \tau_i : i \in k \setminus \{\ell(\sigma)\} \ra$, and label $\ell(\tau_i) = i$ for each $i \in k \setminus \{\ell(\sigma)\}$.

For a $\sigma \in \Nb^{<\Nb}$ with $|\sigma| \geq 1$, let $\sigma^- = \sigma \rst (|\sigma|-1)$ denote $\sigma$ with the last term cut off.  For $\sigma \in R$, let $\findlad^{\vec{P},\sigma}$ denote the iteration of $\findlad^{\vec{P}}$ given by
\begin{align*}
\findlad^{\vec{P}, \emptyset}(A) &= A\\
\findlad^{\vec{P}, \sigma}(A) &= \findlad^{\vec{P}}(\findlad^{\vec{P}, \sigma^-}(A), \ell(\sigma)) & \text{if $|\sigma| \geq 1$}.
\end{align*}
So if $\sigma \in R$ has $|\sigma| \geq 1$, then $\findlad^{\vec{P}, \sigma}(A)$ is
\begin{align*}
\findlad^{\vec{P}}\bigl(\cdots \findlad^{\vec{P}}\bigl(\findlad^{\vec{P}}(A, \ell(\sigma \rst 1)), \ell(\sigma \rst 2)\bigr) \cdots, \ell(\sigma)\bigr).
\end{align*}

Now use $\isig{2}$ in the form of bounded $\Pi^0_2$ comprehension to form the subtree $T \subseteq R$ given by
\begin{align*}
T = \bigl\{\sigma \in R : \text{$\findlad^{\vec{P}, \sigma}(A)$ is total}\bigr\}.
\end{align*}
Notice that if $\sigma$ is in $T$, then $\findlad^{\vec{P}, \sigma}(A)$ computes an infinite ascending sequence in $C_{\ell(\sigma)}$ by Lemma~\ref{lem-LadderFinder} item~\ref{it-TotLadder} and the fact that $A$ is an infinite ascending sequence in $C_0 = C_{\ell(\emptyset)}$.

There are now two cases.  The first case is that $T$ has a leaf $\sigma$ at some level $<\! k$.  Let $\la \tau_i : i \in k \setminus \{\ell(\sigma)\} \ra$ again denote the indexing of $\sigma$'s children in $R$.  As $\sigma$ is in $T$, $\findlad^{\vec{P}, \sigma}(A)$ computes the infinite ascending sequence $B = \la b_n : n \in \Nb \ra$ in $C_{\ell(\sigma)}$ given by $b_n = \findlad^{\vec{P}, \sigma}(A)(n)$ for all $n$.  If no tail of $B$ is $(0,\infty)$-homogeneous, then by Lemma~\ref{lem-CtrExPts} there is an $i < k$ such that
\begin{align*}
\forall m \; \exists n \; \exists p \in C_i \; \bigl(p \geq_P b_m \;\andd\; p \mid_P b_n \bigr).
\end{align*}
It cannot be that $i = \ell(\sigma)$, so there must be such an $i \in k \setminus \{\ell(\sigma)\}$.  Thus, by Lemma~\ref{lem-LadderFinder} item~\ref{it-IsTot},
\begin{align*}
\findlad^{\vec{P}, \tau_i}(A) = \findlad^{\vec{P}}(\findlad^{\vec{P}, \sigma}(A), \ell(\tau_i)) = \findlad^{\vec{P}}(B, i)
\end{align*}
is total.  This means that $\tau_i \in T$, which contradicts that $\sigma$ is a leaf of $T$.  Therefore some tail of $B$ is indeed $(0,\infty)$-homogeneous.  We may thin this tail so that its range exists as a set, thereby producing a $(0,\infty)$-homogeneous chain in $P$.

The second case is that every leaf of $T$ is at level $k$.  Then $T$ and $\ell$ satisfy the hypotheses of Lemma~\ref{lem-TreeHelper}.  Let $\sigma$ be as in the conclusion of Lemma~\ref{lem-TreeHelper} for $T$ and $\ell$.  For each child $\tau$ of $\sigma$ in $T$, let $\eta_\tau \in T$ be such that $\eta_\tau \succ \tau$ and $\ell(\eta_\tau) = \ell(\sigma)$.  As $\sigma \in T$, $\findlad^{\vec{P}, \sigma}(A)$ computes the infinite ascending sequence $X = \la x_n : n \in \Nb \ra$ in $C_{\ell(\sigma)}$ given by $x_n = \findlad^{\vec{P}, \sigma}(A)(n)$ for all $n$.  Likewise, $\findlad^{\vec{P}, \eta_\tau}(A)$ also computes an infinite ascending sequence in $C_{\ell(\eta_\tau)} = C_{\ell(\sigma)}$ for each child $\tau$ of $\sigma$ in $T$.  Define the infinite sequence $Y = \la y_n : n \in \Nb \ra$ by setting
\begin{align*}
y_n = \textstyle{\max_{<_P}}\left\{\findlad^{\vec{P}, \eta_\tau}(A)(n) : \text{$\tau$ is a child of $\sigma$ in $T$}\right\}
\end{align*}
for each $n$.  Then $Y$ is an infinite ascending sequence in $C_{\ell(\sigma)}$ that is cofinal in the union of the sequences computed by the $\findlad^{\vec{P}, \eta_\tau}(A)$ for the children $\tau$ of $\sigma$ in $T$.

\begin{ClaimADSimpRSpo}\label{claim-GoodChild}
Let $\tau$ be a child of $\sigma$ in $T$.  Then every $p \in C_{\ell(\tau)}$ is either above every element of $X$ or below some element of $Y$ (and hence below almost every element of $Y$).
\end{ClaimADSimpRSpo}

\begin{proof}[Proof of Claim]
We have that $\findlad^{\vec{P}, \sigma}(A)$ computes the the infinite ascending sequence $X$ and that
\begin{align*}
\findlad^{\vec{P}, \tau}(A) = \findlad^{\vec{P}}(\findlad^{\vec{P}, \sigma}(A), \ell(\tau)) = \findlad^{\vec{P}}(X, \ell(\tau))
\end{align*}
is total, so $\findlad^{\vec{P}, \tau}(A)$ computes a ladder $Z$ for $X$ in $C_{\ell(\tau)}$.  Consider any $p \in C_{\ell(\tau)}$ and its location with respect to the elements of $Z$.  If $p$ is above every element of $Z$, then $p$ is above every element of $X$ because $Z$ is a ladder for $X$.  Suppose instead that $p$ is below some element $\findlad^{\vec{P}, \tau}(A)(n)$ of $Z$.  Consider now the path $\tau = \alpha_0 \prec \alpha_1 \prec \cdots \prec \alpha_{m-1} = \eta_\tau$ from $\tau$ to $\eta_\tau$ in $T$.  For each $i < m$, $\findlad^{\vec{P}, \alpha_i}(A)$ is total because $\alpha_i \in T$.  Thus $\findlad^{\vec{P}, \alpha_{i+1}}(A)$ computes a ladder for $\findlad^{\vec{P}, \alpha_i}(A)$ for each $i < m-1$.  Therefore
\begin{align*}
p &\leq_P \findlad^{\vec{P}, \tau}(A)(n) \leq_P \findlad^{\vec{P}, \alpha_1}(A)(n) \leq_P \cdots \leq_P \findlad^{\vec{P}, \alpha_{m-2}}(A)(n)\\
&\leq_P \findlad^{\vec{P}, \eta_\tau}(A)(n) \leq_P y_n.
\end{align*}
So $p$ is below an element of $Y$.
\end{proof}

\begin{ClaimADSimpRSpo}\label{claim-BadChild}
There is an $m$ such that whenever $i \in k \setminus \bigl(\{\ell(\sigma)\} \cup \{\ell(\tau) : \text{$\tau$ is a child of $\sigma$ in $T$}\}\bigr)$ and $p \in C_i$, then either $p$ is comparable with almost every element of $X_{\geq m}$, or $p$ is incomparable with every element of $X_{\geq m}$.
\end{ClaimADSimpRSpo}

\begin{proof}[Proof of Claim]
Let $I$ denote the finite set
\begin{align*}
I = k \setminus \bigl(\{\ell(\sigma)\} \cup \{\ell(\tau) : \text{$\tau$ is a child of $\sigma$ in $T$}\}\bigr),
\end{align*}
let $i \in I$, let $\tau$ be the child of $\sigma$ in $R$ with $\ell(\tau) = i$, and notice that $\tau \notin T$.  We have that $\findlad^{\vec{P}, \sigma}(A)$ computes the the infinite ascending sequence $X$, so 
\begin{align*}
\findlad^{\vec{P}, \tau}(A) = \findlad^{\vec{P}}(\findlad^{\vec{P}, \sigma}(A), \ell(\tau)) = \findlad^{\vec{P}}(X, \ell(\tau))
\end{align*}
must not be total because otherwise $\tau$ would be in $T$.  Therefore by Lemma~\ref{lem-LadderFinder} item~\ref{it-IsTot},
\begin{align*}
\exists m \; \forall n \; \forall p \in C_i \; \bigl(p \ngeq_P x_m \;\orr\; p \comp_P x_n \bigr).
\end{align*}
By applying $\bsig{2}$, we obtain a fixed $m$ such that $\forall i \in I \; \forall n \; \forall p \in C_i \; \left(p \ngeq_P x_m \;\orr\; p \comp_P x_n \right)$.  We show that this $m$ satisfies the claim.

Let $p \in C_i$ for an $i \in I$.  If $p$ is below an element of $X_{\geq m}$, then $p$ is below almost every element of $X_{\geq m}$ because $X_{\geq m}$ is an ascending sequence.  If $p$ is above an element of $X_{\geq m}$, then $p \geq_P x_m$, in which case $p$ is comparable with every element of $X_{\geq m}$.  So if $p$ is comparable with some element of $X_{\geq m}$, then $p$ is comparable with almost every element of $X_{\geq m}$.  That is, either $p$ is comparable with almost every element of $X_{\geq m}$, or $p$ is incomparable with every element of $X_{\geq m}$.
\end{proof}

We can now assemble a $(0,\infty)$-homogeneous chain $B$ from $X$ and $Y$.  Let $m$ be as in Claim~\ref{claim-BadChild}.  Thin the infinite ascending sequences $X_{\geq m}$ and $Y$ to infinite ascending sequences $\wh{X}_{\geq m}$ and $\wh{Y}$ whose ranges exist as sets.  Notice that $\wh{X}_{\geq m} \cup \wh{Y}$ is a chain because $\wh{X}_{\geq m}$ and $\wh{Y}$ are both contained in the chain $C_{\ell(\sigma)}$.  If $\wh{Y} \pwb \wh{X}_{\geq m}$, then take $B = \wh{X}_{\geq m}$.  Otherwise there is an $n$ such that $y_n$ is above every element of $\wh{X}_{\geq m}$.  In this case, take $B = \wh{X}_{\geq m} \cup \wh{Y}_{\geq n}$.  We show that $B$ is $(0,\infty)$-homogeneous.

Consider any $p \in P$.  If $p \in C_{\ell(\sigma)}$, then $p$ is comparable with every element of $B$ because $B \subseteq C_{\ell(\sigma)}$.  Suppose that $p \in C_i$ for an $i \in \{\ell(\tau) : \text{$\tau$ is a child of $\sigma$ in $T$}\}$.  By Claim~\ref{claim-GoodChild}, either $p$ is above every element of $X$ or below almost every element of $Y$.  In either case, $p$ is comparable with infinitely many elements of $B$.  Finally, suppose that $p \in C_i$ for an $i \in k \setminus \bigl(\{\ell(\sigma)\} \cup \{\ell(\tau) : \text{$\tau$ is a child of $\sigma$ in $T$}\}\bigr)$.  By Claim~\ref{claim-BadChild}, either $p$ is comparable with almost every element of $X_{\geq m}$, or $p$ is incomparable with every element of $X_{\geq m}$.  If $p$ is comparable with almost every element of $X_{\geq m}$, then $p$ is comparable with infinitely many elements of $B$.  Suppose instead that $p$ is incomparable with every element of $X_{\geq m}$.  If we took $B = \wh{X}_{\geq m}$, then $p$ is incomparable with every element of $B$.  Suppose that we took $B = \wh{X}_{\geq m} \cup \wh{Y}_{\geq n}$.  If $p$ is incomparable with every element of $\wh{Y}_{\geq n}$, then $p$ is incomparable with every element of $B$.  If $p$ is below an element of $\wh{Y}_{\geq n}$, then $p$ is below almost every element of $\wh{Y}_{\geq n}$ and therefore is comparable with infinity many elements of $B$.  If $p$ is above an element of $\wh{Y}_{\geq n}$, then $p$ is above every element of $\wh{X}_{\geq m}$ and therefore is comparable with infinitely many elements of $B$.  This completes the proof that $B$ is $(0,\infty)$-homogeneous for $P$.
\end{proof}

\begin{Theorem}\label{thm-ForwardRSpok}
For each fixed standard $k$, $\rca + \ads \vdash \rspo_k$.
\end{Theorem}

\begin{proof}
The proof is essentially the same as that of Theorem~\ref{thm-ForwardRSpo}.  In the proof of Theorem~\ref{thm-ForwardRSpo}, the sole use of $\isig{2}$ is the application of bounded $\Pi^0_2$ comprehension to form the subtree $T = \bigl\{\sigma \in R : \text{$\findlad^{\vec{P}}(A, \sigma)$ is total}\bigr\}$ of the tree $R = (k-1)^{\leq k}$.  In the case where $k$ is fixed and standard, $R$ and its elements have fixed standard codes, so we may instead form $T$ by a giant case analysis, using excluded middle for the predicates $\bigl(\text{$\findlad^{\vec{P}}(A, \sigma)$ is total}\bigr) \;\orr\; \bigl(\text{$\findlad^{\vec{P}}(A, \sigma)$ is not total}\bigr)$ for the $\sigma \in R$.  The proof then continues exactly as in that of Theorem~\ref{thm-ForwardRSpo}.  The proof of Theorem~\ref{thm-ForwardRSpo} does make use of $\bsig{2}$ in addition to the aforementioned use of $\isig{2}$, but the relevant instances of $\bsig{2}$ are provable in $\rca$ when $k$ is fixed and standard (plus $\bsig{2}$ is available here anyway because $\rca + \ads \vdash \bsig{2}$ as explained in Section~\ref{sec-background}).
\end{proof}

When working under the assumption that the partial order $(P, <_P)$ contains an infinite ascending sequence, the proof of Theorem~\ref{thm-ForwardRSpo} produces either a chain of order-type $\omega$ (when taking $B = \wh{X}_{\geq m})$ or a chain of order-type $\omega + \omega$ (when taking $B = \wh{X}_{\geq m} \cup \wh{Y}_{\geq n}$).  Therefore, $\rca + \isig{2} + \ads$ in fact proves that every infinite partial order of finite width contains a $(0,\infty)$-homogeneous chain of order-type $\omega$, $\omega + \omega$, $\omega^*$, or $\omega^* + \omega^*$.  Likewise, for each fixed standard $k$, $\rca + \ads$ proves that every infinite partial order of width $\leq\! k$ contains a $(0,\infty)$-homogeneous chain of order-type $\omega$, $\omega + \omega$, $\omega^*$, or $\omega^* + \omega^*$.

In the particular case of $\rspocd_2$, we may weaken $\ads$ to $\sads$ and show that $\rca + \sads \vdash \rspocd_2$.  To do this, we make use of a strict version of $\srt^2_k$ where we assume not only that the coloring $c \colon [\Nb]^2 \imp k$ is stable, but that there is a fixed bound $n$ such that for every $x$ there are at most $n$ many $y > x$ where $c(x, y) \neq c(x, y+1)$.

\begin{Definition}
{\ }
\begin{itemize}
\item A $k$-coloring of pairs $c \colon [\Nb]^2 \imp k$ is \emph{$n$-stable} if, for every $x$, $|\{y > x : c(x,y) \neq c(x, y+1)\}| \leq n$.

\medskip

\item $\ssrt{n}{2}{k}$ is the restriction of $\rt^2_k$ to $n$-stable $k$-colorings of pairs $c$.

\medskip

\item $\ssrt{n}{2}{<\infty}$ denotes $\forall k \, (\ssrt{n}{2}{k})$.
\end{itemize}
\end{Definition}

\begin{Proposition}\label{prop-SSRT}
{\ }
\begin{enumerate}[(1)]
\item\label{it-SSRTRCA} For each fixed standard $n$ and $k$, $\rca \vdash \ssrt{n}{2}{k}$.

\medskip

\item\label{it-SSRTBSig2} For each fixed standard $n$, $\bsig{2}$ and $\ssrt{n}{2}{<\infty}$ are equivalent over $\rca$.

\medskip

\item\label{it-SSRTISig2} $\isig{2}$ and $\forall n \, (\ssrt{n}{2}{<\infty})$ are equivalent over $\rca$.
\end{enumerate}
\end{Proposition}

\begin{proof}
For item~\ref{it-SSRTRCA}, let $c \colon [\Nb]^2 \imp k$ be an $n$-stable $k$-coloring of pairs.  Let $\varphi(x, i, c)$ be the $\Sigma^0_1$ formula expressing that there is a sequence $x < y_0 < y_1 < \dots < y_{i-1}$ of length $i$ where $\forall j < i \; (c(x, y_j) \neq c(x, y_j + 1))$.  Note that $\varphi(x, 0, c)$ holds for every $x$.  By repeated applications of excluded middle, there is a maximum $i \leq n$ such that $\varphi(x, i, c)$ holds for infinitely many $x$.  By the maximality of $i$, there is a bound $b$ such that if $x > b$ and $j > i$, then $\neg \varphi(x, j, c)$.  $\rca$ suffices to formalize the well-known fact that every infinite recursively enumerable set contains an infinite recursive subset.  That is, for every $\Sigma^0_1$ formula $\psi(x)$ (possibly with parameters), $\rca$ proves that if $\psi(x)$ holds for infinitely many $x$, then there is an infinite set $X$ such that $\forall x \, (x \in X \imp \psi(x))$.  Applying this to the $\Sigma^0_1$ formula $(x > b) \andd \varphi(x, i, c)$, we obtain an infinite set $X$ such that $\forall x \, (x \in X \;\imp\; (x > b) \andd \varphi(x, i, c))$.  Then for every $x \in X$, there are exactly $i$ many elements $y > x$ with $c(x, y) \neq c(x, y+1)$.  So for every $x \in X$, there is a unique sequence $x < y_0^x < \cdots < y_{i-1}^x$ where $\forall j < i \; (c(x, y_j^x) \neq c(x, y_j^x + 1))$.  It follows that $c(x, y) = c(x, y_{i-1}^x + 1)$ for all $x \in X$ and $y > y_{i-1}^x$.  Define a $k$-coloring of singletons $f \colon X \imp k$ by $f(x) = c(x, y_{i-1}^x + 1)$.  Apply $\rt^1_k$ to $f$ to obtain a set $G \subseteq X$ that is homogeneous for $f$ for some color $\ell < k$.  Now define an infinite set $H = \{x_s : s \in \Nb\} \subseteq G$ with $x_0 < x_1 < x_2 < \cdots$ by letting each $x_s$ be the least element of $G$ with $x_s \geq \max\{y_{i-1}^{x_0}, y_{i-1}^{x_1}, \dots, y_{i-1}^{x_{s-1}}\} + 1$.  Then $H$ is homogeneous for $c$.  If $s < t$, then $x_t \geq y_{i-1}^{x_s} + 1$, so $c(x_s, x_t) = c(x_s, y_{i-1}^{x_s} + 1) = f(x_s) = \ell$.

For the forward direction of item~\ref{it-SSRTBSig2}, we may use the same proof as for item~\ref{it-SSRTRCA}, except now $k$ is arbitrary so we must use $\rt^1_{<\infty}$ in place of $\rt^1_k$.  For the reversal, it is easy to see that $\rca + \ssrt{0}{2}{<\infty} \vdash \rt^1_{<\infty}$.  Let $f \colon \Nb \imp k$ be a $k$-coloring of singletons for some $k$, and define $c \colon [\Nb]^2 \imp k$ by $c(x,y) = f(x)$.  Then $c$ is a $0$-stable $k$-coloring of pairs, and every $H$ that is homogeneous for $c$ is also homogeneous for $f$.

For the forward direction of item~\ref{it-SSRTISig2}, we may use the same proof as for item~\ref{it-SSRTBSig2}, except now $n$ is also arbitrary, so we must use $\isig{2}$ in the form of the $\Pi^0_2$ least element principle (see Section~\ref{sec-background}) to obtain $i$.  To do this, apply the $\Pi^0_2$ least element principle to obtain the least $j$ such that $\forall s \; \exists x > s \; \varphi(x, n-j, c)$.  Then $i = n-j$ is the greatest $i \leq n$ such that $\varphi(x, i, c)$ holds for infinitely many $x$.  For the reversal, we show that $\rca + \forall n \, (\ssrt{n}{2}{<\infty})$ proves the $\Pi^0_2$ least element principle.  Let $\forall x \, \exists y \, \psi(m, x, y)$ be a $\Pi^0_2$ formula, possibly with undisplayed parameters, where $\psi$ is $\Sigma^0_0$.  Let $k$ be such that $\forall x \, \exists y \, \psi(k, x, y)$.  We want to find the least $i$ such that $\forall x \, \exists y \, \psi(i, x, y)$.  Define a $(k+2)$-coloring of pairs $c \colon [\Nb]^2 \imp k+2$ by
\begin{align*}
c(s,t) =
\begin{cases}
i & \text{if $i \leq k$ is least such that $\forall x \leq s \; \exists y \leq t \; \psi(i, x, y)$}\\
k+1 & \text{if $\forall i \leq k \; \exists x \leq s \; \forall y \leq t \; \neg\psi(i, x, y)$.}
\end{cases}
\end{align*}
The coloring $c$ is $(k+2)$-stable because, for fixed $s$, if $s < t_0 < t_1$, then $c(s, t_0) \geq c(s, t_1)$.  By $\forall n \, (\ssrt{n}{2}{<\infty})$, let $H$ be a set that is homogeneous for $c$ with color $i$.  Then $i$ is least such that $\forall x \, \exists y \, \psi(i, x, y)$.  To see this, we first show that if $j \leq k$ and $\forall x \, \exists y \, \psi(j, x, y)$, then $i \leq j$.  Let $s \in H$, and, by $\bsig{0}$ and the assumption $\forall x \, \exists y \, \psi(j, x, y)$, let $t \in H$ be large enough so that $t > s$ and $\forall x \leq s \; \exists y \leq t \; \psi(j, x, y)$.  Then $i = c(s,t) \leq j$, so $i \leq j$.  It follows that $i \leq k$ by the assumption $\forall x \, \exists y \, \psi(k, x, y)$.  Now we show that $\forall x \, \exists y \, \psi(i, x, y)$.  Consider any $x_0$, and let $s, t \in H$ be such that $x_0 < s < t$.  Then $c(s,t) = i$ and $i \leq k$, so $i$ is least such that $\forall x \leq s \; \exists y \leq t \; \psi(i, x, y)$.  In particular, $\exists y \, \psi(i, x_0, y)$ because $x_0 < s$.  Therefore $\forall x \, \exists y \, \psi(i, x, y)$.  Thus $i$ is least such that $\forall x \, \exists y \, \psi(i, x, y)$.
\end{proof}

Our proof of $\rspocd_2$ in $\rca + \sads$ only uses the principle $\ssrt{2}{2}{3}$ from Proposition~\ref{prop-SSRT} item~\ref{it-SSRTRCA}.  We include items~\ref{it-SSRTBSig2} and~\ref{it-SSRTISig2} for completeness.  For fixed standard $k \geq 2$, one might also consider the principle $\forall n \, (\ssrt{n}{2}{k})$.  That $\rca + \isig{2} \vdash \forall n \, (\ssrt{n}{2}{k})$ follows from Proposition~\ref{prop-SSRT} item~\ref{it-SSRTISig2}, but there cannot be a reversal because $\rca + \srt^2_k \vdash \forall n \, (\ssrt{n}{2}{k})$, but $\rca + \srt^2_k \nvdash \isig{2}$ (see Section~\ref{sec-background}).  We did not determine if $\rca + \bsig{2} \vdash \forall n \, (\ssrt{n}{2}{k})$.

\begin{Theorem}\label{thm-ForwardRSpoCD2}
$\rca + \sads \vdash \rspocd_2$.
\end{Theorem}

\begin{proof}
By inspecting the proofs of Theorems~\ref{thm-ForwardRSpo} and~\ref{thm-ForwardRSpok}, we see that the only use of $\ads$ is to produce either an infinite ascending sequence or an infinite descending sequence in the partial order.  That is, $\rca$ proves the following for each fixed standard $k$.

\begin{itemize}
\item[$(\star)$] Let $(P, <_P)$ be an infinite $k$-chain decomposable partial order that contains either an infinite ascending sequence or an infinite descending sequence.  Then $P$ contains a $(0,\infty)$-homogeneous chain.
\end{itemize}

Let $(P, <_P)$ be an infinite partial order with $2$-chain decomposition $C_0, C_1$.  The plan of the proof is to either produce a $(0,\infty)$-homogeneous chain outright or to apply $\sads$ to obtain an infinite ascending sequence or an infinite descending sequence in $P$, in which case we may produce a $(0,\infty)$-homogeneous chain in $P$ by applying $(\star)$.  Notice that we may make free use of $\bsig{2}$ because $\rca + \sads \vdash \bsig{2}$ as explained in Section~\ref{sec-background}.

The partial order $P$ is infinite, so at least one of $C_0$ and $C_1$ is infinite.  Assume that $C_0$ is infinite for the sake of argument.

\begin{ClaimSADSimpRSpocd2}\label{claim-CofIncop}
Suppose that there are an infinite $D_0 \subseteq C_0$ and a finite $D_1 \subseteq C_1$ such that $D_0 \mid_P (C_1 \setminus D_1)$.  Then $P$ contains a $(0,\infty)$-homogeneous chain.
\end{ClaimSADSimpRSpocd2}

\begin{proof}[Proof of Claim]
Suppose that $D_1 = \{q_0, \dots, q_{n-1}\}$, and define a $2^n$-coloring $f \colon D_0 \imp 2^n$ by $f(p) = \la b_0, \dots, b_{n-1}\ra$, where, for each $i < n$, $b_i = 0$ if $p \mid_P q_i$ and $b_i = 1$ if $p \comp_P q_i$.  Apply $\rt^1_{<\infty}$ to $f$ to obtain an infinite set $H \subseteq D_0$ that is homogeneous for $f$.  Then $H$ is a $(0,\infty)$-homogeneous chain for $P$.  $H$ is a chain because $H \subseteq D_0 \subseteq C_0$ and $C_0$ is a chain.  Consider a $q \in P$.  If $q \in C_0$, then $q$ is comparable with every element of $H$ because $H \subseteq C_0$ and $C_0$ is a chain.  If $q \in C_1 \setminus D_1$, then $q$ is incomparable with every element of $D_0$ by assumption and therefore is incomparable with every element of $H$ because $H \subseteq D_0$.  If $q \in D_1$, then $q = q_i$ for some $i < n$, so by the homogeneity of $H$, either $q$ is comparable with every element of $H$ or $q$ is incomparable with every element of $H$.
\end{proof}

If $C_1$ is finite, then we may apply Claim~\ref{claim-CofIncop} with $D_0 = C_0$ and $D_1 = C_1$ to obtain a $(0,\infty)$-homogeneous chain for $P$.  Thus we may assume that $C_0$ and $C_1$ are both infinite.  Let $\la p_n : n \in \Nb \ra$ be a bijective enumeration of $C_0$, and let $\la q_n : n \in \Nb \ra$ be a bijective enumeration of $C_1$.  The goal is now to work into a situation where the following claim applies.

\begin{ClaimSADSimpRSpocd2}\label{claim-InjToChain}
Suppose that there is an infinite $H \subseteq \Nb$ and an injection $f \colon H \imp \Nb$ where $\forall n \in H \; \bigl(p_n \comp_P q_{f(n)}\bigr)$.  Then $P$ contains a $(0,\infty)$-homogeneous chain.
\end{ClaimSADSimpRSpocd2}

\begin{proof}[Proof of Claim]
We may assume that the range of $f$ exists as a set by shrinking $H$ if necessary.

For this argument, for a $p \in C_0$, let $p\ua = \{x \in C_0 : x \geq_P p\}$ and $p\da = \{x \in C_0 : x \leq_P p\}$ denote the upward and downward closures of $p$ in $C_0$.  Similarly, for a $q \in C_1$, let $q\ua$ and $q\da$ denote the upward and downward closures of $q$ in $C_1$.

Let $X = \{n \in H : p_n <_P q_{f(n)}\}$, and let $Y = \{n \in H : p_n >_P q_{f(n)}\}$.  Then $H = X \cup Y$, so at least one of $X$ and $Y$ is infinite.  The two cases are symmetric, so suppose that $X$ is infinite for the sake of argument.  Let $L = \{p_n : n \in X\}$.  Then $L \subseteq C_0$ and $C_0$ is a chain, so $L$ is an infinite linear order.  If $L$ is also stable, then it has either an infinite ascending sequence or an infinite descending sequence by $\sads$.  Thus $P$ has either an infinite ascending sequence or an infinite descending sequence, so $P$ has a $(0,\infty)$-homogeneous chain by $(\star)$.

If $L$ is not stable, then let $n \in X$ be such that $p_n\ua \cap L$ and $p_n\da \cap L$ are both infinite.  Suppose that there is an $m \in X$ where $p_n \leq_P p_m$ and $q_{f(m)}\ua$ is infinite.  Then $p_n \leq_P p_m <_P q_{f(m)}$, so $C = p_n \da \cup q_{f(m)}\ua$ is an infinite chain in $P$.  Furthermore, $C$ is $(0,\infty)$-homogeneous because it has infinite intersection with both $C_0$ and $C_1$.  Finally, suppose instead that $q_{f(m)}\ua$ is finite whenever $m \in X$ and $p_n \leq_P p_m$.  Then the set $\{q_{f(m)} : m \in X \;\andd\; p_m \geq_P p_n\} \subseteq C_1$ is an infinite chain because $p_n \ua \cap L$ is infinite, but it has no least element.  We may therefore define an infinite descending sequence in $\{q_{f(m)} : m \in X \;\andd\; p_m \geq_P p_n\}$.  Thus $P$ has an infinite descending sequence, so it has a $(0,\infty)$-homogeneous chain by $(\star)$.
\end{proof}

To finish the proof, define a coloring $c \colon [\Nb]^2 \imp 3$ as follows.
\begin{align*}
c(n,m) =
\begin{cases}
0 & \text{if $\forall i \leq m \; (p_n \mid_P q_i)$}\\
1 & \text{if $\exists i \, (n < i \leq m \;\andd\; p_n \comp_P q_i)$}\\
2 & \text{if $\exists i \, (i \leq n \;\andd\; p_n \comp_P q_i) \;\andd\; \forall i \, (n < i \leq m \;\imp\; p_n \mid_P q_i)$}.
\end{cases}
\end{align*}
For every $n$, the color of $c(n,m)$ changes at most twice.  This can be seen by considering the value of $c(n, n+1)$ and making the following observations.
\begin{itemize}
\item If $c(n, m_0) = 0$ for some $m_0 > n$, then $c(n,m) \neq 2$ for all $m \geq m_0$.

\medskip

\item If $c(n, m_0) = 1$ for some $m_0 > n$, then $c(n,m) = 1$ for all $m \geq m_0$.

\medskip

\item If $c(n, m_0) = 2$ for some $m_0 > n$, then $c(n,m) \neq 0$ for all $m \geq m_0$.
\end{itemize}
Apply $\ssrt{2}{2}{3}$, which is provable in $\rca$ by Proposition~\ref{prop-SSRT} item~\ref{it-SSRTRCA}, to $c$ to obtain a set $H$ that is homogeneous for $c$, and consider the color for which $H$ is homogeneous.

Suppose that $H$ is $0$-homogeneous.  Let $C = \{p_n : n \in H\}$.  Then $C \subseteq C_0$ is an infinite chain, and thus every element of $C_0$ is comparable with every element of $C$.  Furthermore, every element of $C_1$ is incomparable with every element of $C$ because $H$ is $0$-homogeneous.  So $C$ is a $(0,\infty)$-homogeneous chain for $P$.

Suppose that $H$ is $1$-homogeneous.  Define a function $f \colon H \imp \Nb$ by $f(n) = i$, where $i$ is $<$-least with $i > n$ and $p_n \comp_P q_i$.  Such an $i$ exists by the $1$-homogeneity of $H$.  The function $f$ is injective because if $n$ and $m$ are members of $H$ with $n < m$, then $f(n) \leq m$ and $f(m) > m$.  Thus $H$ and $f$ satisfy the hypothesis of Claim~\ref{claim-InjToChain}, so $P$ contains a $(0,\infty)$-homogeneous chain.

Finally, suppose that $H$ is $2$-homogeneous.  First, suppose that there is a bound $m_0$ such that $\forall n \in H \; \forall i \geq m_0 \; (p_n \mid_P q_i)$.  In this case, $D_0 = \{p_n : n \in H\}$ and $D_1 = \{q_0, \dots, q_{m_0 - 1}\}$ satisfy the hypothesis of Claim~\ref{claim-CofIncop}, so $P$ contains a $(0,\infty)$-homogeneous chain.  If there is no such bound $m_0$, then $\forall m_0 \; \exists n \in H \; \exists i \geq m_0 \; (p_n \comp_P q_i)$.  However, given $m_0$, there cannot be a witnessing $n \in H$ and $i \geq m_0$ with $n < m_0$.  If there were, then we could choose an $m \in H$ with $m > i$, in which case we would have $n, m \in H$, $n < i < m$, and $p_n \comp_P q_i$.  We would then have that $c(n,m) = 1$, which contradicts that $H$ is $2$-homogeneous.  Therefore, the situation is that $\forall m \, \exists n \geq m \, \exists i \geq m \, (n \in H \;\andd\; p_n \comp_P q_i)$.  We can thus define an infinite subset $H_0 = \{n_0, n_1, n_2, \dots\}$ of $H$ and an injection $f \colon H_0 \imp \Nb$ as follows.  Given $n_0 < n_1 < \cdots < n_{\ell-1}$ and $f(n_0), f(n_1), \dots, f(n_{\ell-1})$, search for the first pair $\la n, i \ra$ with $n$ and $i$ both greater than $\max_<\{n_0, \dots, n_{\ell-1}, f(n_0), \dots, f(n_{\ell-1})\}$, with $n \in H$, and with $p_n \comp_P q_i$; put $n_\ell = n$; and put $f(n_\ell) = i$.  Then $H_0$ and $f$ satisfy the hypothesis of Claim~\ref{claim-InjToChain}, so $P$ contains a $(0,\infty)$-homogeneous chain.
\end{proof}

Finally, we show that $\cofrspo_2$ is provable in $\rca + \isig{2} + \ads$.  To do this, we first adapt Lemma~\ref{lem-LadderFinder} for use with partial orders of width $\leq\! 2$.

\begin{Lemma}\label{lem-W2LadderFinder}
The following is provable in $\rca$. Let $(P, <_P)$ be a partial order of width $\leq\! 2$, and let $\vec{P}$ denote $\vec{P} = P \oplus {<_P}$.  Then there is a Turing functional $\findlad^{\vec{P}} \colon P^\Nb \imp P^\Nb$ relative to $\vec{P}$ with the following properties for every infinite sequence $A = \la a_n : n \in \Nb \ra$ in $P$.

\begin{enumerate}[(1)]
\item\label{it-W2IsTot} If $A$ is an infinite ascending sequence in $P$ and
\begin{align*}
\forall m \; \exists n > m \; \exists p \in P \; \bigl(p \geq_P a_m \;\andd\; p \mid_P a_n\bigr),
\end{align*}
then $\findlad^{\vec{P}}(A)$ is total.

\medskip

\item\label{it-W2TotLadder} If $\findlad^{\vec{P}}(A)$ is total, then it computes a ladder for $A$ in $P$ such that
\begin{align*}
\forall m \; \exists n > m \; \bigl(\findlad^{\vec{P}}(A)(m) \mid_P a_n\bigr).
\end{align*}
\end{enumerate}
\end{Lemma}

\begin{proof}
Let $A = \la a_n : n \in \Nb \ra$ be an infinite sequence in $P$.  Compute $\findlad^{\vec{P}}(A)(0)$ by searching for an $n > 0$ and a $p_0 \in P$ with $p_0 \geq_P a_0$ and $p_0 \mid_P a_n$.  If $p_0$ is found, then output $\findlad^{\vec{P}}(A)(0) = p_0$.  To compute $\findlad^{\vec{P}}(A)(m+1)$, first compute $p_m = \findlad^{\vec{P}}(A)(m)$.  Then search for an $n > m+1$ and a $p_{m+1} \in P$ with $p_{m+1} >_P p_m$, $p_{m+1} \geq_P a_{m+1}$, and $p_{m+1} \mid_P a_n$.  If $p_{m+1}$ is found, then output $\findlad^{\vec{P}}(A)(m+1) = p_{m+1}$.

Item~\ref{it-W2TotLadder} follows immediately from the definition of $\findlad^{\vec{P}}$.  If $A = \la a_n : n \in \Nb \ra$ is an infinite sequence in $P$ and if $\findlad^{\vec{P}}(A)$ is total, then for every $m$ it must be that $\findlad^{\vec{P}}(A)(m) \in P$, that $a_m \leq_P \findlad^{\vec{P}}(A)(m)$, that $\findlad^{\vec{P}}(A)(m) <_P \findlad^{\vec{P}}(A)(m+1)$, and that there is an $n > m$ such that $\findlad^{\vec{P}}(A)(m) \mid_P a_n$.

For item~\ref{it-W2IsTot}, let $A = \la a_n : n \in \Nb \ra$ be an infinite ascending sequence in $P$, and suppose that for every $m$ there are an $n > m$ and a $p \in P$ with $p \geq_P a_m$ and $p \mid_P a_n$.  We use $\isig{1}$ to show that $\findlad^{\vec{P}}(A)$ is total.  By assumption, there are an $n > 0$ and a $p_0 \in P$ with $p_0 \geq_P a_0$ and $p_0 \mid_P a_n$.  Thus $\findlad^{\vec{P}}(A)(0)$ is defined.  Inductively assume that $p_m = \findlad^{\vec{P}}(A)(m)$ is defined.  Then $p_m \geq_P a_m$, and there is an $\ell > m$ such that $p_m \mid_P a_\ell$.  By assumption, there are an $s > \ell$ and an $x \in P$ with
\begin{align*}
x \geq_P a_\ell \quad\text{and}\quad x \mid_P a_s.
\end{align*}
Notice that $a_{m+1} \leq_P a_\ell \leq_P a_s$ because $m+1 \leq \ell \leq s$ and $A$ is an ascending sequence.  The element $p_m$ is comparable either with $x$ or with $a_s$ because $x \mid_P a_s$ and $P$ has width $\leq\! 2$.  Also, $x \nleq_P p_m$ because otherwise we would have the contradiction $a_\ell \leq_P x \leq_P p_m$.  Similarly, $a_s \nleq_P p_m$ because otherwise we would have the contradiction $a_\ell \leq_P a_s \leq_P p_m$.  Therefore either $x >_P p_m$ or $a_s >_P p_m$.

If $x >_P p_m$, then $x$ satisfies $x >_P p_m$, $x \geq_P a_\ell \geq_P a_{m+1}$, and $x \mid_P a_s$.  That is, there are an $n > m+1$ and a $p_{m+1} \in P$ with $p_{m+1} >_P p_m$, $p_{m+1} \geq_P a_{m+1}$, and $p_{m+1} \mid_P a_n$.  So $\findlad^{\vec{P}}(A)(m+1)$ is defined.

Otherwise $a_s >_P p_m$.  In this case, again by assumption there are a $t > s$ and a $y \in P$ with
\begin{align*}
y \geq_P a_s \quad\text{and}\quad y \mid_P a_t.
\end{align*}
This $y$ satisfies $y \geq_P a_s >_P p_m$, $y \geq_P a_s \geq_P a_{m+1}$, and $y \mid_P a_t$.  That is, there are an $n > m+1$ and a $p_{m+1} \in P$ with $p_{m+1} >_P p_m$, $p_{m+1} \geq_P a_{m+1}$, and $p_{m+1} \mid_P a_n$.  So $\findlad^{\vec{P}}(A)(m+1)$ is defined.
\end{proof}

\begin{Theorem}\label{thm-cofRSpo2InADSISig2}
$\rca + \isig{2} + \ads \vdash \cofrspo_2$.
\end{Theorem}

\begin{proof}
Let $(P, <_P)$ be an infinite partial order of width $\leq\! 2$.  $\rca \vdash \cc_2$ by Proposition~\ref{prop-CC-CA}, so we may apply $\cc_2$ to $P$ to obtain an infinite chain $C$.  Apply $\ads$ to $C$ to obtain either an infinite ascending sequence or an infinite descending sequence in $C$.  By reversing $P$ if necessary, we may assume that $P$ contains an infinite ascending sequence $A = \la a_n : n \in \Nb \ra$.

Let $\vec{P} = P \oplus {<_P}$, and let $\findlad^{\vec{P}} \colon P^\Nb \imp P^\Nb$ be the Turing functional relative to $\vec{P}$ from Lemma~\ref{lem-W2LadderFinder}.  Let $\findlad^{\vec{P}, i}$ denote the $i$\textsuperscript{th} iteration of $\findlad^{\vec{P}, i}$ given by
\begin{align*}
\findlad^{\vec{P}, 0}(A) &= A\\
\findlad^{\vec{P}, i+1}(A) &= \findlad^{\vec{P}}(\findlad^{\vec{P}, i}(A)).
\end{align*}
There are two cases:  either $\findlad^{\vec{P}, i}(A)$ is partial for some $i$, or $\findlad^{\vec{P}, i}(A)$ is total for all $i$.

First suppose that $\findlad^{\vec{P}, i}(A)$ is partial for some $i$.  Then, by $\isig{2}$ in the form of the $\Sigma^0_2$ least element principle, there is a least $i$ such that $\findlad^{\vec{P}, i}(A)$ is partial.  As $\findlad^{\vec{P}, 0}(A)$ is total, it must be that $i > 0$ and that $\findlad^{\vec{P}, i-1}(A)$ is total.  Then $\findlad^{\vec{P}, i-1}(A)$ computes an infinite ascending sequence in $P$ because either $i-1 = 0$, in which case $\findlad^{\vec{P}, i-1}(A)$ is $A$; or $i-1 > 0$, in which case $\findlad^{\vec{P}, i-1}(A) = \findlad^{\vec{P}}(\findlad^{\vec{P}, i-2}(A))$ computes an infinite ascending sequence by Lemma~\ref{lem-W2LadderFinder} item~\ref{it-W2TotLadder}.  Let $B = \la b_n : n \in \Nb \ra$ denote the infinite ascending sequence computed by $\findlad^{\vec{P}, i-1}(A)$, where $b_n = \findlad^{\vec{P}, i-1}(A)(n)$ for each $n$.  We claim that some tail of $B$ is $(0,\infty)$-homogeneous.  If not, then every tail of $B$ has a counterexample, which implies that
\begin{align*}
\forall m \; \exists n > m \; \exists p \in P \; \bigl(p \geq_P b_m \;\andd\; p \mid_P b_n\bigr).
\end{align*}
Therefore
\begin{align*}
\findlad^{\vec{P}, i}(A) = \findlad^{\vec{P}}(\findlad^{\vec{P}, i-1}(A)) = \findlad^{\vec{P}}(B)
\end{align*}
is total by Lemma~\ref{lem-W2LadderFinder} item~\ref{it-W2IsTot}.  This contradicts that $\findlad^{\vec{P}, i}(A)$ is partial.  Therefore some tail $B_{\geq n}$ of $B$ is $(0,\infty)$-homogeneous.  Ascending $(0,\infty)$-homogeneous sequences are necessarily $(0,\cof)$-homogeneous, and infinite subsets of $(0,\cof)$-homogeneous chains are necessarily $(0,\cof)$-homogeneous as well.  Thus we may thin $B_{\geq n}$ to an infinite sequence whose range exists as a set and thereby obtain a $(0,\cof)$-homogeneous chain for $P$.

Now suppose that $\findlad^{\vec{P}, i}(A)$ is total for all $i$.  Then the sequence $(\findlad^{\vec{P}, i}(A) : i \in \Nb)$ is uniformly computable.  Let $A_i = \la a^i_n : n \in \Nb \ra$ denote the infinite sequence computed by $\findlad^{\vec{P}, i}(A)$, where $a^i_n = \findlad^{\vec{P}, i}(A)(n)$ for all $i$ and $n$.  For every $i$, $A_i$ is an infinite ascending sequence in $P$ and $A_{i+1}$ is a ladder for $A_i$ by the fact that $A_0 = A$, by the assumption that every $\findlad^{\vec{P}, i}(A)$ is total, and by Lemma~\ref{lem-W2LadderFinder} item~\ref{it-W2TotLadder}.  We have that for every $i_0, i_1, n_0, n_1$, if $i_0 \leq i_1$ and $n_0 \leq n_1$, then $a^{i_0}_{n_0} \leq_P a^{i_1}_{n_1}$.  This is because 
\begin{align*}
a^{i_0}_{n_0} \leq_P a^{i_0}_{n_1} \leq_P a^{i_0 + 1}_{n_1} \leq_P a^{i_0 + 2}_{n_1} \leq_P \cdots \leq_P a^{i_1}_{n_1}.
\end{align*}
The inequality $a^{i_0}_{n_0} \leq_P a^{i_0}_{n_1}$ is because $A_{i_0}$ is an ascending sequence.  The inequalities $a^{i_0}_{n_1} \leq_P a^{i_0 + 1}_{n_1} \leq_P \cdots \leq_P a^{i_1}_{n_1}$ are because $A_{i+1}$ is a ladder for $A_i$ for each $i$.  Additionally, if $n_0 < n_1$, then the inequality $a^{i_0}_{n_0} <_P a^{i_0}_{n_1}$ is strict and therefore the inequality $a^{i_0}_{n_0} <_P a^{i_1}_{n_1}$ is strict as well.

Define an infinite sequence $B = \la b_i : i \in \Nb \ra$ by $b_i = a^i_i$ for each $i$.  The sequence $B$ is ascending in $P$ because $a_i^i <_P a_{i+1}^{i+1}$ for each $i$.  The ascending sequence $B$ is also $(0,\cof)$-homogeneous.  To see this, consider a $p \in P$ such that $p \nleq_P b_i$ for all $i$.  We show that $p \geq_P b_i$ for all $i$.  Given $i$, there is an $n > i+1$ with $a^{i+1}_{i+1} \mid_P a^i_n$ by Lemma~\ref{lem-W2LadderFinder} item~\ref{it-W2TotLadder}.  The element $p$ is comparable either with $b_{i+1} = a^{i+1}_{i+1}$ or with $a^i_n$ because $P$ has width $\leq 2$.  We have that $p \nleq_P b_{i+1}$ by assumption and that $p \nleq_P a^i_n$ because $p \leq_P a^i_n$ yields the contradiction $p \leq_P a^i_n \leq_P a^n_n = b_n$.  Thus either $p \geq_P b_{i+1} \geq_P b_i$ or $p \geq_P a^i_n \geq_P a^i_i = b_i$.  Therefore $p \geq_P b_i$, as desired.  Thus every $p \in P$ is either below some element of $B$, in which case it is below almost every element of $B$ because $B$ is an ascending sequence; or is above all elements of $B$.  So $B$ is a $(0,\cof)$-homogeneous ascending sequence.  As above, we may thin $B$ to an infinite sequence whose range exists as a set and thereby obtain a $(0,\cof)$-homogeneous chain for $P$.
\end{proof}

Unfortunately, the method of Theorem~\ref{thm-cofRSpo2InADSISig2} does not appear to readily generalize even to $3$-chain decomposable partial orders.

\section{Reversals and equivalences}\label{sec-reverse}

We supply the following reversals.
\begin{itemize}
\item $\rca + \rspocd_3 \vdash \ads$ (Lemma~\ref{lem-RSpoCD3ImpADS}).

\medskip

\item $\rca + \rspocd_2 \vdash \sads$ (Lemma~\ref{lem-RSpoCD2ImpSADS}).

\medskip

\item $\rca + \rspocd_{<\infty} \vdash \isig{2}$ (Lemma~\ref{lem-RSpoCDImpISig2}).

\medskip

\item $\rca + \cofrspocd_2 \vdash \ads$ (Lemma~\ref{lem-cofRSpoCD2ImpADS}).
\end{itemize}

Theorem~\ref{thm-RSpoEquivs} then combines these reversals with the results of the previous section in order to characterize the axiomatic strength of several versions of the Rival--Sands theorem for partial orders.  We also present a few questions.

\begin{Lemma}\label{lem-RSpoCD3ImpADS}
$\rca + \rspocd_3 \vdash \ads$.
\end{Lemma}

\begin{proof}
Let $(L, <_L)$ be an infinite linear order.  Let $(Q, <_Q)$ be the three-element partial order $Q = \{a, b, z\}$ with $a, b <_Q z$ and $a \mid_Q b$.  Consider the product partial order $(P, <_P)$, where $P = L \times Q$ and $\la \ell_0, q_0 \ra \leq_P \la \ell_1, q_1 \ra$ if and only if $\ell_0 \leq_L \ell_1$ and $q_0 \leq_Q q_1$.  The partial order $P$ has the $3$-chain decomposition $C_a = L \times \{a\}$, $C_b = L \times \{b\}$, $C_z = L \times \{z\}$.  Thus by $\rspocd_3$, let $C$ be a $(0,\infty)$-homogeneous chain for $P$.

Notice that $C$ cannot intersect both $C_a$ and $C_b$ because $\la \ell_0, a \ra \mid_P \la \ell_1, b \ra$ for every $\ell_0, \ell_1 \in L$.  Thus either $C \subseteq C_a \cup C_z$ or $C \subseteq C_b \cup C_z$.  Assume for the sake of argument that $C \subseteq C_a \cup C_z$.  The $C \subseteq C_b \cup C_z$ case is symmetric.

We claim that $C \cap C_z$ has no maximum element.  Suppose for a contradiction that $\la m, z \ra$ is the maximum element of $C \cap C_z$, and consider the element $\la m, b \ra$.  Then $\la m, b \ra <_Q \la m, z \ra$.  However, every other element of $C$ is either of the form $\la \ell, a \ra$, in which case $\la m, b \ra \mid_P \la \ell, a \ra$ because $b \mid_Q a$; or of the form $\la \ell, z \ra$ with $\ell <_L m$, in which case $\la m, b \ra \mid_P \la \ell, z \ra$ because $m >_L \ell$ and $b <_Q z$.  Thus $\la m, b \ra$ is comparable with exactly one element of $C$, contradicting that $C$ is $(0,\infty)$-homogeneous.

If $C \cap C_z \neq \emptyset$, then $C \cap C_z$ is non-empty and has no maximum element, so we can define an infinite ascending sequence $\la \ell_0, z \ra <_P \la \ell_1, z \ra <_P \la \ell_2, z \ra <_P \cdots$ in $C \cap C_z$.  This yields an infinite ascending sequence $\ell_0 <_L \ell_1 <_L \ell_2 <_L \cdots$ in $L$.

If $C \cap C_z = \emptyset$, then $C \subseteq C_a$.  In this case, we claim that $C$ has no minimum element.  Suppose for a contradiction that $\la m, a \ra$ is the minimum element of $C$, and consider the element $\la m, z \ra$.  Then $\la m, a \ra <_P \la m, z \ra$.  However, every other element of $C$ is of the form $\la \ell, a \ra$ with $m <_L \ell$, in which case $\la m, z \ra \mid_P \la \ell, a \ra$ because $m <_L \ell$ and $z >_Q a$.  Thus $\la m, z \ra$ is comparable with exactly one element of $C$, contradicting that $C$ is $(0,\infty)$-homogeneous.  We may now define an infinite descending sequence $\la \ell_0, a \ra >_P \la \ell_1, a \ra >_P \la \ell_2, a \ra >_P \cdots$ in $C$ and hence an infinite descending sequence $\ell_0 >_L \ell_1 >_L \ell_2 >_L \cdots$ in $L$.

Thus $L$ has either an infinite ascending sequence or an infinite descending sequence.
\end{proof}

\begin{Lemma}\label{lem-RSpoCD2ImpSADS}
$\rca + \rspocd_2 \vdash \sads$.
\end{Lemma}

\begin{proof}
Let $(L, <_L)$ be an infinite stable linear order.  Let $(Q, <_Q)$ be the two-element linear order $Q = \{a, z\}$ with $a <_Q z$, and let $(P, <_P)$ be the product partial order $L \times Q$.  The partial order $P$ has the $2$-chain decomposition $C_a = L \times \{a\}$, $C_z = L \times \{z\}$.  Thus by $\rspocd_2$, let $C$ be a $(0,\infty)$-homogeneous chain for $P$.

The linear order $L$ is stable, so every $\ell \in L$ has either finitely many $<_L$-predecessors or finitely many $<_L$-successors.  We claim that $C \cap C_a$ cannot contain two elements $\la \ell, a \ra$ and $\la r, a \ra$ where $\ell$ has only finitely many $<_L$-predecessors and $r$ has only finitely many $<_L$-successors.  Suppose for a contradiction that $C \cap C_a$ does contain such an $\la \ell, a \ra$ and $\la r, a \ra$, and consider the element $\la \ell, z \ra$.  The element $\la \ell, z \ra$ is comparable with an $\la x, a \ra \in C_a$ if and only if $x \leq_L \ell$, and there are only finitely many such elements $x \in L$.  Thus $\la \ell, z \ra$ is comparable with only finitely many elements of $C \cap C_a$.  On the other hand, the element $\la r, a \ra$ is comparable with an $\la x, z \ra \in C_z$ if and only if $r \leq_L x$, and there are only finitely many such elements $x \in L$.  Thus $C \cap C_z$ is finite because all of its elements are comparable with $\la r, a \ra$.  It follows that $\la \ell, z \ra$ is comparable with $\la \ell, a \ra \in C$ and is comparable with only finitely many elements of $C$ in total.  This contradicts that $C$ is $(0,\infty)$-homogeneous for $P$.  Symmetric reasoning shows that $C \cap C_z$ also cannot contain two elements $\la \ell, z \ra$ and $\la r, z \ra$ where $\ell$ has only finitely many $<_L$-predecessors and $r$ has only finitely many $<_L$-successors.

The chain $C$ is infinite, so either $C \cap C_a$ is infinite or $C \cap C_z$ is infinite.  Suppose for the sake of argument that $C \cap C_a$ is infinite.  The other case is symmetric.  By the claim above, it must be that either $\ell$ has only finitely many $<_L$-predecessors whenever $\la \ell, a \ra \in C \cap C_a$ or that $\ell$ has only finitely many $<_L$-successors whenever $\la \ell, a \ra \in C \cap C_a$.  Suppose that $\ell$ has only finitely many $<_L$-predecessors whenever $\la \ell, a \ra \in C \cap C_a$.  Then for every $\la \ell, a \ra \in C \cap C_a$, there is an $\la r, a \ra \in C \cap C_a$ with $\la \ell, a \ra <_P \la r, a \ra$.  We may thus define an infinite ascending sequence $\la \ell_0, a \ra <_P \la \ell_1, a \ra <_P \la \ell_2, a \ra <_P \cdots$ in $C \cap C_a$ and hence an infinite ascending sequence $\ell_0 <_L \ell_1 <_L \ell_2 <_L \cdots$ in $L$.  Similarly, if $\ell$ has only finitely many $<_L$-successors whenever $\la \ell, a \ra \in C \cap C_a$, then for every $\la \ell, a \ra \in C \cap C_a$, there is an $\la r, a \ra \in C \cap C_a$ with $\la r, a \ra <_P \la \ell, a \ra$.  We may thus define an infinite descending sequence $\la \ell_0, a \ra >_P \la \ell_1, a \ra >_P \la \ell_2, a \ra >_P \cdots$ in $C \cap C_a$ and hence an infinite descending sequence $\ell_0 >_L \ell_1 >_L \ell_2 >_L \cdots$ in $L$.  Thus $L$ has either an infinite ascending sequence or an infinite descending sequence.
\end{proof}

\begin{Lemma}\label{lem-RSpoCDImpISig2}
$\rca + \rspocd_{<\infty} \vdash \isig{2}$.
\end{Lemma}

\begin{proof}
We show that $\rca + \rspocd_{<\infty}$ proves the $\Pi^0_2$ least element principle, which is equivalent to $\isig{2}$ over $\rca$ as explained in Section~\ref{sec-background}.  Notice that $\rca + \rspocd_{<\infty} \vdash \bsig{2}$ because $\rca + \rspocd_{<\infty} \vdash \ads$ by Lemma~\ref{lem-RSpoCD3ImpADS}, and $\rca + \ads \vdash \bsig{2}$ as explained in Section~\ref{sec-background}.  Thus we may make use of $\rt^1_{<\infty}$ in the following argument.

Let $\forall x \, \exists y \, \varphi(n, x, y)$ be a $\Pi^0_2$ formula, possibly with undisplayed parameters, where $\varphi$ is $\Sigma^0_0$.  Let $n$ be such that $\forall x \, \exists y \, \varphi(n, x, y)$.  We want to find the least $i$ such that $\forall x \, \exists y \, \varphi(i, x, y)$.  Define a partial order $(P, <_P)$ by
\begin{align*}
P = \Bigl\{\la i, s, t \ra : (i \leq n) \;\andd\; \bigl(\forall x \leq s \; \exists y \leq t \; \varphi(i,x,y)\bigr) \;\andd\; \bigl(\exists x \leq s \; \forall y < t \; \neg\varphi(i,x,y)\bigr)\Bigr\}
\end{align*}
and
\begin{align*}
\la i_0, s_0, t_0 \ra \leq_P \la i_1, s_1, t_1 \ra \quad\Biimp\quad i_1 \leq i_0 \;\andd\; s_1 \geq s_0.
\end{align*}
That is, $P$ consists of all triples $\la i, s, t \ra$ where $i \leq n$ and $t$ is least such that $\forall x \leq s \; \exists y \leq t \; \varphi(i,x,y)$.  Notice that given $i$ and $s$, there is at most one $t$ with $\la i, s, t \ra \in P$.

By assumption, $\forall x \, \exists y \, \varphi(n, x, y)$.  Thus given any $s$, we have that $\forall x \leq s \; \exists y \; \varphi(n, x, y)$, and therefore by $\bsig{0}$ there is a $t$ such that $\forall x \leq s \; \exists y \leq t \; \varphi(n, x, y)$.  Moreover, there is a least such $t$ by the $\Sigma^0_0$ least element principle.  This shows that for every $s$ there is a $t$ with $\la n, s, t \ra \in P$.  Therefore $P$ is infinite.  Furthermore, $P$ has the $(n+1)$-chain decomposition $C_0, \dots, C_n$, where $C_i = \{\la i, s, t \ra: \la i, s, t \ra \in P\}$ for each $i \leq n$.  By $\rspocd_{<\infty}$, let $C$ be a $(0,\infty)$-homogeneous chain for $P$.  By $\rt^1_{<\infty}$, there is an $i$ such that $C \cap C_i$ is infinite.  We show that $i$ is least such that $\forall x \, \exists y \, \varphi(i, x, y)$.

First, as $C \cap C_i$ is infinite, given any $x_0$ there are an $s$ and a $t$ with $\la i, s, t \ra \in P$ and $s \geq x_0$.  The fact that $\la i, s, t \ra \in P$ means that $\forall x \leq s \; \exists y \leq t \; \varphi(i, x, y)$.  Thus $\exists y \, \varphi(i, x_0, y)$ because $x_0 \leq s$.  Therefore $\forall x \, \exists y \, \varphi(i, x, y)$.

Second, if $j < i$, then $C \cap C_j = \emptyset$.  This is because for any $\la j, s_0, t_0 \ra \in P$, there is an $\la i, s_1, t_1 \ra \in C \cap C_i$ with $s_1 > s_0$ because $C \cap C_i$ is infinite.  Then $\la j, s_0, t_0 \ra \mid_P \la i, s_1, t_1 \ra$, so $\la j, s_0, t_0 \ra \notin C$ because $C$ is a chain.  Now suppose for a contradiction that there is a $j < i$ such that $\forall x \, \exists y \, \varphi(j, x, y)$.  Let $\la i, s_0, t_0 \ra$ be any element of $C \cap C_i$.  By the same argument as for $n$, the assumption $\forall x \, \exists y \, \varphi(j, x, y)$ implies that for every $s$ there is a $t$ with $\la j, s, t \ra \in P$.  Therefore there is a $t_1$ such that $\la j, s_0, t_1 \ra \in P$, and we have that $\la j, s_0, t_1 \ra >_P \la i, s_0, t_0 \ra$.  However, $C \subseteq \bigcup_{k = i}^n C_k$, and for $\la j, s_0, t_1 \ra$ to be comparable with some $\la k, s, t \ra \in \bigcup_{k = i}^n C_k$, it must be that $s \leq s_0$.  There are only finitely many $\la k, s, t \ra \in \bigcup_{k = i}^n C_k$ with $s \leq s_0$, so $\la j, s_0, t_1 \ra$ is comparable with only finitely many elements of $C$.  Thus $\la j, s_0, t_1 \ra$ is comparable with $\la i, s_0, t_0 \ra \in C$, but it is comparable with only finitely many elements of $C$ in total.  This contradicts that $C$ is $(0,\infty)$-homogeneous for $P$.  Therefore we cannot have that $\forall x \, \exists y \, \varphi(j, x, y)$, so $i$ is least such that $\forall x \, \exists y \, \varphi(i, x, y)$.
\end{proof}

\begin{Lemma}\label{lem-cofRSpoCD2ImpADS}
$\rca + \cofrspocd_2 \vdash \ads$.
\end{Lemma}

\begin{proof}
Let $(L, <_L)$ be an infinite linear order.  As in the proof of Lemma~\ref{lem-RSpoCD2ImpSADS}, let $(Q, <_Q)$ be the two-element linear order $Q = \{a, z\}$ with $a <_Q z$, and let $(P, <_P)$ be the product partial order $L \times Q$.  The partial order $P$ has the $2$-chain decomposition $C_a = L \times \{a\}$, $C_z = L \times \{z\}$.  Thus by $\cofrspocd_2$, let $C$ be a $(0,\cof)$-homogeneous chain for $P$.

The chain $C$ is infinite, so either $C \cap C_a$ is infinite or $C \cap C_z$ is infinite.  First suppose that $C \cap C_a$ is infinite.  Then $C \cap C_a$ has no minimum element.  Suppose for a contradiction that $\la m, a \ra$ is the minimum element of $C \cap C_a$, and consider the element $\la m, z \ra$.  Then $\la m, a \ra <_P \la m, z \ra$.  However, every other element of $C \cap C_a$ is of the form $\la \ell, a \ra$ with $m <_L \ell$, in which case $\la m, z \ra \mid_P \la \ell, a \ra$ because $m <_L \ell$ and $z >_Q a$.  Thus $\la m, z \ra$ is comparable with at least one element of $C$, but it is not comparable with cofinitely many elements of $C$ because it is incomparable with every element of $C \cap C_a$ except $\la m, a \ra$.  This contradicts that $C$ is $(0,\cof)$-homogeneous for $P$.  Thus $C \cap C_a$ is infinite and has no minimum element.  We may therefore define an infinite descending sequence $\la \ell_0, a \ra >_P \la \ell_1, a \ra >_P \la \ell_2, a \ra >_P \cdots$ in $C \cap C_a$ and hence an infinite descending sequence $\ell_0 >_L \ell_1 >_L \ell_2 >_L \cdots$ in $L$.

The case where $C \cap C_z$ is infinite is dual to the previous case.  If $C \cap C_z$ has maximum element $\la m, z \ra$, then $\la m, a \ra$ witnesses that $C$ is not $(0,\cof)$-homogeneous.  Thus $C \cap C_z$ is infinite and has no maximum element.  We may therefore define an infinite ascending sequence $\la \ell_0, z \ra <_P \la \ell_1, z \ra <_P \la \ell_2, z \ra <_P \cdots$ in $C \cap C_z$ and hence an infinite ascending sequence $\ell_0 <_L \ell_1 <_L \ell_2 <_L \cdots$ in $L$.

Thus $L$ has either an infinite ascending sequence or an infinite descending sequence.
\end{proof}

Thus for $k \geq 2$, $\cofrspo_k$ implies $\ads$, which implies that every $(0,\cof)$-homogeneous chain in a partial order of width $\leq\! k$ has a suborder of type either $\omega$ or $\omega^*$.  We therefore have the following proposition.

\begin{Proposition}\label{prop-COFvsAscDesc}
$\rca$ proves the statement ``For every $k \geq 2$, $\cofrspo_k$ holds if and only if every infinite partial order of width $\leq\! k$ has a $(0,\infty)$-homogeneous chain of order-type either $\omega$ or $\omega^*$.''
\end{Proposition}

\begin{proof}
Let $k \geq 2$.  In any partial order, a $(0,\infty)$-homogeneous chain of order-type $\omega$ or $\omega^*$ is necessarily $(0,\cof)$-homogeneous.  Thus if every infinite partial order of width $\leq\! k$ has a $(0,\infty)$-homogeneous chain of order-type either $\omega$ or $\omega^*$, then $\cofrspo_k$ holds.  Conversely, suppose that $\cofrspo_k$ holds.  Then $\cofrspo_2$ holds because $k \geq 2$, so $\ads$ holds by Lemma~\ref{lem-cofRSpoCD2ImpADS}.  Let $(P, <_P)$ be an infinite partial order of width $\leq\! k$.  Then $P$ has a $(0,\cof)$-homogeneous chain $C$ by $\cofrspo_k$, and $C$ has a suborder $B$ of type either $\omega$ or $\omega^*$ by $\ads$.  The chain $B$ is also $(0,\cof)$-homogeneous, and therefore it is $(0,\infty)$-homogeneous.  Thus $B$ is a $(0,\infty)$-homogeneous chain in $P$ of order-type either $\omega$ or $\omega^*$.
\end{proof}

Of course, Proposition~\ref{prop-COFvsAscDesc} also holds with $\cofrspocd_k$ and ``that is $k$-chain decomposable'' in place of $\cofrspo_k$ and ``of width $\leq\! k$.''

The following theorem characterizes the strength the Rival--Sands theorem for partial orders.

\begin{Theorem}\label{thm-RSpoEquivs}
{\ }
\begin{enumerate}[(1)]
\item\label{it-equivRSpo} $\rspo_{<\infty}$, $\rspocd_{<\infty}$, and $\isig{2} + \ads$ are pairwise equivalent over $\rca$.

\medskip

\item\label{it-equivRSpok} For each fixed standard $k \geq 3$, $\rspo_k$, $\rspocd_k$, and $\ads$ are pairwise equivalent over $\rca$.

\medskip

\item\label{it-equivRSpoCD2} $\rspocd_2$ and $\sads$ are equivalent over $\rca$.

\medskip

\item\label{it-equivRSpo2} $\rspo_2$, $\rspocd_2$, and $\sads$ are pairwise equivalent over $\wkl$.

\medskip

\item\label{it-equivCofRSpo2} $\cofrspo_2$, $\cofrspocd_2$, and $\ads$ are pairwise equivalent over $\rca + \isig{2}$.
\end{enumerate}
\end{Theorem}

\begin{proof}
Item~\ref{it-equivRSpo} is by Theorem~\ref{thm-ForwardRSpo}, Lemmas~\ref{lem-RSpoCD3ImpADS} and~\ref{lem-RSpoCDImpISig2}, and the fact that $\rspo_{<\infty}$ implies $\rspocd_{<\infty}$.  Item~\ref{it-equivRSpok} is by Theorem~\ref{thm-ForwardRSpok}, Lemma~\ref{lem-RSpoCD3ImpADS}, and the fact that $\rspo_k$ implies $\rspocd_k$.  Item~\ref{it-equivRSpoCD2} is by Theorem~\ref{thm-ForwardRSpoCD2} and Lemma~\ref{lem-RSpoCD2ImpSADS}.  Item~\ref{it-equivRSpo2} is by item~\ref{it-equivRSpoCD2} and the fact that $\wkl$ proves the equivalence of $\rspo_2$ and $\rspocd_2$, as explained in Section~\ref{sec-Dilworth}.  Item~\ref{it-equivCofRSpo2} is by Theorem~\ref{thm-cofRSpo2InADSISig2} and Lemma~\ref{lem-cofRSpoCD2ImpADS}
\end{proof}

We end this section with a few questions.  First, we still know no better proof of $\cofrspo_{<\infty}$ than the proof in $\pica$ from Theorem~\ref{thm-cofRSpoInPica}.  Of course, $\cofrspo_{<\infty}$ cannot be equivalent to $\pica$ over $\rca$ because it is a true $\Pi^1_2$ sentence, and true $\Pi^1_2$ sentences cannot imply $\pica$ over $\rca$ (see~\cite{AharoniMagidorShore}*{Proposition~4.17}).

\begin{Question}
What is the strength of $\cofrspo_{<\infty}$?  What is the strength of $\cofrspo_k$ for each fixed $k \geq 3$?
\end{Question}

In the case $k=2$, we do not know if $\rca + \isig{2}$ can be weakened to $\rca$ in Theorem~\ref{thm-RSpoEquivs} item~\ref{it-equivCofRSpo2}.

\begin{Question}
Are $\cofrspo_2$, $\cofrspocd_2$, and $\ads$ also pairwise equivalent over $\rca$?
\end{Question}

Finally, we do not know if the equivalence of $\rspo_2$ and $\sads$ over $\wkl$ of Theorem~\ref{thm-RSpoEquivs} item~\ref{it-equivRSpo2} also holds over $\rca$.

\begin{Question}
What is the strength of $\rspo_2$ relative to $\rca$?  In particular, does $\rca + \sads \vdash \rspo_2$?
\end{Question}

\section{Extending the Rival--Sands theorem to partial orders without infinite antichains}\label{sec-ExtendedRSpo}

It is possible for a partial order to have arbitrarily large finite antichains (and hence to \emph{not} have finite width) but still have no infinite antichain.  The goal of this section is to extend $\rspo$ and $\cofrspo$ to countably infinite partial orders that do not have infinite antichains.  To our knowledge, these extensions are new combinatorial results.  Furthermore, we show that the extension of $\rspo$ to countably infinite partial orders without infinite antichains is equivalent to $\aca$ over $\rca$.

Unions of ideals play the role of unions of chains when working with partial orders without infinite antichains.  Recall that an \emph{ideal} in a partial order $(P, <_P)$ is a set $I \subseteq P$ that is downward-closed:  $\forall p, q \in P \; ((p \in I \;\andd\; q \leq_P p) \;\imp\; q \in I)$ and upward-directed:  $\forall p, q \in I \; \exists r \in I \; (p \leq_P r \;\andd\; q \leq_P r)$.

A theorem of Bonnet~\cite{Bonnet}*{Lemma~2} states that a partial order has no infinite antichain if and only if every initial interval (i.e., downward-closed set) is a finite union of ideals.  Frittaion and Marcone determined that this theorem is equivalent to $\aca$ over $\rca$.

\begin{Theorem}[\cite{FrittaionMarcone}*{Theorem~4.5}]\label{thm-IdealDecomp}
The following are equivalent over $\rca$.
\begin{enumerate}[(1)]
\item\label{it-IdealACA} $\aca$.

\medskip

\item\label{it-IdealDecomp} For every partial order $(P, <_P)$, if $P$ has no infinite antichain, then every initial interval of $P$ is a finite union of ideals.
\end{enumerate}
\end{Theorem}

Furthermore, Frittaion and Marcone observe that in Theorem~\ref{thm-IdealDecomp} item~\ref{it-IdealDecomp}, it may additionally be assumed that the partial order $(P, <_P)$ is an \emph{essential} union of finitely many ideals.  This means that $P = \bigcup_{i<k}I_i$ for ideals $I_i \subseteq P$ with $i < k$ for some $k$, where additionally $I_i \nsubseteq \bigcup_{\substack{j < k \\ j \neq i}}I_j$ for every $i < k$ (see~\cite{FrittaionMarcone}*{Lemma~3.3}).  We warn the reader that when we write a partial order as a union of ideals, we may \emph{not} necessarily assume that the ideals are disjoint as we do with chain decompositions.

\begin{Theorem}\label{thm-CofExtendedRSpo}
$\pica$ proves the statement ``Every infinite partial order with no infinite antichain has a $(0,\cof)$-homogeneous chain.''
\end{Theorem}

\begin{proof}
Let $(P, <_P)$ be an infinite partial order that does not have infinite antichains.  Then $P$ must have an infinite chain by $\cac$, which must have either an infinite ascending sequence or an infinite descending sequence by $\ads$.  Thus we may assume that $P$ contains an infinite ascending sequence $A$ by reversing the partial order if necessary.  Use $\pica$ (and the fact that the $\Sigma^1_1$ sets are the complements of the $\Pi^1_1$ sets) to define the set $Q = \{q \in P : \text{$q\ua$ is reverse ill-founded}\}$ of elements of $P$ that have infinite ascending sequences above them.  Notice that $Q$ is non-empty because $A \subseteq Q$.

The proof proceeds in $\aca$ from this point onward.  The partial order $(Q, <_P)$ has no infinite antichain because it is a suborder of $P$.  Therefore $Q$ is an essential union of finitely many ideals $Q = \bigcup_{i < k} I_i$ for some $k > 0$ by the \ref{it-IdealACA}~$\Imp$~\ref{it-IdealDecomp} direction of Theorem~\ref{thm-IdealDecomp} and the comment that follows it.  No ideal $I_i$ for $i < k$ has a maximum element.  Suppose for a contradiction that $I_i$ has maximum element $q$.  As $q \in Q$, there is an infinite ascending sequence $q <_P b_0 <_P b_1 <_P \cdots$ in $P$ and therefore in $Q$.  The element $b_0$ is not in $I_i$ because $q$ is the maximum element of $I_i$.  Thus $b_0 \in I_j$ for some $j < k$, $j \neq i$.  But then $I_i \subseteq I_j$ because $b_0 \in I_j$, $q <_P b_0$, $I_j$ is downward-closed, and $q$ is the maximum element of $I_i$.  This contradicts that the union $\bigcup_{i < k} I_i$ is essential.  The partial order $Q$ is infinite, so $\rt^1_{<\infty}$ implies that ideal $I_\ell$ is infinite for some $\ell < k$.  (In fact, $I_i$ is infinite for every $i < k$ because $I_i$ is non-empty and has no maximum element.)  We may thus define an infinite ascending sequence $C = \la c_n : n \in \Nb \ra$ that is cofinal in $I_\ell$ as follows.  Let $\la x_n : n \in \Nb \ra$ enumerate the elements of $I_\ell$.  Let $c_0 = x_0$, and for each $n$, let $c_{n+1}$ be the $<_\Nb$-least element of $I_\ell$ with $c_{n+1} >_P c_n$ and $c_{n+1} >_P x_{n+1}$.  Such a $c_{n+1}$ necessarily exists because $I_\ell$ is an ideal with no maximum element.

We finish the proof by showing that some tail of $C$ is $(0,\infty)$-homogeneous for $P$ and therefore is $(0,\cof)$-homogeneous for $P$.  Suppose for a contradiction that no tail of $C$ is $(0,\infty)$-homogeneous.  Then every tail of $C$ has a counterexample $d$.  For each $n$, let $d_n$ be the $<_\Nb$-least counterexample to the tail $C_{\geq n}$.  The partial order $P$ has no infinite antichain, so by $\cac$ applied to the infinite suborder $\{d_n : n \in \Nb\}$, there is a sequence $n_0 < n_1 < n_2 < \cdots$ such that $\{d_{n_j} : j \in \Nb\}$ is a chain.

Arguing as in the proof of Lemma~\ref{lem-CXChain}, we have that for every $i$, it is the case that $d_{n_i} <_P d_{n_j}$ for all sufficiently large $j$.  Fix an $i$.  The element $d_{n_i}$ is a counterexample to $C_{\geq n_i}$, so there is an $s \geq n_i$ such that $d_{n_i} >_P c_s$ and $d_{n_i} \mid_P C_{\geq s+1}$.  Let $j > s+1$, and consider $d_{n_j}$.  The element $d_{n_j}$ is a counterexample to $C_{\geq n_j}$, so there is a $t \geq n_j$ such that $d_{n_j} >_P c_t$.  We cannot have that $d_{n_j} \leq_P d_{n_i}$ because this would yield that $c_{s+1} <_P c_t <_P d_{n_j} \leq_P d_{n_i}$, contradicting that $d_{n_i} \mid_P c_{s+1}$.  Note here that $s+1 < j \leq n_j \leq t$, so $c_{s+1} < _P c_t$ because $C$ is an ascending sequence.  Therefore it must be that $d_{n_i} <_P d_{n_j}$ because we know that $d_{n_i} \comp_P d_{n_j}$.

Using the above, we may thin the sequence $n_0 < n_1 < n_2 < \cdots$ so that $d_{n_0} <_P d_{n_1} <_P d_{n_2} <_P \cdots$ is an infinite ascending sequence in $P$.  It follows that $\{d_{n_j} : j \in \Nb\} \subseteq Q$.  By $\rt^1_{<\infty}$, there is an $s < k$ such that $I_s$ contains $d_{n_j}$ for infinitely many $j$.  We may therefore further thin the sequence $n_0 < n_1 < n_2 < \cdots$ so that $\{d_{n_j} : j \in \Nb\} \subseteq I_s$.  We cannot have that $s = \ell$ because every element of $I_\ell$ is below a tail of $C$ by the construction of $C$, whereas $d_{n_j}$ is incomparable with a tail of $C$ for every $j$.  Finally, we see that $I_\ell \pwb C$ by the construction of $C$ and that $C \pwb \{d_{n_j} : j \in \Nb\}$ because $c_{n_j} <_P d_{n_j}$ for every $j$.  Therefore $I_\ell \pwb \{d_{n_j} : j \in \Nb\} \subseteq I_s$, so $I_\ell \subseteq I_s$ because $I_s$ is an ideal.  This contradicts that the union $\bigcup_{i < k} I_i$ is essential because $s \neq \ell$.
\end{proof}

As in the case of $\rspo_{<\infty}$ versus $\cofrspo_{<\infty}$, if we only want to produce a chain that is $(0,\infty)$-homogeneous rather than $(0,\cof)$-homogeneous in an infinite partial order $(P, <_P)$ with no infinite antichain, then we may split into the cases of where $P$ contains a suborder of type $\zeta$ and where it does not.

\begin{Theorem}\label{thm-ExtendedRSpo}
$\aca$ proves the statement ``Every infinite partial order with no infinite antichain has a $(0,\infty)$-homogeneous chain.''
\end{Theorem}

\begin{proof}
The proof is the same as the proof of Theorem~\ref{thm-CofExtendedRSpo}, except now we use Lemma~\ref{lem-ChainSplitting} item~\ref{it-poSplitWF} in place of $\pica$.  Let $(P, <_P)$ be an infinite partial order that does not have infinite antichains.  A chain of order-type $\zeta$ is necessarily $(0,\infty)$-homogeneous, so we may additionally assume that $P$ has no suborder of type $\zeta$.  As in the proof of Theorem~\ref{thm-CofExtendedRSpo}, we may assume that $P$ has an infinite ascending sequence by applying $\cac$ and $\ads$ and by reversing the order if needed.  By the \ref{it-ACACS}~$\Imp$~\ref{it-poSplitWF} direction of Lemma~\ref{lem-ChainSplitting} and the comments that follow it, we may form the set $Q = \{q \in P : \text{$q\ua$ is reverse ill-founded}\}$.  The proof now continues exactly as in that of Theorem~\ref{thm-CofExtendedRSpo}.
\end{proof}

Recall that a partial order $(P, <_P)$ is a \emph{well-partial-order} if for every function $h \colon \Nb \imp P$, there are $m, n \in \Nb$ with $m < n$ such that $f(m) \leq_P f(n)$.  A function $h \colon \Nb \imp P$ witnessing that $P$ is not a well-partial-order (i.e., such that $h(m) \nleq_P h(n)$ whenever $m < n$) is called a \emph{bad sequence}.  Any bijective enumeration of an infinite antichain in a partial order $(P, <_P)$ is a bad sequence, so well-partial-orders do not have infinite antichains.

To reverse Theorem~\ref{thm-ExtendedRSpo} for well-partial-orders, we employ the construction of~\cite{FrittaionHendtlassMarconeShaferVanderMeeren}*{Definition~4.2}, which is a generalization of Construction~\ref{const-TFstages}.  This construction takes an injection $f \colon \Nb \imp \Nb$ and a finite partial order $P$ with a distinguished element $x \in P$ and produces an infinite partial order $\Xi_f(P,x)$ such that the range of $f$ is recursive in the join of $f$ with any bad sequence in $\Xi_f(P,x)$.  That is, either $\Xi_f(P,x)$ is a well-partial-order or the range of $f$ exists as a set.  Thus given an injection $f$, the strategy is to construct $\Xi_f(P,x)$ for the $2$-element antichain $P = \{x, y\}$, show that $\Xi_f(P,x)$ has no $(0,\infty)$-homogeneous chain, conclude that $\Xi_f(P,x)$ is not a well-partial-order, and then conclude that the range of $f$ exists as a set.

We include Construction~\ref{const-Xi} for the reader's convenience.  Given an injection $f \colon \Nb \imp \Nb$, let
\begin{align*}
T_s &= \{n < s : \text{$n$ is true at stage $s$}\}\\
&= \{n < s : \forall k \, (n < k \leq s \;\imp\; f(n) < f(k))\}
\end{align*}
for each $s \in \Nb$.

\begin{Construction}[\cite{FrittaionHendtlassMarconeShaferVanderMeeren}*{Definition~4.2}]\label{const-Xi}
Let $f \colon \Nb \imp \Nb$ be an injection, let $(P, <_P)$ be a finite partial order, and let $x \in P$.  Define the partial order $(Q, <_Q) = \Xi_f(P,x)$ as follows.  Make countably many disjoint copies $(P_n, <_{P_n})$ of $P$ by setting $P_n = \{n\} \times P$ and by setting $\la n, y \ra <_{P_n} \la n, z \ra$ if and only if $y <_P z$ for all $n \in \Nb$ and all $x, y \in P$.  Let $x_n = \la n, x \ra$ denote the copy of $x$ in $P_n$.  The domain of $Q$ is $\bigcup_{n \in \Nb}P_n$.  Define $<_Q$ in stages, where at stage $s$, $<_Q$ is defined on $\bigcup_{n \leq s}P_n$.

\begin{itemize}
\item At stage $0$, define $<_Q$ to be $<_{P_0}$ on $P_0$.

\medskip

\item Suppose $<_Q$ is defined on $\bigcup_{n \leq s} P_s$.  There are two cases.

\medskip

\begin{enumerate}[(1)]
\item If $T_{s+1} \subsetneq T_s \cup \{s\}$, let $n_0$ be the least element of $(T_s \cup \{s\}) \setminus T_{s+1}$, and place $P_{s+1}$ immediately above $x_{n_0}$.  That is, place the elements of $P_{s+1}$ above all $y \in \bigcup_{n \leq s} P_s$ such that $y \leq_Q x_{n_0}$, below all $y \in \bigcup_{n \leq s} P_s$ such that $y >_Q x_{n_0}$, and incomparable with all $y \in \bigcup_{n \leq s} P_s$ that are incomparable with $x_{n_0}$.

\medskip

\item If $T_{s+1} = T_s \cup \{s\}$, place $P_{s+1}$ immediately below $x_s$.  That is, place the elements of $P_{s+1}$ above all $y \in \bigcup_{n \leq s} P_s$ such that $y <_Q x_s$, below all $y \in \bigcup_{n \leq s} P_s$ such that $y \geq_Q x_s$, and incomparable with all $y \in \bigcup_{n \leq s} P_s$ that are incomparable with $x_s$.
\end{enumerate}

\medskip

In both cases, define $<_Q$ to be $<_{P_{s+1}}$ on $P_{s+1}$.
\end{itemize}
This construction can be carried out in $\rca$.
\end{Construction}

The following two lemmas encapsulate the important properties of Construction~\ref{const-Xi}.

\begin{Lemma}[\cite{FrittaionHendtlassMarconeShaferVanderMeeren}*{Lemma~4.3}]\label{lem-QProps}
The following is provable in $\rca$.
Let $f \colon \Nb \imp \Nb$ be an injection, let $P$ be a finite partial order, let $x \in P$, and let $(Q, <_Q) = \Xi_f (P,x)$.  Consider $m, n \in \Nb$ with $n < m$.
\begin{enumerate}[(1)]
\item\label{it-QPropTrue} If $n \in T_m$, then $P_m <_Q x_n$ and $\forall y \in P_n \; (x_n \mid_Q y \;\imp\; P_m \mid_Q y)$.

\medskip

\item\label{it-QPropFalse} If $n \notin T_m$, then $x_n <_Q P_m$.
\end{enumerate}
\end{Lemma}

\begin{Lemma}[\cite{FrittaionHendtlassMarconeShaferVanderMeeren}*{Lemma~4.4}]\label{lem-WPOACAreversal}
The following is provable in $\rca$.
Let $f \colon \Nb \imp \Nb$ be an injection, let $P$ be a finite partial order, let $x \in P$, and let $(Q, <_Q) = \Xi_f (P,x)$.  If $Q$ is not a well-partial-order, then the range of $f$ exists as a set.
\end{Lemma}

We now give the reversal for Theorem~\ref{thm-ExtendedRSpo}.

\begin{Lemma}\label{lem-RSpoWPORev}
The statement ``Every infinite well-partial-order has a $(0,\infty)$-homogeneous chain'' implies $\aca$ over $\rca$.
\end{Lemma}

\begin{proof}
Let $f \colon \Nb \imp \Nb$ be an injection.  Let $(P, <_P)$ be the $2$-element partial order $P = \{x, y\}$ with $x \mid_P y$.  Let $(Q, <_Q) = \Xi_f(P, x)$, and recall that $Q = \bigcup_{n \in \Nb}P_n$, where $P_n = \{n\} \times P$ for each $n$.  Let $x_n = \la n, x \ra$ and $y_n = \la n, y \ra$ denote the copies of $x$ and $y$ in $P_n$ for each $n$.

First, assume that there is an infinite set $C \subseteq Q$ consisting only of elements $y_n$ for false numbers $n$, and furthermore assume that there are no false numbers $n$ and $k$ with $y_n \in C$ and $y_n <_Q x_k$.  Then the number $k$ is true whenever there is an $n$ such that $y_n \in C$ and $y_n <_Q x_k$.  Moreover, if $k$ is a true number, then there is an $m > k$ with $y_m \in C$ because $C$ is infinite, and we have that $y_m <_Q x_k$ by Lemma~\ref{lem-QProps} item~\ref{it-QPropTrue}.  Thus the true numbers $k$ are $\Sigma^0_1$-definable by the formula $\exists m \, (y_m \in C \;\andd\; y_m <_Q x_k)$.  The original definition of the true numbers from Definition~\ref{def-true} is $\Pi^0_1$, so the set of true numbers exists by $\Delta^0_1$ comprehension.  Thus the range of $f$ exists as a set.

Second, assume instead that whenever $C \subseteq Q$ is an infinite set consisting only of elements $y_n$ for false numbers $n$, there are false numbers $n$ and $k$ with $y_n \in C$ and $y_n <_Q x_k$.  The goal of this case is to show that $Q$ has no $(0,\infty)$-homogeneous chain.  The hypothesis ``Every infinite well-partial-order has a $(0,\infty)$-homogeneous chain'' then implies that $Q$ is not a well-partial-order, so the range of $f$ exists as a set by Lemma~\ref{lem-WPOACAreversal}.

Consider an infinite chain $C \subseteq Q$.  We show that $C$ is not $(0,\infty)$-homogeneous.  Observe that if $n$ is a true number, then $n \in T_m$ for all $m > n$ and therefore that $y_n \mid_P P_m$ for all $m > n$ by Lemma~\ref{lem-QProps} item~\ref{it-QPropTrue}.    There are now three sub-cases.

\begin{itemize}
\item  There is a number $n$ with $x_n \in C$.  Let $m$ be a true number with $m > n$.  Then $y_m \comp_Q x_n$ by Lemma~\ref{lem-QProps}, but $y_m$ is only comparable with finitely many elements of $Q$ because $m$ is true.  So $y_m$ witnesses that $C$ is not $(0,\infty)$-homogeneous.

\medskip

\item There is a true number $n$ with $y_n \in C$.  Then $y_n$ is comparable with only finitely many elements of $Q$, so $C$ cannot be an infinite chain.

\medskip

\item The chain $C$ consists only of elements $y_n$ for false numbers $n$.  Then by the case assumption there are false numbers $n$ and $k$ with $y_n \in C$ and $y_n <_Q x_k$.  Let $m > k$ be such that $m$ is true and that $k \notin T_m$.  Then $y_n <_Q x_k <_Q y_m$ by Lemma~\ref{lem-QProps} item~\ref{it-QPropFalse}, but $y_m$ is only comparable with finitely many elements of $Q$ because $m$ is true.  So $y_m$ witnesses that $C$ is not $(0,\infty)$-homogeneous.
\end{itemize}

We have shown that the range of $f$ exists as a set in both of the main cases.  We may therefore conclude that $\aca$ holds by Lemma~\ref{lem-ACAinjection}.
\end{proof}

\begin{Theorem}\label{thm-ExtendedRSpoEquiv}
The following are equivalent over $\rca$.
\begin{enumerate}[(1)]
\item\label{it-RSpoExtACA} $\aca$.

\medskip

\item\label{it-RSpoExt} Every infinite partial order with no infinite antichain has a $(0,\infty)$-homogeneous chain.

\medskip

\item\label{it-RSpoExtWPO} Every infinite well-partial-order has a $(0,\infty)$-homogeneous chain.
\end{enumerate}
\end{Theorem}

\begin{proof}
We have that \ref{it-RSpoExtACA}~$\Imp$~\ref{it-RSpoExt} by Theorem~\ref{thm-ExtendedRSpo}, that \ref{it-RSpoExt}~$\Imp$~\ref{it-RSpoExtWPO} because well-partial-orders do not have infinite antichains, and that \ref{it-RSpoExtWPO}~$\Imp$~\ref{it-RSpoExtACA} by Lemma~\ref{lem-RSpoWPORev}.
\end{proof}

It follows that the statement ``Every infinite partial order with no infinite antichain has a $(0,\cof)$-homogeneous chain'' implies $\aca$ over $\rca$.  As with $\cofrspo_{<\infty}$, we know no better upper bound for the statement than $\pica$, yet the statement cannot be equivalent to $\pica$ over $\rca$ because it is $\Pi^1_2$.

\begin{Question}
What is the strength of the statement ``Every infinite partial order with no infinite antichain has a $(0,\cof)$-homogeneous chain'' relative to $\rca$?
\end{Question}

\section{Maximal chain principles and comments on the original Rival--Sands proof}\label{sec-MaxChains}

The original proof of $\rspo_{<\infty}$ given by Rival and Sands in~\cite{RivalSands} relies on the following maximality principle, which we name $\mmlc$ for the \emph{maximal max-less chain principle}.

\begin{Definition}
Call a chain $C$ in a partial order $(P, <_P)$ \emph{max-less} if $C$ has no maximum element:  $\forall x \in C \; \exists y \in C \; (x <_P y)$.  The \emph{maximal max-less chain principle} ($\mmlc$) is the following statement.  For every partial order $(P, <_P)$, there is a max-less chain that is $\subseteq$-maximal among the max-less chains of $P$.  That is, there is a max-less chain $C \subseteq P$ for which $C \subseteq D \subseteq P$ implies $C = D$ for all max-less chains $D$ of $P$.  Call such a $C$ a \emph{maximal max-less chain} in $P$.
\end{Definition}

Similarly, call a chain $C$ in a partial order $(P, <_P)$ \emph{min-less} if $C$ has no minimum element:  $\forall x \in C \; \exists y \in C \; (y <_P x)$.  By reversing the partial order, we see that, over $\rca$, $\mmlc$ is equivalent to the statement ``For every partial order $(P, <_P)$, there is a min-less chain that is $\subseteq$-maximal among the min-less chains of $P$.''

We very roughly sketch the Rival and Sands proof, using some of the terminology we have introduced here.  Let $(P, <_P)$ be a countably infinite partial order of finite width $k$ for some $k$.  Suppose for a contradiction that $(P, <_P)$ contains no $(0,\infty)$-homogeneous chain.  The partial order $P$ contains either an infinite ascending sequence or an infinite descending sequence, and we may assume that $P$ contains an infinite ascending sequence by reversing the order if necessary.  Define a sequence $(S_i, C_i, D_i)_{i \leq k+1}$ of triples of subsets of $P$ as follows.  First, let $S_0 = P$, by $\mmlc$ let $C_0$ be a maximal max-less chain in $P$, and let $D_0$ be a cofinal infinite ascending sequence in $C_0$.  Given $(S_i, C_i, D_i)$, let $S_{i+1}$ consist of the elements of $P$ that are counterexamples to the $(0,\infty)$-homogeneity of $D_i$ in the sense of Definition~\ref{def-counterexample}.  Again by $\mmlc$, let $C_{i+1}$ be a maximal max-less chain in $S_{i+1}$, and let $D_{i+1}$ be a cofinal infinite ascending sequence in $C_{i+1}$.  Using the maximality of the $C_j$'s, show that $D_i \subseteq S_j$ whenever $j \leq i \leq k+1$.  This property allows us to choose an antichain $\{d_0, \dots, d_k\}$ with $d_i \in D_{k+1-i}$ for each $i \leq k$, which contradicts that $P$ has width $k$.  The $d_i$'s are chosen back-to-front, with $d_0$ chosen from $D_{k+1}$, then $d_1$ chosen from $D_k$, and so on.  Thus this method does not suffice to prove the extensions of Theorems~\ref{thm-CofExtendedRSpo} and~\ref{thm-ExtendedRSpo} to partial orders without infinite antichains but not necessarily of finite width.

We analyze the strength of $\mmlc$ and of a few other statements concerning maximal chains in partial orders.  In particular, we show that $\mmlc$ is equivalent to $\pica$ over $\rca$.  Therefore, the original proof by Rival and Sands requires $\pica$ in the sense that it relies on a principle that is equivalent to $\pica$.  Furthermore, if a standard bound on the width $k$ of the partial orders being considered is not fixed in advance, then the original Rival and Sands proof is not a proof in $\pica$ because it iterates applying $\mmlc$ a finite-but-arbitrarily-large number of times.

It is well-known and easy to show that $\rca$ suffices to produce a maximal chain and a maximal antichain in a given partial order $(P, <_P)$.  Enumerate the elements of $P$ as $\la p_n : n \in \Nb \ra$, and add element $p_n$ to the chain (antichain) if and only if it is comparable (incomparable) to all the elements previously added to the chain (antichain).  However, if we want to extend a given chain $C$ to a maximal chain $D$, then $\aca$ is required.

\begin{Proposition}\label{prop-ExtendChainMax}
The following are equivalent over $\rca$.
\begin{enumerate}[(1)]
\item $\aca$.

\medskip

\item\label{it-ExtendChainMax} For every partial order $(P, <_P)$ and chain $C \subseteq P$, there is a maximal chain $D$ with $C \subseteq D$.
\end{enumerate}
\end{Proposition}

\begin{proof}
For the forward direction, let $(P, <_P)$ be a partial order, and let $C \subseteq P$ be a chain.  Let $Q = \{q \in P : \forall c \in C \; (q \comp_P c)\}$ be the set of elements of $P$ that are comparable with all elements of $C$, and consider the partial order $(Q, <_P)$.  Notice that $C \subseteq Q$ because $C$ is a chain.  Let $D$ be a maximal chain in $Q$, which can be produced in $\rca$ as described above.  As $D \subseteq Q$, every element of $D$ is comparable with every element of $C$.  Therefore $C \subseteq D$ by the maximality of $D$ in $Q$.  We claim that $D$ is also a maximal chain in $P$.  Suppose for a contradiction that it is not.  Then there is a $p \in P \setminus D$ that is comparable with every element of $D$.  Then $p$ is also comparable with every element of $C$, so $p \in Q$.  Thus there is a $p \in Q \setminus D$ that is comparable with every element of $D$.  This contradicts the maximality of $D$ in $Q$.  Therefore $D$ is a maximal chain in $P$ with $C \subseteq D$.

For the reversal, let $f \colon \Nb \imp \Nb$ be an injection.  We define a partial order $(P, <_P)$ and a chain $C \subseteq P$ in such a way that the range of $f$ can be extracted from $f$ and any maximal chain $D \supseteq C$.

Let $(L, <_L)$ be the linear order defined in Construction~\ref{const-TFstages} for $f$, where $L = \{\ell_n : n \in \Nb\}$.  Let $P = \{c_n, \ell_n : n \in \Nb\}$, let $c_n <_P c_m$ if and only if $n < m$, and let $\ell_n <_P \ell_m$ if and only if $\ell_n <_L \ell_m$.  Moreover, for each $n \leq m$, define
\begin{enumerate}[(1)]
\item\label{constr1} $c_m <_P \ell_n$ if $\forall k \, (n < k \leq m \;\imp\; f(n) < f(k))$ (i.e., if $n$ is true at stage $m$), and $c_m \mid_P \ell_n$ otherwise;

\medskip

\item\label{constr2} $c_n <_P \ell_m$.
\end{enumerate}
Notice that $c_n <_P \ell_n$ for every $n$.

We must verify that $(P, <_P)$ is a partial order.  Clearly $<_P$ is irreflexive, so we check that $<_P$ is transitive.  For no $a, b \in \Nb$ it is the case that $\ell_a <_P c_b$, so we need only verify the following four cases.
\begin{itemize}
\item If $\ell_a <_P \ell_b$ and $\ell_b <_P \ell_d$ for some $a, b, d \in \Nb$, then $\ell_a <_P \ell_d$ because the restriction of $<_P$ to $L$ is a linear order.

\medskip

\item Similarly, if $c_a <_P c_b$ and $c_b <_P c_d$ for some $a, b, d \in \Nb$, then $c_a <_P c_d$ because the restriction of $<_P$ to $\{c_n : n \in \Nb\}$ is a linear order.

\medskip

\item Suppose that $c_s <_P c_m$ and $c_m <_P \ell_n$ for some $s, m, n \in \Nb$ with $s \neq m$.  Notice that $c_s <_P c_m$ implies $s < m$.  If $s \leq n$, then from condition~\ref{constr2} we immediately obtain $c_s <_P \ell_n$.  If instead $n < s$, then $n < s < m$.  That $c_m <_P \ell_n$ means that $n$ is true at stage $m$ by condition~\ref{constr1}.  Thus $n$ is also true at stage $s$, so $c_s <_P \ell_n$ as well.

\medskip

\item Finally, suppose that $c_s <_P \ell_m$ and $\ell_m <_P \ell_n$ for some $s, m, n \in \Nb$ with $m \neq n$.  If $s \leq n$, then it follows immediately from condition~\ref{constr2} that $c_s <_P \ell_n$.  So suppose that $n < s$. We claim that $n < m$.  Suppose on the contrary that $m < n$.  Then $m$ is false at stage $n$ by Construction~\ref{const-TFstages} because $\ell_m <_P \ell_n$.  Thus $m$ is false at stage $s$ as well because $m < n < s$, and this contradicts $c_s <_P \ell_m$.  We therefore have that $n < s$, that $n < m$, and that $n$ is true at stage $m$ because $\ell_m <_P \ell_n$.  If $s \leq m$, then $n < s \leq m$, so $n$ is true at stage $s$ as well, and we have $c_s <_P \ell_n$ as desired.  Otherwise $m < s$, so $n < m < s$ and $m$ is true at stage $s$ because $c_s <_P \ell_m$.  Suppose for a contradiction that $n$ is false at stage $s$.  Then there is a $k$ with $n < k \leq s$ and $f(k) < f(n)$.  It must be that $k > m$ because otherwise $k$ would witness that $n$ is false at stage $m$.  It must therefore also be that $f(m) < f(k)$ because otherwise $k$ would witness that $m$ is false at stage $s$.  Thus $f(m) < f(k) < f(n)$, so in fact $m$ witnesses that $n$ is false at stage $m$, which is a contradiction.  Thus $n$ is true at stage $s$, so $c_s <_P \ell_n$.
\end{itemize}

We conclude that $(P, <_P)$ is indeed a partial order.  Let $C$ be the chain $C = \{c_n : n \in \Nb\}$, and let $D$ be a maximal chain in $P$ with $C \subseteq D$.  Then $\{n : \ell_n \in D\}$ is the set of true numbers for $f$.  To see this, first observe that if $n$ is true, then $\ell_n$ is comparable with every element of $P$, so $\ell_n$ must be in $D$ by maximality.  On the other hand, if $n$ is false, then there is a $m > n$ with $f(m) < f(n)$ witnessing that $n$ is false.  Then $c_m \mid_P \ell_n$, so $\ell_n$ cannot be in $D$.  The true numbers for $f$ thus form a set, and therefore the range of $f$ exists as a set.  This implies $\aca$ by Lemma~\ref{lem-ACAinjection}.
\end{proof}

The partial order constructed in the reverse direction of Proposition~\ref{prop-ExtendChainMax} is $2$-chain decomposable.  Thus Proposition~\ref{prop-ExtendChainMax} item~\ref{it-ExtendChainMax} remains equivalent to $\aca$ when restricted to $2$-chain decomposable partial orders or to partial orders of width $\leq\! 2$.  Extending a given antichain to a maximal antichain in a partial order is also equivalent to $\aca$ over $\rca$ by~\cite{FrittaionMarcone}*{Lemma~5.5}.

Finally, we turn to $\mmlc$ and show that it is equivalent to $\pica$ over $\rca$, even when restricted to linear orders.  Consider a partial order $(P, <_P)$.  The axiomatic difficulty in producing a maximal max-less chain in $P$ is in determining whether a given $p \in P$ can be a member of a max-less chain.  It is easy to see that $p$ is a member of a max-less chain if and only if $p\ua$ is reverse ill-founded.  However, this is a $\Sigma^1_1$ property of $p$, and we show that producing the set of such $p$ requires $\pica$ in general.  For the reversal, recall the \emph{Kleene--Brouwer ordering} of $\Nb^{<\Nb}$, whereby $\tau \ltKB \sigma$ if either $\tau$ is a proper extension of $\sigma$ or $\tau$ is to the left of $\sigma$.  That is, $\tau \ltKB \sigma$ if and only if
\begin{align*}
\tau \succ \sigma \quad\orr\quad \exists n < \min(|\sigma|, |\tau|) \; \bigl(\tau(n) < \sigma(n) \;\andd\; \forall i < n \; (\sigma(i) = \tau(i))\bigr).
\end{align*}

\begin{Theorem}\label{thm-MaxLess}
The following are equivalent over $\rca$.
\begin{enumerate}[(1)]
\item\label{it-PicaML} $\pica$.

\medskip

\item\label{it-MMLC} $\mmlc$.

\medskip

\item\label{it-MMLCLO} $\mmlc$ restricted to linear orders.

\medskip

\item\label{it-FindWFPO} For every partial order $(P, <_P)$, there is a set $W \subseteq P$ such that $\forall p \in P \; (p \in W \;\biimp\; \textup{$p\da$ is well-founded})$.

\medskip

\item\label{it-FindWFLO} For every linear order $(L, <_P)$, there is a set $W \subseteq L$ such that $\forall \ell \in L \; (\ell \in W \;\biimp\; \textup{$\ell\da$ is well-founded})$.
\end{enumerate}	
\end{Theorem}

\begin{proof}
For~\ref{it-PicaML}~$\Imp$~\ref{it-MMLC}, let $(P, <_P)$ be a partial order.  Use $\pica$ (and the fact that the $\Sigma^1_1$ sets are the complements of the $\Pi^1_1$ sets) to define the set $Q = \{q \in P : \text{$q\ua$ is reverse ill-founded}\}$, and consider the partial order $(Q, <_P)$.  Let $C$ be a maximal chain in $Q$, which may be produced in $\rca$.

We first show that $C$ is max-less.  To see this, suppose for a contradiction that $C$ has a maximum element $m$.  Then $m \in C \subseteq Q$, so $m\ua$ is reverse ill-founded in $P$.  Thus there is an infinite ascending sequence $m <_P a_0 <_P a_1 <_P a_2 <_P \cdots$ in $P$.  Each $a_i\ua$ is reverse ill-founded in $P$ as well, so $\{a_i : i \in \Nb\} \subseteq Q$.  Then $C \cup \{a_i : i \in \Nb\} \subseteq Q$ is a chain in $Q$ properly extending $C$, contradicting that $C$ is a maximal chain in $Q$.  Thus $C$ is max-less.

We now show that $C$ is maximal among the max-less chains of $P$.  Suppose that $D \subseteq P$ is a max-less chain with $C \subseteq D$.  Let $d \in D$.  As $D$ is max-less, we may define an infinite ascending sequence $d = d_0 <_P d_1 <_P d_2 <_P  \cdots$ of elements of $D$.  Thus $d\ua$ is reverse ill-founded in $P$, so $d \in Q$.  This shows that $D \subseteq Q$.  That is, $C$ and $D$ are chains in $Q$ with $C \subseteq D$.  Therefore $C = D$ by the maximality of $C$ in $Q$.  Thus $C$ is a maximal max-less chain in $P$.

We have that~\ref{it-PicaML}~$\Imp$~\ref{it-FindWFPO} because the set $W$ required by item~\ref{it-FindWFPO} is $\Pi^1_1$.  Furthermore, \ref{it-MMLC}~$\Imp$~\ref{it-MMLCLO} because item~\ref{it-MMLCLO} is a special case of item~\ref{it-MMLC}, and likewise \ref{it-FindWFPO}~$\Imp$~\ref{it-FindWFLO} because item~\ref{it-FindWFLO} is a special case of item~\ref{it-FindWFPO}.

For~\ref{it-MMLCLO}~$\Imp$~\ref{it-FindWFLO}, let $(L, <_L)$ be a linear order, and by item~\ref{it-MMLCLO} applied to the reverse of $L$, let $C$ be a maximal min-less chain in $L$.  Then $C$ consists of exactly the $\ell \in L$ for which $\ell\da$ is ill-founded.  If $\ell \in C$, then there is an infinite descending sequence in $C$ below $\ell$ because $C$ is min-less.  Thus $\ell\da$ is ill-founded.  Conversely, consider an $\ell \in L$ with $\ell\da$ ill-founded.  Then there is an infinite descending sequence $\ell = d_0 >_L d_1 >_L d_2 >_L \cdots$ in $L$ below $\ell$.  Thin the sequence so that its range exists as a set $D$ with $\ell \in D$.  Then $C$ and $D$ are both min-less, so $C \cup D$ is min-less as well, plus $C \cup D$ is a chain because every subset of a linear order is a chain.  Therefore $C = C \cup D$ by the maximality of $C$, so $\ell \in C$.  Thus $W = L \setminus C$ consists of exactly the $\ell \in L$ for which $\ell\da$ is well-founded.

To finish the proof, it suffices to show that \ref{it-FindWFLO}~$\Imp$~\ref{it-PicaML}.  To do this, we show that item~\ref{it-FindWFLO} implies the leftmost path principle $\lpp$, which is equivalent to $\pica$ over $\rca$ by~Theorem~\ref{thm-LPP}.  Let $T \subseteq \Nb^{<\Nb}$ be an ill-founded tree, and apply item~\ref{it-FindWFLO} to the linear order $(T, \ltKB)$ (and take complements) to obtain the set $I$ of $\sigma \in T$ such that $\sigma\da$ is ill-founded with respect to $\ltKB$.  Notice that $I$ is a subtree of $T$ and that $I$ has no $\ltKB$-minimum element.

Recursively define the following \emph{a priori} possibly partial function $f \colon \Nb \imp \Nb$ such that for every $n$, if $f \rst n$ is defined, then $f \rst n \in I$.  Given $n$, if $f \rst n$ is defined, let $m$ be least such that $(f \rst n)^\smf m \in I$ (if there is such an $m$) and set $f(n) = m$.  We show that $f$ is total.  Suppose for a contradiction that $f$ is partial, and by the $\Pi^0_1$ least element principle, let $n$ be least such that $f(n)$ is undefined.  Then $f \rst n$ is defined and in $I$.  As $I$ has no $\ltKB$-least element, let $\tau \in I$ be such that $\tau \ltKB f \rst n$.  If $\tau$ is to the left of $f \rst n$, then let $j < n$ be such that $\tau(j) < f(j)$ and $\tau \rst j = f \rst j$.  Then $\tau \rst (j+1) = (f \rst j)^\smf \tau(j)$ is in $I$ and $\tau(j) < f(j)$, which contradicts the definition of $f(j)$.  On the other hand, if $\tau \succ f \rst n$, then $|\tau| > n$ and $\tau \rst (n+1) = (f \rst n)^\smf \tau(n)$ is in $I$.  Thus there is an $m$ such that $(f \rst n)^\smf m$ is in $I$, and therefore there is a least such $m$.  Thus $f(n)$ is defined, which is a contradiction.  Therefore $f$ is total.  Thus $f$ is an infinite path through $I$ and hence is an infinite path through $T$.  We show that $f$ is the leftmost infinite path through $T$.

Consider an infinite path $g$ through $T$.  Then $g \rst 0 \gtKB g \rst 1 \gtKB g \rst 2 \gtKB \cdots$ is an infinite $\ltKB$-descending sequence.  Thus $(g \rst n)\da$ is ill-founded for every $n$, so $g \rst n \in I$ for every $n$.  Now suppose for a contradiction that $g$ is to the left of $f$.  Then there is an $n$ such that $g(n) < f(n)$ and $g \rst n = f \rst n$.  Thus $g \rst (n+1) = (f \rst n)^\smf g(n)$ is in $I$ and $g(n) < f(n)$, which contradicts the definition of $f(n)$.  Therefore $g$ cannot be to the left of $f$, so $f$ is the leftmost path through $T$.  This concludes the proof of $\lpp$.
\end{proof}

By taking complements and/or reversing the order, the statement ``$p\da$ is well-founded'' may be replaced by any of ``$p\da$ is ill-founded,'' ``$p\ua$ is reverse well-founded,'' and ``$p\ua$ is reverse ill-founded'' in Theorem~\ref{thm-MaxLess} item~\ref{it-FindWFPO}, and the analogous replacement may be made in item~\ref{it-FindWFLO} as well.

\section*{Acknowledgments}
We thank Keita Yokoyama for helpful discussions, especially concerning the formalization of Kierstead's theorem in $\rca$.  We thank our anonymous reviewer for several helpful suggestions.  For their generous support, we thank the \emph{Workshop on Ramsey Theory and Computability} at the University of Notre Dame Rome Global Gateway; Dagstuhl Seminar 18361:  \emph{Measuring the Complexity of Computational Content:  From Combinatorial Problems to Analysis}; the BIRS/CMO workshop 19w5111: \emph{Reverse Mathematics of Combinatorial Principles}; and the University of Leeds School of Mathematics Research Visitors' Centre.  Fiori-Carones was supported by the Mathematical Center in Akademgorodok under the agreement No. 075-15-2022-281 with the Ministry of Science and Higher Education of the Russian Federation.  Marcone was supported by the departmental PRID funding \emph{HiWei --- The higher levels of the Weihrauch hierarchy} and by the Italian PRIN 2017 grant \emph{Mathematical Logic:  models, sets, computability}.  Shafer was supported by the John Templeton Foundation grant ID 60842 \emph{A new dawn of intuitionism: mathematical and philosophical advances} and by EPSRC grant EP/T031476/1 \emph{Reverse mathematics of general topology}.  The opinions expressed in this work are those of the authors and do not necessarily reflect the views of the John Templeton Foundation.

\bibliographystyle{amsplain}
\bibliography{FioriCaronesMarconeShaferSoldaRSp}

\vfill

\end{document}